\documentclass[11pt]{amsart}
\usepackage{amsmath, amssymb, latexsym, graphicx, mathrsfs, enumerate, amsmath, amsthm, textcomp, url, tikz-cd,bbm}
\usepackage[T1]{fontenc}
\usepackage{mathdots}
\usepackage{comment}
\usepackage[toc,page]{appendix}
\usepackage{hyperref}
\usepackage{tikz-cd}
\hypersetup{
    colorlinks=true,
    linkcolor=blue,
    filecolor=magenta,      
    urlcolor=cyan,
}

\usepackage{color}

\input xy
\xyoption{all}

\allowdisplaybreaks
\allowdisplaybreaks

\numberwithin{equation}{section}
\newtheorem{theorem}{Theorem}[section]
\newtheorem{proposition}[theorem]{Proposition}
\newtheorem{corollary}[theorem]{Corollary}
\newtheorem{lemma}[theorem]{Lemma}

\newtheorem{definition}[theorem]{Definition}

\newtheorem*{remark*}{Remark}
\newtheorem{strategy}[theorem]{Strategy}

\newtheorem{example}[theorem]{Example}
\newtheorem{remark}[theorem]{Remark}

\newcommand{\mfo}{\mathfrak{o}}

\newcommand{\mff}{\mathfrak{f}}

\newcommand{\Mfo}{\mathcal{O}}

\newcommand{\cl}{\mathrm{Cl}}
\newcommand{\clb}{\overline{\mathrm{Cl}}}

\newcommand{\Z}{\mathbb{Z}}

\newcommand{\ord}{\mathrm{ord}}
\newcommand*{\Rom}[1]{\expandafter\@slowromancap\romannumeral #1@}
\begin{document}

\title[Ideal class monoids of cubic orders]
{Ideal class monoids of cubic orders}


\author[Sungmun Cho]{Sungmun Cho}
\author[Jungtaek Hong]{Jungtaek Hong}
\author[Yuchan Lee]{Yuchan Lee}
\thanks{The authors were supported by   Basic Science Research Institute Fund whose NRF grant number is RS-2021-NR060139 and the National Research Foundation of Korea(NRF) grant
funded by the Korea government(MSIT) (No. RS-2026-25508638).
}

\address{Sungmun Cho \\  Department of Mathematics, POSTECH, 77, Cheongam-ro, Nam-gu, Pohang-si, Gyeongsangbuk-do, 37673, KOREA}

\email{sungmuncho12@gmail.com}

\address{Jungtaek Hong \\ Department of Mathematics, POSTECH, 77, Cheongam-ro, Nam-gu, Pohang-si, Gyeongsangbuk-do, 37673, KOREA}

\email{jungtaekhong123@gmail.com}

\address{Yuchan Lee \\ Department of Mathematics, POSTECH, 77, Cheongam-ro, Nam-gu, Pohang-si, Gyeongsangbuk-do, 37673, KOREA}

\email{yuchanlee329@gmail.com}
\begin{abstract}
Let $R$ be an order in a number field, let $\overline{\mathrm{Cl}}(R)$ be its ideal class monoid, and let $\mathrm{Cl}(R)$ act on it by multiplication.  The local-global product formula identifies the orbit set $\mathrm{Cl}(R)\backslash\overline{\mathrm{Cl}}(R)$ with a product of local orbit sets; in this sense, it is the genus set of fractional $R$-ideals.  For a Gorenstein order $R$ in a cubic extension of number fields, we give a closed Euler product formula for the cardinality of this genus set.  The local factors come from an explicit classification of local cubic overorders: for arbitrary local cubic orders, we parametrize all overorders, determine their inclusion relations, and identify the Gorenstein ones.  As an application to Bhargava's parametrization of $2\times3\times3$ cubes, our formula gives the exact number of $\mathrm{Cl}(R)$-equivalence classes of integral $\mathrm{GL}_2(\mathbb Z)\times\mathrm{SL}_3(\mathbb Z)\times\mathrm{SL}_3(\mathbb Z)$-orbits whose associated cubic ring is the prescribed Gorenstein order $R$.
\end{abstract}

\maketitle

\setcounter{tocdepth}{1}
\tableofcontents

\section{Introduction}

Let \(R\) be an order in a number field \(E\).  The ideal class monoid of
\(R\), denoted \(\overline{\mathrm{Cl}}(R)\), is the monoid of equivalence
classes of fractional \(R\)-ideals, where \(I\sim aI\) for \(a\in E^\times\).
Its group of units is the ideal class group \(\mathrm{Cl}(R)\), so
\(\overline{\mathrm{Cl}}(R)\) refines \(\mathrm{Cl}(R)\) by retaining, in
addition to the invertible classes, the classes represented by non-invertible
fractional ideals.  The natural multiplication action of \(\mathrm{Cl}(R)\) on
\(\overline{\mathrm{Cl}}(R)\) gives the orbit set
\[
        \mathrm{Cl}(R)\backslash \overline{\mathrm{Cl}}(R),
\]
which is the main global object studied in this paper.

The product formula recalled in Proposition \ref{prop:global_local}, based on
\cite[Proposition 5.3 and Corollary 5.5.(2)]{CHL}, identifies this orbit set with
a product of local orbit sets.  We use the word \emph{genus} in this product
formula sense: two global classes lie in the same genus precisely when they
determine the same local orbit data.  Thus computing the quotient reduces to
the local orbit sets of fractional $R_v$-ideals at the finitely many primes where
$R_v$ is not maximal.  In the cubic case, the local count is then reduced to
enumerating the overorders of $R_v$, i.e. the intermediate orders between
$R_v$ and its maximal order, and identifying which of them are Gorenstein.

The contrast with the Bass case is already visible at the level of local
overorders.  For a local Bass order, the overorders form a chain.  For example,
if $E/F$ is a totally ramified quadratic extension of non-Archimedean local
fields and $R=\mathfrak{o}[\pi^{10}\pi_E]$, then the overorders are exactly
\[
\begin{tikzcd}[column sep=3em]
    \mathbf{R} \arrow[r, hook] &
    \mathbf{\mathfrak{o}[\pi^9\pi_E]} \arrow[r, hook] &
    \cdots \arrow[r, hook] &
    \mathbf{\mathfrak{o}[\pi\pi_E]} \arrow[r, hook] &
    \mathbf{\mathcal O_E}=\mathfrak{o}[\pi_E].
\end{tikzcd}
\]
Thus the local enumeration is a one-variable problem.  In contrast, cubic
overorders form a branched poset governed by several valuation parameters.  For
example, if $E/F$ is a totally ramified cubic extension with uniformizer
$\pi_E$, then the order $R=\mathfrak{o}[\pi^3\pi_E^2]$ has the same Serre
invariant, namely the length
$S(R)=\ell_{\mathfrak{o}}(\mathcal O_E/R)=10$, as the Bass example above, but
its overorder diagram in Example~\ref{ex:simpleram} already contains several
incomparable branches.  Sections \ref{sec:irred}--\ref{sec:split} make this
branching explicit.

The main global result is a closed Euler product for
\[
        \#\bigl(\mathrm{Cl}(R)\backslash \overline{\mathrm{Cl}}(R)\bigr)
\]
when $R$ is a Gorenstein cubic order.  The local results are more general: for
arbitrary local cubic orders we give explicit parametrizations of all overorders
and of the Gorenstein overorders among them.

\subsection{Main result}

To state the main theorem, let \(R\) be a Gorenstein order in a cubic extension
\(E/F\) of number fields, and let \(\mathfrak o\) be the ring of integers of
\(F\).  For a finite prime \(v\) of \(\mathfrak o\), write
\(R_v=R\otimes_{\mathfrak o}\mathfrak o_v\) and
\(E_v=E\otimes_F F_v\).  The local enumeration is organized by two numerical
invariants of \(R_v\).  The first is the Serre invariant
\[
        S(R_v)=\ell_{\mathfrak o_v}(\mathcal O_{E_v}/R_v),
\]
the length of the quotient of the maximal order by \(R_v\).  The second is the
conductor of \(R_v\) in \(\mathcal O_{E_v}\).  Globally, write the conductor of
\(R\) as
\[
        \mathfrak f(R)=
        \prod_{\mathfrak P\subset \mathcal O_E}
        \mathfrak P^{f_{\mathfrak P}} .
\]
For Gorenstein cubic orders, Proposition \ref{prop:gorenstein_cond} gives
linear relations between the Serre invariant and the local conductor exponents.
These relations reduce the local finite enumeration formulas to the rational
local factors appearing in the theorem below.

Let \(S_R\) be the finite set of nonzero primes \(v\) of \(\mathfrak o\)
dividing \(\mathfrak f(R)\).  For \(v\in S_R\), let
\[
        \mathbf f(v)=(f_{\mathfrak P})_{\mathfrak P\mid v}
\]
be the tuple of local conductor exponents, and let
\[
        \tau(v)\in\{\textit{unr},\ \textit{tr},\ (1\,2),\ (1\,1^2),\ (1\,1\,1)\}
\]
be the splitting type of the cubic \(F_v\)-algebra \(E\otimes_F F_v\), as used
in the local classification; see Sections \ref{sec:irred}--\ref{sec:split} and
Equation \eqref{eq:splittingtype}.

\begin{theorem}[Global Gorenstein product formula; proved as Theorem \ref{thm:globalGorensteinProduct}]
\label{thm:intro_globalGorensteinProduct}
With the notation above,
\begin{multline*}
\#(\mathrm{Cl}(R)\backslash\clb(R))
=
\prod_{\substack{v\in S_R\\ \tau(v)=\textit{unr}}}
G_{\textit{unr}}(q_v;\mathbf f(v))
\cdot
\prod_{\substack{v\in S_R\\ \tau(v)=\textit{tr}}}
G_{\textit{tr}}(q_v;\mathbf f(v))
\cdot
\prod_{\substack{v\in S_R\\ \tau(v)=(1\,2)}}
G_{(1\,2)}(q_v;\mathbf f(v))\\
\cdot
\prod_{\substack{v\in S_R\\ \tau(v)=(1\,1^2)}}
G_{(1\,1^2)}(q_v;\mathbf f(v))
\cdot
\prod_{\substack{v\in S_R\\ \tau(v)=(1\,1\,1)}}
G_{(1\,1\,1)}(q_v;\mathbf f(v)).
\end{multline*}
Here $q_v$ is the cardinality of the residue field of $\mfo$ at $v$, and for each admissible conductor tuple $\mathbf f(v)$ the local factor $G_\tau(q_v;\mathbf f(v))$ is an explicit polynomial in $q_v$; compact formulae for these polynomials are listed in Definition \ref{def:Gtau}.
\end{theorem}

The formula is an exact genus count for the fractional ideals of the fixed
Gorenstein cubic order \(R\).  Each local factor is obtained from the local
ideal class monoid count of \(R_v\), computed through the overorder structure of
\(R_v\).

\subsection{Key ingredients: local-global reduction and explicit local enumeration}
The proof is purely algebraic.  It has two main inputs.

\subsubsection{Local-global reduction and the counting formula}
First, the local-global product formula recalled in Proposition
\ref{prop:global_local}, based on \cite[Proposition 5.3 and Corollary
5.5.(2)]{CHL}, identifies
\(\mathrm{Cl}(R)\backslash\overline{\mathrm{Cl}}(R)\) with the product of the
local ideal class monoids \(\overline{\mathrm{Cl}}(R_v)\).  Thus the global
genus count reduces to computing the finite local numbers
\(\#\overline{\mathrm{Cl}}(R_v)\) at the primes where \(R_v\) is not maximal.
The local count is converted into overorder counting by Proposition
\ref{prop:orderidealcounting}:
\[
        \#\overline{\mathrm{Cl}}(R_v)
        =
        2\cdot\#\{\text{overorders of }R_v\}
        -
        \#\{\text{Gorenstein overorders of }R_v\}.
\]
Thus the entire global theorem rests on an explicit enumeration of two finite
sets attached to each local cubic order: all overorders and the Gorenstein
overorders among them.

\subsubsection{Explicit local enumeration}
Sections \ref{sec:irred} and \ref{sec:split} carry out this enumeration
case-by-case according to the local cubic algebra $E\otimes_F F_v$.  In the
irreducible cases we treat the unramified and totally ramified cubic field
extensions.  In the split cases we treat the splitting types $(1\,2)$,
$(1\,1^2)$, and $(1\,1\,1)$. 

For each splitting type we introduce normal forms for local orders, impose the
conditions for such a normal form to contain the fixed order $R_v$, record the
inclusion relations among the resulting overorders, and identify precisely which
of them are Gorenstein.

For arbitrary local cubic orders this produces the summation formulas in
Theorems \ref{thm:unrfor}, \ref{thm:totramfor},
\ref{thm:splitunroverorders}, \ref{thm:splitramorders}, and
\ref{thm:totsplitformula}.  Under the Gorenstein hypothesis on $R$, the
conductor constraints simplify these summations to the polynomial local factors
$G_\tau(q_v;\mathbf f(v))$ computed in Corollaries
\ref{cor:gorenstein_calculation_unr}, \ref{cor:gorenstein_calculation_tr},
\ref{cor:gorenstein_calculation12}, \ref{cor:gorenstein_calculation112}, and
\ref{cor:gorenstein_calculation111}.

\subsection{Application to exact orbit counting for Bhargava's $2 \times 3 \times 3$ cubes}
The representation theory of prehomogeneous vector spaces enters arithmetic
statistics through asymptotic counting, for instance in the work of
Sato--Shintani \cite{SatoShintani74} and Bhargava \cite{Bhargava05}.  The
present paper addresses a different question: fix the associated cubic order
$R$ and count the exact number of $\mathrm{Cl}(R)$-equivalence classes, or
genera in the product formula sense, of integral orbits lying over that fixed
order.

Bhargava's parametrization of the $\Gamma$-space
$\mathbb Z^2\otimes\mathbb Z^3\otimes\mathbb Z^3$, where
\(\Gamma=\mathrm{GL}_2(\mathbb Z)\times\mathrm{SL}_3(\mathbb Z)\times
\mathrm{SL}_3(\mathbb Z)\), relates integral $\Gamma$-orbits to balanced pairs
of fractional ideals of cubic rings \cite{Bha04}.  When $R$ is Gorenstein, the
associated orbits with cubic ring $R$ are parameterized by
$\overline{\mathrm{Cl}}(R)$.  Taking the quotient by $\mathrm{Cl}(R)$ is the
passage to genera, using the terminology justified by the product formula
above.  Consequently, combining the closed formula of Theorem
\ref{thm:intro_globalGorensteinProduct} with the fixed-ring orbit
identification proved in Theorem \ref{thm:bhargava_orbit_count}, we obtain an
exact formula for the number of $\mathrm{Cl}(R)$-equivalence classes of
integral $\Gamma$-orbits with prescribed Gorenstein cubic ring $R$.

This application illustrates that the overorder enumeration is useful beyond intrinsic ideal theory.  
The same local formulas also provide arithmetic input for automorphic and geometric questions, as explained next.

\subsection{Applications and further directions}
The same local enumeration has two further roles.  First, it has already been
used in an automorphic application to Beyond Endoscopy.  Second, it supplies a
finite local algebraic dictionary---normal forms, inclusion criteria,
Gorenstein criteria, and finite counting formulas---for geometric and
stack-theoretic developments beyond the present counting theorem.  The following
paragraphs place the paper in this broader context.

\subsubsection{Beyond Endoscopy for $\mathrm{GL}_3(\mathbb{Q})$}
The first application concerns Langlands' Beyond Endoscopy program.  
Following Altu\u{g}'s implementation of this strategy for $\mathrm{GL}_2(\mathbb{Q})$ via the trace formula and Poisson summation \cite{Alt1,Alt2,Alt3}, Deng and Espinosa developed a $\mathrm{GL}_3(\mathbb{Q})$ analogue \cite{DE}. 
Their approach is conditional on a conjectural factorization formula, whose role is to establish the functional equation for the completed $L$-function attached to a cubic order.
In \cite{Lee26}, one of the authors proves this functional equation unconditionally for Gorenstein cubic orders, using the explicit local overorder enumeration of the present paper.  
Consequently, the corresponding isolation of the trivial representation in \cite{DE} becomes unconditional.

The link with the present paper is direct.  The local factors of the relevant
$L$-function are expressed as sums over overorders; each overorder contributes a
factor depending on whether it is Gorenstein, together with a local zeta factor
and an index term.  The explicit overorder enumeration developed in Sections
\ref{sec:irred}--\ref{sec:split} is the main ingredient that rewrites each
local factor as a polynomial in $p^{-s}$.  This polynomial formulation plays a
key role in verifying the local functional equations case by case according to
the splitting type.

\subsubsection{Local input for Hitchin-theoretic extensions}
A companion paper \cite{CH_Hitchin} treats the Bass case from a geometric point
of view: it places the local smoothening method inside the
$\mathrm{GL}_n$-Hitchin fibration and relates the resulting strata to the
compactified Jacobian.  In that work, the enumeration of local overorders yields
a stratification of the compactified Jacobian and, under additional assumptions,
a closed formula for its cohomology.  As in the Bass case, the classifications
of overorders in Sections \ref{sec:irred}--\ref{sec:split} provide the finite
local input needed to formulate Hitchin-theoretic questions in the cubic case.
We plan to pursue this geometric extension in subsequent work.

\subsubsection{Non-Gorenstein local formulas and stack-theoretic comparison}\label{subsubsec:derived}
The closed Euler product in Theorem \ref{thm:globalGorensteinProduct} uses the
Gorenstein hypothesis, but the local analysis leading to it is not confined to
the Gorenstein case.  Theorems \ref{thm:unrfor}, \ref{thm:totramfor},
\ref{thm:splitunroverorders}, \ref{thm:splitramorders}, and
\ref{thm:totsplitformula} give finite explicit summation formulas for
arbitrary local cubic orders.

These formulas identify the source of the added local complexity.  Outside the
Gorenstein locus, the numbers of overorders and Gorenstein overorders remain
finite and explicitly computable.  However, they no longer collapse to the
one-parameter chain structure familiar from Bass orders, nor do they admit the
conductor-controlled polynomial local factors available in the Gorenstein cubic
case.

These finite formulas provide the arithmetic benchmark for any geometric or
categorical interpretation of the non-Gorenstein contribution to the fixed-ring
orbit problem.  A sequel in preparation studies this comparison through the
Artin stack of Wood balanced pairs, its groupoid-volume interpretation, and the
derived homological structure of the non-Gorenstein balancing condition.

\vspace{0.5em}\noindent\textbf{Organization.}  In Section \ref{sec:global_strategy}, we explain the global strategy and the reduction to local overorder counting. Section \ref{sec:local_notations} establishes local notations. Sections \ref{sec:irred}-\ref{sec:split} enumerate all local overorders and Gorenstein overorders in the irreducible and split cases. Section \ref{sec:globalGorenstein} combines these results to prove the global closed formula (Theorem \ref{thm:globalGorensteinProduct}). Section \ref{sec:bhargava} applies the formula to exact orbit genus counting in Bhargava's prehomogeneous vector space. Appendix \ref{app:formula} gives Python algorithms for evaluating the split-case summations.
          
\vspace{0.5em}\noindent\textbf{Acknowledgments.} We sincerely thank Stefano Marseglia for helpful comments.

\section{Global strategy and ideal classes}\label{sec:global_strategy}

 This is taken from \cite[Notations, Part 2]{CHL}.
\begin{itemize}
\item For a ring $A$, the set of maximal ideals is denoted by $|A|$.
\item If $A$ is a local ring, then the maximal ideal is denoted by $\mathfrak{m}_A$, and the residue field is denoted by $\kappa_A$. 
If $K$ is  a non-Archimedean local field, then  we sometimes use $\kappa_K$ to denote the residue field of the ring of integers in $K$, if there is no confusion. 

\item For $a\in A$ or $\psi(x) \in A[x]$ with a local ring $A$, $\overline{a}\in \kappa_A$ or $\overline{\psi(x)}\in \kappa_A[x]$ is the reduction of $a$ or $\psi(x)$ modulo $\mathfrak{m}_A$, respectively.

\item Let $F$ be a number field with $\mfo$ its ring of integers.
\item For $v\in |\mfo|$, 
let $F_v$ be the $v$-adic completion of $F$ with $\mfo_v$  the ring of integers of $F_v$, $\pi_v$  a uniformizer in $\mfo_v$, and $\kappa_v$ its residue field. Let $q_v=\#\kappa_v$.

\item Let $E$ be a cubic field extension of $F$ and let $R$ be an order of $E$ in the sense of Definition \ref{def:setup}.(1), which will be stated below.




\end{itemize}
 We will define the notion of order,  fractional ideal, conductor, ideal quotient, ideal class group, and ideal class monoid in a general situation especially containing \'etale algebras and number fields, following \cite{Ma24}. 

\begin{definition}[{\cite[Definition A]{CHL}}]\label{def:setup}
Let $Z$ be a Dedekind domain with field of fractions $Q$. 
Let $K$ be an \'etale $Q$-algebra.
Note that $Z$ is always  $\mfo$ or $\mfo_v$ in this paper.

\begin{enumerate}
   
\item{\cite[the second paragraph of Section 2.2]{Ma24}}  An order of $K$ is a subring $\Mfo$ of $K$ such that $\Mfo$ is a finitely generated $Z$-module containing $Z$ and such that $\Mfo\otimes_ZQ\cong K$. 
     An order $\Mfo'$ of $K$ is called an overorder of $\Mfo$ if $\Mfo \subset \Mfo'$.

\item{\cite[the first paragraph of Section 2.3]{Ma24}}    
 A fractional $\Mfo$-ideal $I$ is a finitely generated $\Mfo$-submodule of $K$ such that $I\otimes_ZQ\cong K$. The set of fractional $\Mfo$-ideals is closed under multiplication and thus forms a monoid.
Loc. cit. explains the following criterion of being  a fractional ideal.
\begin{enumerate}
    \item If $I$ is a finitely generated $\Mfo$-submodule of $K$ (e.g.  an ideal of $\Mfo$), then $I$ is a fractional $\Mfo$-ideal if and only if $I$ contains a non-zero divisor of $K$. 

\item If $\Mfo$ is a Noetherian domain, then the above criterion yields that a finitely generated  $\Mfo$-submodule $I (\neq 0)$ of $K$ is a fractional $\Mfo$-ideal. 
In particular, if $xI\subset \Mfo$ for an $\Mfo$-submodule $I (\neq 0)$ of $K$ with   $x\neq 0$ in  $\Mfo$, then $I$ is a fractional $\Mfo$-ideal.
\end{enumerate}
\item The maximal order of $K$ is denoted by $\Mfo_K$. 
Here unique existence of $\Mfo_K$ is explained in \cite[the first paragraph of Section 2]{Ma20} or \cite[the third paragraph of Section 2.2]{Ma24}.

\item The conductor $\mathfrak{f}(\Mfo)$ of an order $\Mfo$ is the biggest ideal of $\Mfo_K$ which is contained in $\Mfo$. In other words, $\mathfrak{f}(\Mfo)=\{a\in \Mfo_K\mid a\Mfo_K\subset \Mfo\}$.
Note that $\mathfrak{f}(\Mfo)$ is also an ideal of $\Mfo$.

\item For an order $\Mfo$ of $K$ and for an ideal $I$ of $\Mfo_K$, 
$\langle \Mfo, I\rangle$ is the subring of $\Mfo_K$ generated by elements of $\Mfo$ and $I$.

\item    The ideal quotient $(I:J)$ for two fractional $\Mfo$-ideals $I$ and $J$ is defined to be 
$    (I:J)=\{x\in K \mid xJ\subset I\}$.
Then $(I:J)$ is also a fractional $\Mfo$-ideal. 
Here we emphasize that $(I:I)$ is the biggest order over which $I$ is a fractional ideal. 
We refer to \cite[Sections 2.1-2.3]{Ma24} for detailed explanations. 

\item A fractional $\Mfo$-ideal $I$ is called invertible if there exists a fractional $\Mfo$-ideal $J$ such that $IJ=\Mfo$. If it exists, then it is uniquely characterized by $J=\left(\Mfo:I\right)$.
The set of invertible ideals is closed under multiplication and inverse, so as to form a group.

\item The ideal class group $\mathrm{Cl}(\Mfo)$ of $\Mfo$ is defined to be the group of equivalence classes of invertible $\Mfo$-ideals up to multiplication by an element of $K^{\times}$. 

\item The ideal class monoid $\overline{\mathrm{Cl}}(\Mfo)$
of $\Mfo$ is defined to be the monoid of equivalence classes of fractional $\Mfo$-ideals up to multiplication by an element of $K^{\times}$. 
\end{enumerate}
\end{definition}

\begin{definition}[{\cite[Definition 2.1]{CHL}, \cite[Proposition 3.4]{Ma24}}]\label{def:invariantofO}
For an order $\Mfo$ of $K$, 
 \begin{enumerate}
    \item the Serre invariant $S(\Mfo):=[\Mfo_{K}:\Mfo]_Z$ of $\Mfo$ is the length of $\Mfo_K/\Mfo$ as a $Z$-module.
    
    \item An order $\Mfo$ is called Gorenstein if every fractional $\Mfo$-ideal $I$ with $\Mfo=(I:I)$ is invertible.
\end{enumerate}
\end{definition}

For the global setting with $F$ and $E$, we establish the local completions as follows.
For $v\in |\mfo|$, $w\in |R|$, and $\mathfrak{P}\in |\Mfo_E|$, let
\[
\left\{
\begin{array}{l}
\textit{$R_v\cong R\otimes_\mfo \mfo_v$ be the $v$-adic completion of $R$};\\
\textit{$E_v\cong R_v\otimes_{\mfo_v} F_v$ be the ring of total fractions of $R_v$};\\
\textit{$X_{R_v}$ be the set of fractional $R_v$-ideals so that $\overline{\cl}(R_v)=E_v^\times\backslash X_{R_v}$},
\end{array} \right.
\]
\[
\left\{
\begin{array}{l}
\textit{$R_w$ be the $w$-adic completion of $R$};\\
\textit{$E_w\cong E\otimes_R R_w$ be the ring of total fractions of $R_w$};\\
\textit{$X_{R_w}$ be the set of fractional $R_w$-ideals  so that $\overline{\cl}(R_w)=E_w^\times\backslash X_{R_w}$}.
\end{array} \right.
\]
\[
\left\{
\begin{array}{l}
    \textit{$\Mfo_{E_{\mathfrak{P}}}$ be the $\mathfrak{P}$-adic completion of $\Mfo_E$};\\
    \textit{$E_{\mathfrak{P}}$ be the field of fractions of $\Mfo_{E_{\mathfrak{P}}}$}.
\end{array}\right.
\]
Note that $R_w$ is a local ring (possibly non-integral domain). 

\begin{remark}[{\cite[Remark B]{CHL}}]\label{rmk:descriptionofRv}
 By \cite[(4) in the proof of Lemma 2.16]{Ma24},  
$R_v\cong \bigoplus\limits_{w|v,~ w\in |R|}R_w$.
Applying $(-)\otimes_{\mfo_v} F_v$ yields
 \[ E_v\cong \bigoplus\limits_{w|v,~ w\in |R|}E_w ~~~ \textit{ and }   ~~~ \Mfo_{E_v}\cong \bigoplus\limits_{w\mid v, ~ w\in |R|} \Mfo_{E_w}. \]
 Note that $E_w$ may not be a field.
 \end{remark}

\begin{definition}[{\cite[Definition 5.1]{CHL}}]\label{def:globalclcl}
Let $\Mfo$ be an order in the sense of Definition \ref{def:setup}.
For an overorder $\Mfo'$ of $\Mfo$, we define the following sets:
\[\left\{
\begin{array}{l}
     \mathrm{cl}(\Mfo'):=\{ [I] \in \overline{\mathrm{Cl}}(\Mfo') \mid (I:I)=\Mfo' \}=\{ [I] \in \overline{\mathrm{Cl}}(\Mfo) \mid (I:I)=\Mfo' \};  \\
     \overline{\mathrm{cl}(\Mfo')}:=\mathrm{Cl}(\Mfo)\backslash \mathrm{cl}(\Mfo').
\end{array}\right.
\]    
Here we consider $\overline{\mathrm{Cl}}(\Mfo')$ as a subset of $\overline{\mathrm{Cl}}(\Mfo)$.
The set  $\overline{\mathrm{cl}(\Mfo')}$ is well-defined since $(JI:JI)=JJ^{-1}(I:I)=(I:I)$ for $J\in \mathrm{Cl}(\Mfo)$.
Note that $\mathrm{cl}(\Mfo')$ is denoted by $\mathrm{ICM}_{\Mfo'}(\Mfo)$ in \cite{Ma24}.

\end{definition}


\begin{proposition}[{\cite[Proposition 5.2]{CHL}}]\label{prop:stratforglobal}
For an order $\Mfo$ in the sense of Definition \ref{def:setup},
we have the following results:  
\[
\#\overline{\mathrm{Cl}}(\Mfo)=\sum_{\Mfo\subset \Mfo'\subset \Mfo_K}\#\mathrm{cl}(\Mfo')
~~~~~~~~~\textit{ and } ~~~~~~~~~~~  
\mathrm{Cl}(\Mfo)\backslash \overline{\mathrm{Cl}}(\Mfo)
=\bigsqcup_{\Mfo\subset \Mfo'\subset \Mfo_K}\overline{\mathrm{cl}(\Mfo')}.
\]  
\end{proposition}

\begin{proposition}[Product formula]\cite[Proposition 5.3 and Corollary 5.5.(2)]{CHL}\label{prop:global_local} 
For an order $R$ of a number field $E$ with $\Mfo$ an overorder of $R$, the  following map is bijective:
\[     \mathrm{Cl}(R)\backslash \overline{\mathrm{Cl}}(R)\longrightarrow
    \prod\limits_{v\in |\mfo|}\overline{\cl}(R_v) ~~~  \textit{ and }  ~~~ \overline{\mathrm{cl}(\Mfo)}\cong\prod\limits_{v\in|\mfo|} \mathrm{cl}(\Mfo\otimes_R R_v), ~~~~~~~  \{I\}\mapsto \prod_{v\in |\mfo|} [I\otimes_R R_v].    \]
Note that both $\cl(R_w)$ and $\cl(R_v)$ are trivial by \cite[Lemma 2.17]{Ma24}. 
\end{proposition}
\begin{proof}
    This follows directly from \cite[Proposition 5.3, Corollary 5.5.(2)]{CHL} and Remark \ref{rmk:descriptionofRv}.
\end{proof}

\begin{lemma}\label{lem:overorderwhengoren}
    Let $R$ be an order of a cubic field $E$. Let $\Mfo'$ be an overorder of $R_v$. Then $1\leq\#\mathrm{cl}(\Mfo')\leq 2$, and $\#\mathrm{cl}(\Mfo')=1$ if and only if $\Mfo'$ is Gorenstein.
\end{lemma}

\begin{proof}
    By Remark \ref{rmk:descriptionofRv}, $\Mfo'$ decomposes into the direct product of its local components: $\Mfo'\cong \bigoplus_{w\mid v, w\in |\Mfo'|}\Mfo'_w$.
    We use the \textit{Cohen-Macaulay type} of $\Mfo'$, as defined in \cite[Definition 3.2]{Ma24}.

    First, suppose that $\Mfo'$ is not a local ring. 
    Then each local component $\Mfo'_w$ must have rank at most $2$ over $\mfo_v$. Because any order of rank at most $2$ over a discrete valuation ring is Gorenstein, each $\Mfo'_w$ is Gorenstein. Consequently, $\Mfo'$ is Gorenstein, and its type is $1$ by \cite[Proposition 3.4]{Ma24}.

    Next, suppose that $\Mfo'$ is a local ring. Its Cohen-Macaulay type is bounded by its rank minus $1$ (cf. \cite[Proposition 4.9]{Ma24}), which is $3-1=2$. 

 In all cases, the type of $\Mfo'$ is at most $2$. 
    By \cite[Proposition 3.4]{Ma24}, the type of $\Mfo'$ is $1$ if and only if $\Mfo'$ is Gorenstein if and only if $\#\mathrm{cl}(\Mfo')=1$.    
    If the type of $\Mfo'$ is $2$, then $\Mfo'$ is not Gorenstein. As shown above, this implies that $\Mfo'$ must be a local ring. For a local order of type $2$, \cite[Corollary 6.5]{Ma24} yields $\#\mathrm{cl}(\Mfo')=2$.
This completes the proof. 
\end{proof} 

\begin{strategy}\label{strategy}
    By the product formula in Proposition \ref{prop:global_local}, our strategy is to find a formula for $\#\overline{\mathrm{Cl}}(R_v)$ for all $v\in |\mfo|$.
By Proposition \ref{prop:stratforglobal},  $\#\overline{\mathrm{Cl}}(R_v)=\sum_{\Mfo'}\#\mathrm{cl}(\Mfo')$, where $\Mfo'$ runs over all overorders of $R_v$.
By Lemma \ref{lem:overorderwhengoren}, it suffices to count the number of Gorenstein overorders of $R_v$ and that of non-Gorenstein overorders of $R_v$.
\end{strategy}

\section{Local notations and preliminaries}\label{sec:local_notations}
      In this section and the following Sections \ref{sec:irred}-\ref{sec:split}, we shift our focus to the local setting. To ease the burden of notation, we drop the subscript
          $v$ and simply let $F$ denote a non-Archimedean local field, and let $\mfo$ be its ring of integers.
          The goal is to find a formula for $\#\overline{\mathrm{Cl}}(R)$ (which plays the role of $R_v$ in Strategy \ref{strategy}).
      We reset the following notations.

\begin{itemize}
\item Let $F$ be a  non-Archimedean local field  of any characteristic with $\mfo$  its ring of integers and $\kappa$  its residue field.
Let $\pi$ be a uniformizer in $\mfo$.
Let $q$ be the cardinality of the finite field $\kappa$.

\item
For a finite field extension $F'$ of $F$, we denote by 
$\pi_{F'}$  a uniformizer of $F'$,  by  $\Mfo_{F'}$ the ring of integers of $F'$, and  by  $\kappa_{F'}$ the residue field of $F'$.

\item For an element $x\in F'$, $\ord_{F'}(x)$ is the exponential valuation with respect to $\pi_{F'}$. 
If $F'=F$, then we sometimes use $\ord(x)$, instead of $\ord_{F}(x)$.
We set $\ord(0)=\infty$.

\item Let $E$ be an \'etale $F$-algebra of degree $3$.
The maximal order $\Mfo_E$ of $E$, an order $\Mfo$ of $E$, and the conductor $\mathfrak{f}(\Mfo)$ of $\Mfo$ are as described in Definition \ref{def:setup}. Note that $\Mfo$ is a semilocal ring.

\item We often use $R$ to stand for an order of $E$ and $\Mfo$ to stand for an overorder of $R$. 

\item For $a\in A$ or $\psi(x) \in A[x]$ with a flat $\mfo$-algebra $A$, $\overline{a}\in A \otimes_{\mfo} \kappa$ or $\overline{\psi(x)}\in A\otimes_{\mfo} \kappa[x]$ is the reduction of $a$ or $\psi(x)$ modulo $\pi$, respectively.

\end{itemize}




\begin{proposition}\label{prop:gorenstein_cond}
An order $\Mfo$ of $E$ is Gorenstein if and only if $S(\Mfo)=[\Mfo:\mathfrak{f}(\Mfo)]_\mfo$.
Furthermore, any simple extension of $\mfo$ is Gorenstein. Conversely, a Gorenstein order $\Mfo$ of $E$ is a simple extension of $\mfo$ unless $q=2$ and $\Mfo \cong \mfo \times \mfo \times \mfo$.
\end{proposition}
\begin{proof}
    By \cite[2.5]{Jen15}, $\Mfo$ is Gorenstein if and only if each localization $\Mfo_i$ of $\Mfo$ is Gorenstein.
    For each local ring $\Mfo_i$, \cite[Corollary 3.7]{JKE} states that $\Mfo_i$ is Gorenstein if and only if
    $S(\Mfo_i)=[\Mfo_i:\mathfrak{f}(\Mfo_i)]_\mfo$. The additivity of lengths proves the first equivalence, since $\Mfo$ is a semilocal ring.

    For the second part, we use the parametrization of cubic rings via binary cubic forms (see \cite[Theorem 1.2 and Section 4]{Woo11} or \cite[Section 3]{Bha04}). An order $\Mfo$ corresponds to a $GL_2(\mfo)$-equivalence class of a binary cubic form $f(x,y)$ over $\mfo$. Under this bijection, $\Mfo$ is Gorenstein if and only if $f$ is primitive (\cite[Corollary 2.3]{Woo11}), and $\Mfo$ is a simple extension of $\mfo$ if and only if $f$ represents a unit in $\mfo$ (\cite[Proposition 5]{Bha04} or \cite[Section 2.1]{Woo11}).
    
    If $\Mfo$ is a simple extension, then $f$ represents a unit, so that its coefficients generate the unit ideal $\mfo$. Thus $f$ is primitive and $\Mfo$ is Gorenstein.
    Conversely, let $\Mfo$ be Gorenstein, so that its associated form $f$ is primitive. Let $\bar{f}$ be the reduction of $f$ over the residue field $\kappa$.
    Then $\Mfo$ is a simple extension if and only if $\bar{f}$ does not vanish identically on the projective line $\mathbb{P}^1(\kappa)$.
    Since $\bar{f}$ is a non-zero cubic form, it has at most 3 roots in $\mathbb{P}^1(\kappa)$.
    If $q \ge 3$, then $\#\mathbb{P}^1(\kappa)=q+1 > 3$, and thus $\bar{f}$ always has a non-vanishing point in $\mathbb{P}^1(\kappa)$. 
    If $\Mfo$ is not completely split, then $\bar{f}$ has at most 2 roots, so even for $q=2$ it has a non-vanishing point in $\mathbb{P}^1(\kappa)$. 
    In these cases $\Mfo$ is a simple extension. The only remaining case is $q=2$ where $\bar{f}$ has 3 distinct roots, which occurs if and only if $\bar{f}(x,y) = xy(x+y)$ up to $GL_2(\kappa)$-equivalence. This corresponds to the completely split case $\Mfo \cong \mfo \times \mfo \times \mfo$.
\end{proof}
\begin{remark}
    The failure of monogenicity for $q=2$ corresponds to the unique primitive binary cubic form $xy(x+y)$ over $\mathbb{F}_2$ vanishing on all points of $\mathbb{P}^1(\mathbb{F}_2)$. More generally, any simple extension over a DVR is a complete intersection and thus Gorenstein in any degree. However, for degree $\ge 4$, there exist Gorenstein rings that are not complete intersections, so the equivalence between Gorenstein orders and simple extensions is specific to the quadratic and cubic cases.
\end{remark}

\begin{proposition}\label{prop:orderidealcounting}
    $\#\clb(R)=2\cdot\#\{\text{overorders of $R$}\}-\#\{\text{Gorenstein overorders of $R$}\}$.
\end{proposition}

\begin{proof}
\cite[Proposition 2.6]{CHL} yields 
$\#\overline{\mathrm{Cl}}(R)=\sum_{R\subset \Mfo \subset \Mfo_{E}}\#\mathrm{cl}(\Mfo)$.     
Here loc. cit. assumed $R$ to be an integral domain, but the proof works for general $R$. The claim  follows from Lemma \ref{lem:overorderwhengoren}.
\end{proof}


\section{Local theory: irreducible case}\label{sec:irred}
In this section, we treat the case where $E/F$ is  a cubic field extension.
Let $e$ be the ramification index and let $d=[\kappa_{E}:\kappa]$ so that $3=ed$.
In this case, an order $\Mfo$ is a local domain.  

For an overorder $\Mfo$ of $R$ so that  $R\subset \Mfo\subset \Mfo_{E}$, we define the integer $f(\Mfo)\in \Z_{\geq 0}$ such that $\mathfrak{f}(\Mfo)=\pi_E^{f(\Mfo)}\Mfo_E \left(\subset \Mfo\right)$.
We sometimes call $f(\Mfo)$ the conductor of $\Mfo$, if it does not cause confusion. 
\begin{equation}\label{eq:inequfs}
\textit{Note that } ~~~~~~~~~~    f(\Mfo)\leq  e\cdot S(\Mfo) \leq 3\cdot f(\Mfo).
\end{equation}

\subsection{The unramified case}\label{subsec:unramified} 
Suppose that $E/F$ is unramified with $x\in \Mfo_E^\times$ such that $\Mfo_E=\mfo[x]$ and  $\kappa_E=\kappa[\overline{x}]$. 
Since any order $R$ is an overorder of $\mathfrak{o}[\pi^{f(R)} x]$ by equality $\ord_E(\pi^{f(R)}x)= f(R)$,  we will work with $R=\mfo[\pi^nx]$ until Corollary \ref{cor:criterionoverorder}, and we will treat general $R$ in Theorem \ref{thm:unrfor}.

Let $X^3+c_2X^2+c_1X+c_0$ be the minimal polynomial for $x$ over $\mfo$.
Note that $\overline{X^3+c_2X^2+c_1X+c_0}$ is irreducible over $\kappa$.
We introduce the following notation:
\begin{equation}\label{eq:bracketnotation}
    (u \pi^ax,v \pi^bx^2)_\mfo:=\mfo\langle 1, u \pi^ax, v \pi^bx^2 \rangle ~~~~  \textit{ with $a,b\in \mathbb{Z}_{\geq 0}$ and $u, v\in \mfo^\times$,  as an $\mfo$-module}.
\end{equation}

\begin{lemma}\label{lem:overorder}
    Let $c\in \mfo$.
    \begin{enumerate}
    \item The $\mfo$-module $\Mfo:=(\pi^a(x+cx^2),\pi^bx^2)_\mfo$ is an overorder of $R=\mfo[\pi^nx]$ if and only if $a\leq 2b$, $b\leq 2a$, $a\leq n$, and $b\leq \ord(c)+n$.
    \item The $\mfo$-module $\Mfo':=(\pi^ax,\pi^b(x^2+cx))_\mfo$ is an overorder of $R=\mfo[\pi^nx]$ if and only if $a\leq 2b$, $b\leq 2a$, and $a\leq n$.
    \end{enumerate}
\end{lemma}
\begin{proof}
\begin{enumerate}
\item
    First, we characterize the condition when $\Mfo$ is closed under multiplication.
    We have 
    \begin{gather*}
    (\pi^a(x+cx^2))^2=
    (-2cc_0+c^2c_2c_0)\pi^{2a}+
    (c^2c_2c_1-c^2c_0-2cc_1)\pi^{a}\cdot\pi^a(x+cx^2)\\+(1+c^2c_2^2+c^2c_1-2cc_2-c^3c_2c_1+c^3c_0)\pi^{2a}x^2.
    \end{gather*}
    Suppose that $\ord(1+c^2c_2^2+c^2c_1-2cc_2-c^3c_2c_1+c^3c_0)>0$.
    Then $\overline{x+cx^2}\in \kappa_E$ satisfies a quadratic equation \[ \overline{(x+cx^2)^2}=\overline{(-2cc_0+c^2c_2c_0)+(c^2c_2c_1-c^2c_0-2cc_1)(x+cx^2)}, \]
    which is absurd since $\overline{x+cx^2}\notin \kappa$, so that its minimal polynomial over $\kappa$ is cubic. Thus $\ord(1+c^2c_2^2+c^2c_1-2cc_2-c^3c_2c_1+c^3c_0)=0$, so that $(\pi^a(x+cx^2))^2\in \Mfo$ if and only if $b\leq 2a$.
    Next, 
    \begin{gather*}
    \pi^a(x+cx^2)\cdot \pi^bx^2= (cc_2c_0-c_0)\pi^{a+b}+(cc_2c_1-cc_0-c_1)\pi^{b}\cdot (\pi^a(x+cx^2)) \\
    +(cc_2^2-c_2-c^2c_2c_1+c^2c_0)\pi^a\cdot\pi^bx^2,
    \end{gather*}
    so that it is contained in $\Mfo$ without any conditions.
    Lastly, we have
    \begin{gather*}
    (\pi^bx^2)^2=
    c_2c_0\pi^{2b}+(c_2c_1-c_0)\pi^{2b-a}\cdot \pi^a(x+cx^2)+(c_2^2-c_1-cc_2c_1+cc_0)\pi^{2b}x^2.
    \end{gather*}
    If $\ord(c_2c_1-c_0)>0$ so that $\overline{c_2c_1}=\overline{c_0}$, then the polynomial $\overline{X^3+c_2X^2+c_1X+c_0}$ has a zero $\overline{-c_2}$ in $\kappa$, which is a contradiction to the irreducibility of $\overline{X^3+c_2X^2+c_1X+c_0}$.
    Therefore, $(\pi^bx^2)^2\in \Mfo$ if and only if $2b-a\geq 0$, in other words, $a\leq 2b$.

    Now, we characterize the condition when $\Mfo$ contains $R$.
    First, we see that $\pi^{2n}x^2\in \Mfo$ if and only if $b\leq 2n$.
    Next, we have 
    $\pi^nx=\pi^{n-a}\cdot \pi^a(x+cx^2)-c\pi^nx^2$.
    Therefore, $\pi^nx\in \Mfo$ if and only if $n-a\geq 0$ and $\ord(c)+n\geq b$.

    \item The proof is identical to that of (1).
    We have 
    \begin{gather*}
            (\pi^b(x^2+cx))^2=    (c_2c_0-2cc_0)\pi^{2b}+(c^2-2cc_2+c_2^2-c_1)\cdot \pi^{b}(\pi^b(x^2+cx))\\
            +(2c^2c_2-c^3-cc_2^2-cc_1+c_2c_1-c_0)\pi^{2b}x.
    \end{gather*}
Then $\ord(2c^2c_2-c^3-cc_2^2-cc_1+c_2c_1-c_0)=0$, since otherwise $\overline{x^2+cx}\in \kappa_E\setminus \kappa$ is a root of a quadratic equation over $\kappa$, which is a contradiction.
    Thus, $(\pi^b(x^2+cx))^2\in \Mfo'$ if and only if $a\leq 2b$.
    Next,
    \begin{gather*}
    \pi^ax\cdot \pi^b(x^2+cx)=
    -c_0\pi^{a+b}+(c-c_2)\pi^a\cdot \pi^b(x^2+cx)+(c_2c-c^2-c_1)\pi^{a+b}x,
    \end{gather*}
    which is contained in $\Mfo'$ without any conditions.
    Lastly, we have \[ (\pi^ax)^2=\pi^{2a}x^2=\pi^{2a-b}\cdot \pi^b(x^2+cx)-\pi^{2a}cx, \] which is contained in $\Mfo'$ if and only if $b\leq 2a$.

    It is immediate that $\pi^nx\in \Mfo'$ if and only if $a\leq n$.
    Furthermore, we have $\pi^{2n}x^2=\pi^{2n-b}\cdot \pi^b(x^2+cx)-c\pi^{2n}x$, which is contained in $\Mfo'$ if and only if $b\leq 2n$ and $a\leq 2n+\ord(c)$.
    The latter inequality $a\leq 2n+\ord(c)$ follows if $a\leq n$, and thus $\Mfo'$ contains $R$ if and only if $a\leq n$ and $b\leq 2n$.
    The inequality $b\leq 2n$ is implied from $a\leq n$ if $b\leq 2a$.\qedhere
    \end{enumerate} 
\end{proof}
\begin{proposition}\label{prop:overorderform}
Let $\Mfo$ be an overorder of $R=\mfo[\pi^nx]$ with $s:=S(\Mfo)$ and $f:=f(\Mfo)$. 
Then $0\leq s-f\leq \min(n,f)$. 
$\Mfo$ takes one of the following forms
\[
\Mfo_{s,f,c}^1:=(\pi^{s-f}(x+c\pi^{f-\min(n,f)}x^2),\pi^fx^2 )_\mfo ~~~~  \textit{  or  } ~~~~ \Mfo_{s,f,c}^2:=(\pi^fx,\pi^{s-f}(c\pi x+x^2))_\mfo,
\]
where $c\in \mfo$ and $f>s-f$ for $\Mfo_{s,f,c}^2$. 
When $s$ and $f$ are fixed, $\Mfo_{s,f,c}^1$ (resp. $\Mfo_{s,f,c}^2$) is completely determined by the reduction of $c$ modulo $(\pi^{\min(n,f)-(s-f))})$ (resp. modulo $(\pi^{2f-s-1})$).






\end{proposition}
\begin{proof}
Put $b:=f$ and $a:=s-f$. Then $0\leq a$ by \eqref{eq:inequfs},  $a\leq b$ since $(\pi^bx,\pi^bx^2)_\mfo\subseteq \Mfo$ so that $s\leq 2b$, and $a\leq n$ since $\pi^n x, \pi^b x^2\in \Mfo$ so that $s\leq n+b$.
This completes the desired inequalities.
If $b=0$ then $\Mfo=\Mfo_E$. 
Suppose that $b>0$. Then $\pi^{b-1}x\notin \Mfo$ or $\pi^{b-1}x^2\notin \Mfo$ since $b=f(\Mfo)$. 
We treat them separately in the following. 

Suppose that $\pi^{b-1}x^2\notin \Mfo$. 
We claim that $\Mfo=(\pi^a(x+c\pi^{b-\min(n,b)}x^2),\pi^bx^2 )_\mfo$ for some $c\in \mfo$.
To prove the claim, we observe that $\Mfo/\mfo\langle1,\pi^bx^2\rangle$ is torsion-free as an $\mfo$-module, since $\pi^{b-1}x^2\notin \Mfo$.
Then \cite[Corollary 2.9]{Con14} yields that the set $\{1,\pi^bx^2\}$ is extended to a basis for $\Mfo$ as a free $\mfo$-module by attaching one more element of $\Mfo$, say $a_0+a_1x+a_2x^2$ with $a_i\in \mfo$.
We may and do assume $a_0=0$.

Consider the $\mfo$-submodule of $\Mfo_E$ generated by the set $\{1,a_1x+a_2x^2,x^2\}$, denoted by $\Mfo'$. Then $[\Mfo':\Mfo]=b$.
On the other hand, $\Mfo'$ is also generated by the set $\{1,a_1x,x^2\}$ and thus $[\Mfo_E:\Mfo']=\ord(a_1)=a$. 
By multiplying a unit in $\mfo$, we additionally assume that $a_1=\pi^a$.
Then, since $b=f(\Mfo)$, we have $\pi^bx=\pi^{b-a}(\pi^ax+a_2x^2)-a_2\pi^{b-a}x^2\in \Mfo$  so that $a_2\pi^{b-a}x^2\in \Mfo$. 
Since $\pi^{b-1}x^2\notin \Mfo$, 
 $\ord(a_2)+b-a\geq b$ and thus $\ord(a_2)\geq a$. 
Furthermore, by Lemma \ref{lem:overorder}, we have $b\leq \ord(a_2)+n-a$.
These two inequalities yield the desired form of $\Mfo$. 

If $c-c' \in (\pi^{\min(n,b)-a})$, then $\pi^a(x+c\pi^{b-\min(n,b)}x^2)-\pi^a(x+c'\pi^{b-\min(n,b)}x^2)\in \mfo\langle\pi^bx^2\rangle$, so that $(\pi^a(x+c\pi^{b-\min(n,b)}x^2),\pi^bx^2)_\mfo=(\pi^a(x+c'\pi^{b-\min(n,b)}x^2),\pi^bx^2)_\mfo$.
Suppose that $c-c'\notin (\pi^{\min(n,b)-a})$ and that $\Mfo=(\pi^a(x+c\pi^{b-\min(n,b)}x^2),\pi^bx^2)_\mfo=(\pi^a(x+c'\pi^{b-\min(n,b)}x^2),\pi^bx^2)_\mfo$.
Then $\pi^a(x+c\pi^{b-\min(n,b)}x^2)-\pi^a(x+c'\pi^{b-\min(n,b)}x^2)=\pi^{a+b-\min(n,b)}(c-c')x^2\in \Mfo$, which is a contradiction since $\pi^{b-1}x^2\notin \Mfo$.
This completes the uniqueness of $c$ up to congruence modulo $(\pi^{\min(n,b)-a)})$.

Now, suppose that $\pi^{b-1}x\notin \Mfo$.
By the same argument as above, we obtain $\Mfo=(\pi^bx,\pi^a(\widetilde{c}x+x^2))_\mfo$ for some $\widetilde{c}\in \mfo$, and $\Mfo$ is completely determined by the reduction of $\widetilde{c}$ modulo $(\pi^{b-a})$.
If $b=a$, then $\Mfo=(\pi^bx,\pi^ax^2)_\mfo$, which reduces to the first case. Thus, we consider the case where $b>a$.
If $\widetilde{c}\in \mfo^\times$, then we observe that $(\pi^bx,\pi^a(x^2+\widetilde{c}x))_\mfo=(\pi^a(x+\widetilde{c}^{-1}x^2),\pi^bx^2)_\mfo$. 
This also reduces to the first case since $n\geq b$ by Lemma \ref{lem:overorder}.
Thus, we may and do write $\widetilde{c}=c\pi$ for $c\in \mfo$, so that $\Mfo=(\pi^bx,\pi^a(c\pi x+x^2))_\mfo$.
\end{proof}

\begin{corollary}\label{cor:criterionoverorder}
Let $\Mfo_{s,f,c}^i$ be as in Proposition \ref{prop:overorderform}. 
Then

$\left\{\begin{array}{l}
\textit{$\Mfo_{s,f,c}^1$: overorder of $R$ if and only if } \frac{1}{2}s\leq f\leq \frac{2}{3}s,~ s\leq f+n;\\
\textit{$\Mfo_{s,f,c}^2$: overorder of $R$ if and only if } \frac{1}{2}s< f\leq \frac{2}{3}s,~ f\leq n.
\end{array}\right. ~~~~~~   \textit{  Here}$    

\begin{enumerate}
    \item $S(\Mfo_{s,f,c}^i)=s$ and $f(\Mfo_{s,f,c}^i)=f$  for $i=1,2$;
    \item 
   $\left\{\begin{array}{l}
\Mfo_{s,f,c}^1=\Mfo_{s',f',c'}^1 \textit{  if and only if }
 (s,f)= (s',f') \textit{ and } c-c'\in (\pi^{\min(n,f)-(s-f))});\\ 
\Mfo_{s,f,c}^2=\Mfo_{s',f',c'}^2 \textit{  if and only if }  (s,f)= (s',f') \textit{ and } c-c'\in (\pi^{2f-s-1});
\end{array}\right.$ 
    
    \item $\Mfo_{s,f,c}^1\neq \Mfo_{s',f',c'}^2$ for $s'<2f'$;

    \item  $\Mfo_{s,f,c}^i$ is Gorenstein if and only if  $\frac{3}{2}f=s$ for $i=1,2$.
\end{enumerate}

These give an explicit classification of all overorders of $R=\mfo[\pi^nx]$. 
\end{corollary}

\begin{proof}
The necessary and sufficient conditions for $\Mfo^i_{s,f,c}$ to be an overorder of $R$ are a direct consequence of Lemma \ref{lem:overorder}.
The claim (1)  follows from the inequalities $f\geq s-f$,  
$s-f\leq \ord(c)-\min(n,f)+s$ for $\Mfo_{s,f,c}^1$, and $s-f+\ord(c)+1> s-f$ for $\Mfo_{s,f,c}^2$.
The claim (2) follows from the claim (1) and Proposition \ref{prop:overorderform}. 
    The claim (4) follows from Proposition \ref{prop:gorenstein_cond}.

For the claim (3), $\Mfo_{s,f,c}^1\neq \Mfo_{s',f',c'}^2$ if $(s,f)\neq (s',f')$ by the claim (2). 
When $(s',f')=(s,f)$, $\pi^{s-f}(x+c\pi^{f-\min(n,f)}x^2)\notin (\pi^fx,\pi^{s-f}(c\pi x+x^2))_\mfo$ since $f>s-f$. This completes the proof. 
\end{proof}

We now derive a formula for $\#\clb(R)$ of a general (i.e., not necessarily simple) cubic order $R$.
\begin{lemma}\label{lem:generaloverorder}
    \begin{enumerate}
        \item An order $(\pi^{a'}(x+c'x^2),\pi^{b'}x^2)_\mfo$ contains an order $(\pi^{a}(x+cx^2),\pi^bx^2)_\mfo$ if and only if $b'\leq b, a'\leq a$, and $b'\leq \ord(c-c')+a$.
        \item An order $(\pi^{b'}x,\pi^{a'}(x^2+c'x))$ contains an order $(\pi^{a}(x+cx^2),\pi^bx^2)$ if and only if $b'\leq \ord(c')+b$, $a'\leq b$, $a'\leq\ord(c)+a$, and $b'\leq a+\ord(1-cc')$.
    \end{enumerate}
    The lemma holds as well when we exchange the places of $x$ and $x^2$.
\end{lemma}

\begin{proof}
The first claim follows from equations:
\[ \pi^bx^2=\pi^{b-b'}\cdot \pi^{b'}x^2, ~~~~~~ \pi^{a}(x+cx^2)=\pi^{a-a'}\cdot\pi^{a'}(x+c'x^2)+\pi^{a}(c-c')x^2. \]
 The second claim follows from equations:
        \[ \pi^bx^2=\pi^{b-a'}\cdot \pi^{a'}(x^2+c'x)-c'\pi^bx, ~\pi^{a}(x+cx^2)=c\pi^{a-a'}\cdot \pi^{a'}(x^2+c'x)+(1-cc'
        )\pi^{a}x. \]
    The last assertion is valid by symmetry.
\end{proof}
\begin{remark}\label{rmk:lem29}
    Lemma \ref{lem:generaloverorder} is proved using only module theory over $\mfo$, and thus holds for a totally ramified case in Section \ref{subsec:totram} as well. 
    This will be used in the proof of Theorem \ref{thm:totramfor}.
\end{remark}
\begin{theorem}\label{thm:unrfor}
    Let $R$ be an order that is not necessarily simple.
    Let $s=S(R)$ and $f=f(R)$.
    Then $\#\clb(R)$ is expressed as follows, only depending on $s$ and $f$:
    \[ \#\clb(R)=\begin{cases}
    \displaystyle\frac{q^{\frac{s}{2}+2}+7q^{\frac{s}{2}+1}+4q^{\frac{s}{2}}+(2s-2)q^{\frac{f}{2}}(1-q)-8q^{\frac{f}{2}+1}-3q^{s-f+1}-q^{s-f}}{q^{\frac{f}{2}}(q-1)^2} &\text{if } s,f\text{ even};\\[1.5em]
    \displaystyle\frac{4q^{\frac{s}{2}+1}+7q^{\frac{s}{2}}+q^{\frac{s}{2}-1}+(2s-2)q^{\frac{f-1}{2}}(1-q)-8q^{\frac{f+1}{2}}-q^{s-f+1}-3q^{s-f}}{q^{\frac{f-1}{2}}(q-1)^2}&\text{if } s \text{ even, }f \text{ odd};\\[1.5em]
    \displaystyle\frac{4q^{\frac{s+3}{2}}+7q^{\frac{s+1}{2}}+q^{\frac{s-1}{2}}+(2s-2)q^{\frac{f}{2}}(1-q)-8q^{\frac{f}{2}+1}-3q^{s-f+1}-q^{s-f}}{q^\frac{f}{2}(q-1)^2}& \text{if }s \text{ odd, }f\text{ even};\\[1.5em]
    \displaystyle\frac{q^{\frac{s+3}{2}}+7q^{\frac{s+1}{2}}+4q^{\frac{s-1}{2}}+(2s-2)q^{\frac{f-1}{2}}(1-q)-8q^{\frac{f+1}{2}}-q^{s-f+1}-3q^{s-f}}{q^{\frac{f-1}{2}}(q-1)^2}&\text{if }s,f \text{ odd}.
    \end{cases}. \]
\end{theorem}
 \begin{proof}
Let $R':=\mfo[\pi^fx]$ so that $R'\subset R$ (cf. the beginning of Section \ref{subsec:unramified}). 
Then $R=\Mfo_{s,f,c}^i$ for $i=1 \textit{ or } 2$ by Corollary \ref{cor:criterionoverorder}.
    Our strategy is to count the number of overorders of $R(=\Mfo_{s,f,c}^i)$ with fixed Serre invariant $s'$ and that of  Gorenstein overorders.
 Then we use Proposition \ref{prop:orderidealcounting}. 

     \textbf{1. Case $R=\Mfo_{s,f,c}^1$.}
    We first treat the case when $i=1$. 
    Then $\frac{3}{2}f\leq s\leq 2f$ and $c\in \mfo$ is uniquely determined up to $(\pi^{2f-s})$ by Corollary \ref{cor:criterionoverorder}.
    By Lemma \ref{lem:generaloverorder} and Corollary \ref{cor:criterionoverorder}, an overorder of $R$ is of the form  $\Mfo_{s',f',c'}^{i'}$ for a unique $i', s',$ and $f'$ satisfying the following conditions:
\begin{enumerate}
    \item     When $i'=1$, $c'\in \mfo$ is  uniquely determined  up to    $(\pi^{2f'-s'})$ and 
\[
(i)~ \frac{3}{2}f'\leq s'\leq 2f', ~ (ii)~ f'\leq f,~ (iii)~ s'-f'\leq s-f,~ (iv)~ f'\leq \ord(c-c')+s-f.
\]
For a fixed $(s',f')$ satisfying the above $(i)$-$(iii)$, we count the number of $c'$, up to  $(\pi^{2f'-s'})$,  satisfying $(iv)$. 
By translation in $\mfo/(\pi^{2f'-s'})$, it is the same as the number of $c-c'$ in $(iv)$ up to $(\pi^{2f'-s'})$, which is  $q^{\min((s-f)-(s'-f'),~ 2f'-s')}$.
Next, for a fixed $s'$, we have $\max(\lceil \frac{s'}{2}\rceil,s'-s+f)\leq f'\leq \min(\lfloor\frac{2}{3}s'\rfloor,f)$ by $(i)$-$(iii)$.

\item When $i'=2$, $c'\in \mfo$ is uniquely determined up to $(\pi^{2f'-s'-1})$ and 

\[
(i)~ \frac{3}{2}f'\leq s'< 2f', (ii)~  f'\leq1+ \ord(c')+f,
(iii)~ s'-f'\leq f, (iv)~ s'-f'\leq \ord(c)+s-f, (v)~ f'\leq s-f
\]
Note that $s'-f'<s-f$ by a combination of the inequality $s'-f'<f'$ from $(i)$ and $(v)$.
Thus $(iii)$ (combined with the inequality $s-f\leq f$) and $(iv)$ are redundant.  
    In addition, $(ii)$ is induced by $(v)$ since $s-f\leq f$.
    Therefore, for a fixed $(s',f')$ satisfying $(i)$ and $(v)$, there are $q^{2f'-s'-1}$ choices for $c'$, up to $(\pi^{2f'-s'-1})$.
    For a fixed $s'$, we have $\frac{s'}{2}< f'\leq \min(s-f,\lfloor\frac{2}{3}s'\rfloor)$ by $(i)$ and $(v)$.
    \end{enumerate}

    We sum up these with respect to $s'$.
The number of overorders of $R$ with Serre invariant $s'$ is  
\[ 
\begin{cases}
 \sum\limits_{f'={\lceil\frac{s'}{2}\rceil}}^{\lfloor\frac{2}{3}s'\rfloor}q^{2f'-s'}+\sum\limits_{f'=\lfloor\frac{s'}{2}\rfloor+1}^{\lfloor\frac{2}{3}s'\rfloor}q^{2f'-s'-1}=\sum\limits^{2\lfloor\frac{2s'}{3}\rfloor-s'}_{i=0}q^i  &\text{if } \lfloor\frac{2}{3}s'\rfloor\leq s-f;\\
    \sum\limits_{f'=\max({\lceil\frac{s'}{2}\rceil},s'-s+f)}^{\lfloor\frac{2}{3}s'\rfloor}q^{\min((s-f)-(s'-f'),2f'-s')} +\sum\limits_{f'=\lfloor\frac{s'}{2}\rfloor+1}^{s-f}q^{2f'-s'-1}=\sum\limits^{s-f-s'+\lfloor\frac{2}{3}s'\rfloor}_{i=0}q^i
  &\text{if } s-f< \lfloor\frac{2}{3}s'\rfloor\leq f;\\
 \sum\limits_{f'=\max({\lceil\frac{s'}{2}\rceil},s'-s+f)}^{f}q^{\min((s-f)-(s'-f'),2f'-s')}+\sum\limits_{f'=\lfloor\frac{s'}{2}\rfloor+1}^{s-f}q^{2f'-s'-1}=\sum\limits^{s-s'}_{i=0}q^i &\text{if } f<\lfloor\frac{2}{3}s'\rfloor.
\end{cases}
 \]

The number of Gorenstein overorders of $R$ with Serre invariant $s'$ is  
    \[ \begin{cases}
q^{\frac{s'}{3}}+q^{\frac{s'}{3}-1}  &\text{if } \frac{2}{3}s'\leq s-f, ~ 3\mid s',  \textit{ and } f'=\frac{2}{3}s';\\[1.5em]
q^{s-f-\frac{s'}{3}} &\text{if } s-f<\frac{2}{3}s'\leq f, ~  3\mid s', \textit{ and } f'=\frac{2}{3}s';\\
0 & \text{otherwise}.
    \end{cases}. \]
    
    

   
    We now plug all calculations into the equation in Proposition \ref{prop:orderidealcounting}. 
    \begin{multline*}
        \#\clb(R)=1+
        \sum\limits_{1\leq s',\lfloor\frac{2}{3}s'\rfloor\leq s-f}\frac{2(q^{2\lfloor\frac{2}{3}s'\rfloor-s'+1}-1)}{q-1}-
        \sum\limits_{0<\frac{2}{3}s'\leq s-f, 3\mid s'}(q^{\frac{s'}{3}}+q^{\frac{s'}{3}-1})\\+
        \sum\limits_{s-f<\lfloor\frac{2}{3}s'\rfloor\leq f}\frac{2(q^{s-f-s'+\lfloor\frac{2}{3}s'\rfloor+1}-1)}{q-1}-
        \sum\limits_{s-f<\frac{2}{3}s'\leq f,3\mid s'}q^{s-f-\frac{s'}{3}}
        +\sum\limits_{f<\lfloor\frac{2}{3}s'\rfloor,s'\leq s}\frac{2(q^{s-s'+1}-1)}{q-1}.
    \end{multline*}
    This sum is equal to the formula described in the statement.
    
     \textbf{2. Case $R=\Mfo_{s,f,c}^2$.}
Now let $i=2$. 
Then $\frac{3}{2}f\leq s< 2f$ and $c\in \mfo$ is uniquely determined up to $(\pi^{2f-s-1})$ by Corollary \ref{cor:criterionoverorder}.
By Lemma \ref{lem:generaloverorder} and Corollary \ref{cor:criterionoverorder}, an overorder of $R$ is of the form $\Mfo_{s',f',c'}^{i'}$ for a unique $i', s',$ and $f'$ satisfying the following conditions:
    \begin{enumerate}
    \item When $i'=1$, $c'\in \mfo$ is uniquely determined up to $(\pi^{2f'-s'})$ and
    \[
    (i)~\frac{3}{2}f'\leq s'\leq 2f',(ii)~f'\leq \ord(c')+f,(iii)~s'-f'\leq f,(iv)~s'-f'\leq \ord(c)+1+s-f,(v)~f'\leq s-f.
    \]
 By the same argument as in the case where $i=1$ and $i'=2$, the conditions $(ii), (iii)$, and $(iv)$ are redundant.
    For a fixed $(s',f')$ satisfying $(i)$ and $(v)$, there are $q^{2f'-s'}$ choices for $c'$, up to $(\pi^{2f'-s'})$.
    For a fixed $s'$, we have $\frac{s'}{2}\leq f'\leq \min(s-f,\lfloor\frac{2}{3}s'\rfloor)$ by $(i)$ and $(v)$.
    
    \item When $i'=2$, $c'\in \mfo$ is uniquely determined up to $(\pi^{2f'-s'-1})$ and 
    \[
    (i)~\frac{3}{2}f'\leq s'<2f',(ii)~f'\leq f,(iii)~s'-f'\leq s-f,(iv)~f'\leq \ord(c-c')+s-f+1.
    \]
 By the same argument as in the case where $i=i'=1$, the number of $c'$ satisfying $(iv)$ up to $(\pi^{2f'-s'-1})$, for a fixed $(s',f')$,  is  
    $q^{\min((s-f)-(s'-f'),2f'-s'-1)}$.
    For a fixed $s'$, 
   we have $\max(\lfloor\frac{s'}{2}\rfloor+1,s'-s+f)\leq f'\leq \min(\lfloor\frac{2}{3}s'\rfloor,f)$ by $(i)$-$(iii)$.
    \end{enumerate}
   We sum up these with respect to $s'$. The number of overorders of $R$ with Serre invariant $s'$ is
    \[ 
    \begin{cases}
    \sum\limits_{f'=\lfloor\frac{s'}{2}\rfloor+1}^{\lfloor\frac{2}{3}s'\rfloor}q^{2f'-s'-1}+\sum\limits_{f'={\lceil\frac{s'}{2}\rceil}}^{\lfloor\frac{2}{3}s'\rfloor}q^{2f'-s'}=\sum\limits^{2\lfloor\frac{2s'}{3}\rfloor-s'}_{i=0}q^i  &\text{if } \lfloor\frac{2}{3}s'\rfloor\leq s-f;\\
    \sum\limits_{f'=\max(\lfloor\frac{s'}{2}\rfloor+1,s'-s+f)}^{\lfloor\frac{2}{3}s'\rfloor}q^{\min((s-f)-(s'-f'),2f'-s'-1)}+
    \sum\limits_{f'=\lceil\frac{s'}{2}\rceil}^{s-f}q^{2f'-s'}=\sum\limits^{s-f-s'+\lfloor\frac{2}{3}s'\rfloor}_{i=0}q^i
      &\text{if } s-f< \lfloor\frac{2}{3}s'\rfloor\leq f;\\
     \sum\limits_{f'=\max(\lfloor\frac{s'}{2}\rfloor+1,s'-s+f)}^{f}q^{\min((s-f)-(s'-f'),2f'-s'-1)}+\sum\limits_{f'=\lceil\frac{s'}{2}\rceil}^{s-f}q^{2f'-s'}=\sum\limits^{s-s'}_{i=0}q^i &\text{if } f<\lfloor\frac{2}{3}s'\rfloor.
    \end{cases}
     \]
    The number of Gorenstein overorders of $R$ with Serre invariant $s'$ is  
    \[ \begin{cases}
    q^{\frac{s'}{3}}+q^{\frac{s'}{3}-1}  &\text{when } \frac{2}{3}s'\leq s-f, ~ 3\mid s',  \textit{ and } f'=\frac{2}{3}s';\\[1.5em]
    q^{s-f-\frac{s'}{3}} &\text{when } s-f<\frac{2}{3}s'\leq f, ~  3\mid s', \textit{ and } f'=\frac{2}{3}s';\\
    0 & \text{otherwise}.
    \end{cases}. \] 
The above formulas are exactly the same as those for the case when $i=1$.
    Therefore, the formula for $\#\clb(R)$ when $i=2$ is identical to that for the case when $i=1$. 
    \end{proof}

\begin{corollary}\label{cor:gorenstein_calculation_unr}  
If the order $R$ in Theorem \ref{thm:unrfor} is Gorenstein, then by Proposition \ref{prop:gorenstein_cond} it is a simple extension of $\mfo$, taking the form $R=\mfo[\pi^nx]$ such that $\Mfo_E=\mfo[x]$. In this case, we have $S(R)=3n$ and $\mathfrak{f}(R) = (\pi^{2n})$, so the conductor exponent is $f = 2n$. The formula is simplified as follows:

\[\#\clb(R)=\begin{cases}
         \displaystyle{\frac{4q^{\frac{n+3}{2}} + 7q^{\frac{n+1}{2}} + q^{\frac{n-1}{2}} -( 6n+9)q + (6n- 3)}{(q-1)^2}}&\text{if } n \text{ odd};\\[14pt]
       \displaystyle\frac{q^{\frac{n}{2}+2} + 7q^{\frac{n}{2}+1} + 4q^{\frac{n}{2}} -( 6n+9)q + (6n- 3)}{(q-1)^2}&\text{if } n \text{ even}.
    \end{cases}.\]
\end{corollary}

\begin{example}\label{ex:simple-nontrivial}
We illustrate the classification of overorders and identify which are Gorenstein, using Corollary \ref{cor:criterionoverorder} and Lemma \ref{lem:generaloverorder}.
Let $\{c_1, \cdots, c_q\}$ be representatives for $\mfo/\pi \mfo$ with $c_1 = 0$ and $\ord(c_i)= 0$ for $2\leq i\leq q$. 
In the following diagrams, the Gorenstein overorders are highlighted in \textbf{bold}. 
When $R = \mathfrak{o}[\pi^2 x]$ with $x \in \Mfo_E^\times$ so that $S(R) =6$ and $f(R)=4$, we have the following:

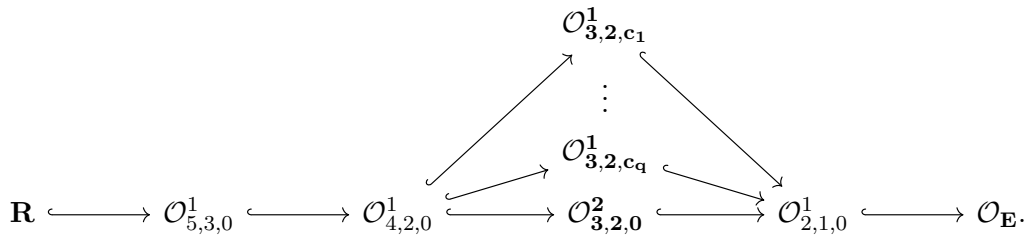
\begin{center}
\begin{tikzcd}[row sep=0.1em, column sep=3.5em]
    & & & \mathbf{\Mfo^1_{3,2,c_1}} \arrow[dddr, hook, crossing over] & & \\
    & & & \vdots & & \\
    & & & \mathbf{\Mfo^1_{3,2,c_q}} \arrow[dr, hook, crossing over] & & \\
    \mathbf{R} \arrow[r, hook] & \Mfo^1_{5,3,0} \arrow[r, hook] & \Mfo^1_{4,2,0} \arrow[uuur, hook] \arrow[ur, hook] \arrow[r, hook] & \mathbf{\Mfo^2_{3,2,0}} \arrow[r, hook] & \Mfo^1_{2,1,0} \arrow[r, hook] & \mathbf{\Mfo_E}.
\end{tikzcd}
\end{center}

When $R = \mathfrak{o}[\pi^3 x]$ with $x \in \Mfo_E^\times$ so that $S(R) =9$ and $f(R)=6$, we have the following diagram:

\begin{center}
\begin{tikzcd}[row sep=0.1em, column sep=2.0em]
    & & & \mathbf{\Mfo^1_{6,4,c_1}} \arrow[r, hook] & \Mfo^1_{5,3,c_1} \arrow[dddr, hook, crossing over] & & \mathbf{\Mfo^1_{3,2,c_1}} \arrow[dddr, hook, crossing over] & & \\
    & & & \vdots & \vdots & & \vdots & & \\
    & & & \mathbf{\Mfo^1_{6,4,c_q}} \arrow[uur, hook] & \Mfo^1_{5,3,c_q} \arrow[dr, hook, crossing over] & & \mathbf{\Mfo^1_{3,2,c_q}} \arrow[dr, hook, crossing over] & & \\
    \mathbf{R} \arrow[r, hook] & \Mfo^1_{8,5,0} \arrow[r, hook] & \Mfo^1_{7,4,0} \arrow[uuur, hook] \arrow[ur, hook] \arrow[r, hook] & \Mfo^1_{6,3,0} \arrow[uuur, hook] \arrow[ur, hook] \arrow[r, hook] & \Mfo^2_{5,3,0} \arrow[r, hook] & \Mfo^1_{4,2,0} \arrow[uuur, hook] \arrow[ur, hook] \arrow[r, hook] & \mathbf{\Mfo^2_{3,2,0}} \arrow[r, hook] & \Mfo^1_{2,1,0} \arrow[r, hook] & \mathbf{\Mfo_E}.
\end{tikzcd}
\end{center}

Here $\Mfo_{6,4,c_1}^1=\mfo[\pi^2x]$ and thus the diagram for $\mfo[\pi^2x]$ is a part of that for $\mfo[\pi^3x]$. 
\end{example}

\subsection{The totally ramified case}\label{subsec:totram}
The structure of this subsection, including proofs, is parallel to that of Section \ref{subsec:unramified}. 
Suppose that $E/F$ is totally ramified with a uniformizer $\pi_E$ of $E$.
Since any order $R$ is an overorder of $\mathfrak{o}[\pi^{\lceil f(R)/3 \rceil} \pi_E]$ by inequality $\ord_E(\pi^{\lceil f(R)/3 \rceil} \pi_E)> f(R)$,  we will work with $R=\mfo[\pi^n\cdot \pi_E]$ until Corollary \ref{cor:totramoverordercriterion}, and we will treat general $R$ in Theorem \ref{thm:totramfor}.
To synchronize the notation with Section \ref{subsec:unramified}, we denote $\pi_E$ by $x$ so that $R=\mfo[\pi^n x]$.
We continue to use the notation $(u\pi^ax,v\pi^bx^2)_\mfo$ in (\ref{eq:bracketnotation}).
Let $X^3+c_2X^2+c_1X+c_0$ be the minimal polynomial for $x$ over $\mfo$ so that $\overline{X^3+c_2X^2+c_X+c_0}=X^3$ with $\ord(c_0)=1$.

\begin{lemma}[Analogue of Lemma \ref{lem:overorder}]\label{lem:totramoverorder}
    Let $c\in \mfo$.
    \begin{enumerate}
\item The $\mfo$-module $\Mfo:=(\pi^a(x+cx^2),\pi^bx^2)_\mfo$ 
        is an overorder of $R=\mfo[\pi^nx]$ if and only if $a\leq 2b+1$, $b\leq 2a$, $a\leq n$, and $\ord(c)+n\geq b$.
        
        \item The $\mfo$-module $\Mfo':=(\pi^ax,\pi^b(x^2+ c\pi x))_\mfo$ 
        is an overorder of $R=\mfo[\pi^nx]$ if and only if $a\leq 2b+1$, $b\leq 2a$, and $a\leq n$.
  \end{enumerate}

\end{lemma}
\begin{proof}
    \begin{enumerate}
    \item First, we characterize the condition when $\Mfo$ is closed under multiplication.
    We have 
     \begin{gather*}
     (\pi^a(x+cx^2))^2=
    (-2cc_0+c^2c_2c_0)\pi^{2a}+
    (c^2c_2c_1-c^2c_0-2cc_1)\pi^{a}\cdot\pi^a(x+cx^2)\\+(1+c^2c_2^2+c^2c_1-2cc_2-c^3c_2c_1+c^3c_0)\pi^{2a}x^2.
    \end{gather*}
    Since $\ord(1+c^2c_2^2+c^2c_1-2cc_2-c^3c_2c_1+c^3c_0)=0$, $\Mfo$ contains $(\pi^a(x+cx^2))^2$ if and only if $b\leq 2a$.
    Next, we have
     \begin{gather*}
    \pi^a(x+cx^2)\cdot \pi^bx^2= (cc_2c_0-c_0)\pi^{a+b}+(cc_2c_1-cc_0-c_1)\pi^{b}\cdot (\pi^a(x+cx^2)) \\
    +(cc_2^2-c_2-c^2c_2c_1+c^2c_0)\pi^a\cdot\pi^bx^2,
    \end{gather*}
    so that it is contained in $\Mfo$ without any conditions.
     Lastly, we have
    \begin{gather*}
    (\pi^bx^2)^2=
    c_2c_0\pi^{2b}+(c_2c_1-c_0)\pi^{2b-a}\cdot \pi^a(x+cx^2)+(c_2^2-c_1-cc_2c_1+cc_0)\pi^{2b}x^2.
    \end{gather*}
    Since $\ord(c_2c_1-c_0)=1$, $(\pi^bx^2)^2$ is contained in $\Mfo$ if and only if $a\leq 2b+1$.

    Now, we characterize the condition when $\Mfo$ contains $R$.
    First, we see that $\pi^{2n}x^2\in \Mfo$ if and only if $b\leq 2n$.
    Next, we have 
    $\pi^nx=\pi^{n-a}\cdot \pi^a(x+cx^2)-c\pi^nx^2$.
    Therefore, $\pi^nx\in \Mfo$ if and only if $n-a\geq 0$ and $\ord(c)+n\geq b$.

    \item The proof is identical to that of (1).
    We have 
    \begin{gather*}
            (\pi^b(x^2+ c\pi x))^2=    (c_2c_0-2\pi cc_0)\pi^{2b}+(c^2\pi^2-2 c c_2\pi+c_2^2-c_1)\cdot \pi^{b}(\pi^b(x^2+ c\pi x))\\
            +(2c^2c_2\pi^2-c^3\pi^3- cc_2^2\pi- cc_1\pi+c_2c_1-c_0)\pi^{2b}x.
    \end{gather*}
    Since $\ord(2c^2c_2\pi^2-c^3\pi^3- cc_2^2\pi- cc_1\pi+c_2c_1-c_0)=1$, $(\pi^b(x^2+c\pi x))^2$ is contained in $\Mfo'$ if and only if $a\leq 2b+1$.
    Next, we have 
    \begin{gather*}
    \pi^ax\cdot \pi^b(x^2+ c\pi x)=
    -c_0\pi^{a+b}+( c\pi-c_2)\pi^a\cdot \pi^b(x^2+ c\pi x)+( c_2c\pi-c^2\pi^2-c_1)\pi^{a+b}x,
    \end{gather*}
    which is contained in $\Mfo'$ without any conditions.
    Lastly, we have \[ (\pi^ax)^2=\pi^{2a}x^2=\pi^{2a-b}\cdot \pi^b(x^2+ c\pi x)- c\pi^{2a+1}x, \] which is contained in $\Mfo'$ if and only if $b\leq 2a$.
    
     It is immediate that $\pi^nx\in \Mfo'$ if and only if $a\leq n$.
    Furthermore, we have $\pi^{2n}x^2=\pi^{2n-b}\cdot \pi^b(x^2+ c\pi x)-c\pi^{2n+1}x$, which is contained in $\Mfo'$ if and only if $b\leq 2n$ and $a\leq 2n+\ord(c)+1$.
    The latter inequality $a\leq 2n+\ord(c)+1$ follows if $a\leq n$, and thus $\Mfo'$ contains $R$ if and only if $a\leq n$ and $b\leq 2n$.
 The inequality $b\leq 2n$ is implied from $a\leq n$ if $b\leq 2a$. 
    \end{enumerate}    
\end{proof}
\begin{proposition}[Analogue of Proposition \ref{prop:overorderform}]\label{prop:totramoverorderform}
    Let $\Mfo$ be an overorder of $R=\mfo[\pi^nx]$ with $s:=S(\Mfo)$ and $f:=f(\Mfo)$.
$\Mfo$ takes one of the following forms
\[
\left\{\begin{array}{l}
\Mfo_{s,f,c}^1:=(\pi^{s-\frac{f}{3}}(x+c\pi^{\frac{f}{3}-\min(n,\frac{f}{3})}x^2),\pi^{\frac{f}{3}}x^2)_\mfo  \textit{ where } 3\mid f \textit{ and } 0\leq s-\frac{f}{3}\leq \frac{f}{3};\\
\Mfo_{s,f,c}^2:=(\pi^\frac{f+1}{3}x,\pi^{s-\frac{f+1}{3}}(c\pi x+x^2))_\mfo \textit{ where } 3\mid (f+1) \textit{ and } 0\leq s-\frac{f+1}{3}<\frac{f+1}{3},
\end{array}\right.
\]
where $c\in \mfo$. 
When $s$ and $f$ are fixed, $\Mfo_{s,f,c}^1$ (resp. $\Mfo_{s,f,c}^2$) is completely determined by the reduction of $c$ modulo $(\pi^{\min(n,\frac{f}{3})-(s-\frac{f}{3})})$ (resp. modulo $(\pi^{\frac{2(f+1)}{3}-s-1})$). 
\end{proposition}

\begin{proof}
    Firstly, we assume that $3\mid f$ so that $f=3b$.
    If $b=0$, then $R=\Mfo_E$.
Suppose that  $b>0$. Put $a:=s-b\geq 0$ by \eqref{eq:inequfs}.
Note that $a\leq b$, since $(\pi^bx,\pi^{b}x^2)_\mfo\subseteq \Mfo$ so that $s\leq 2b$.
    We claim that $\Mfo=(\pi^a(x+c\pi^{b-\min(n,b)}x^2),\pi^bx^2)_\mfo$ for some $c\in \mfo$.

We have $\pi^{b-1}x^2\notin \Mfo$, since $f=3b$. Thus, by the same argument as in the proof of Proposition \ref{prop:overorderform}, the set $\{1,\pi^bx^2\}$ is extended to a basis for $\Mfo$ as a free $\mfo$-module by attaching one more element of $\Mfo$, say $a_0+a_1x+a_2x^2$ with $a_i\in \mfo$. 
    We may and do assume that $a_0=0$.


    Consider the $\mfo$-submodule of $\Mfo_E$ generated by the set $\{1,a_1x+a_2x^2, x^2\}$, denoted by $\Mfo'$.
    Then $[\Mfo':\Mfo]=b$.
    On the other hand, $\Mfo'$ is also generated by the set $\{1,a_1x,x^2\}$ and thus $[\Mfo_E:\Mfo']=\ord(a_1)=a$.
    By multiplying a unit of $\mfo$, we additionally assume that $a_1=\pi^a$.
    Then, since $f(\Mfo)=3b$,  we have $\pi^bx=\pi^{b-a}(\pi^ax+a_2x^2)-a_2\pi^{b-a}x^2\in \Mfo$  so that $a_2\pi^{b-a}x^2\in \Mfo$. 
    Since $\pi^{b-1}x^2\notin \Mfo$, $\ord(a_2)+b-a\geq b$ and thus $\ord(a_2)\geq a$. 
    Furthermore, by Lemma \ref{lem:totramoverorder}, we have $b\leq \ord(a_2)+n-a$.
    These two inequalities yield the desired form of $\Mfo$.

    If $c- c'\in  (\pi^{\min(n,b)-a})$, then $\pi^a(x+c\pi^{b-\min(n,b)}x^2)-\pi^a(x+c'\pi^{b-\min(n,b)}x^2)\in \mfo\langle\pi^bx^2\rangle$, so that $(\pi^a(x+c\pi^{b-\min(n,b)}x^2),\pi^bx^2)_\mfo=(\pi^a(x+c'\pi^{b-\min(n,b)}x^2),\pi^bx^2)_\mfo$.
    Suppose that $c- c'\notin (\pi^{\min(n,b)-a})$ and that $\Mfo=(\pi^a(x+c\pi^{b-\min(n,b)}x^2),\pi^bx^2)_\mfo=(\pi^a(x+c'\pi^{b-\min(n,b)}x^2),\pi^bx^2)_\mfo$.
    Then $\pi^a(x+c\pi^{b-\min(n,b)}x^2)-\pi^a(x+c'\pi^{b-\min(n,b)}x^2)=\pi^{a+b-\min(n,b)}(c-c')x^2\in \Mfo$, which is a contradiction since $\pi^{b-1}x^2\notin \Mfo$.
    This completes the uniqueness of $c$ up to congruence modulo $(\pi^{\min(n,b)-a)})$.

    Now, suppose that $f=3b-1$, and again put $a:=s-b$. 
    Here $0\leq a$ by \eqref{eq:inequfs}.
    We claim that $\Mfo=(\pi^bx,\pi^a(x^2+c \pi x))_\mfo$ for some $c\in \mfo$ with $a<b$.

        Since $f=3b-1$, $\pi^{b-1}x\notin \Mfo$.
By the same argument as above,  $\Mfo$ is generated by $\{1,\pi^bx,a_1x^2+a_2x\}$ as an $\mfo$-module for  $a_i\in \mfo$.
    Consider the $\mfo$-submodule $\Mfo'$ of $\Mfo_E$ generated by $\{1,x,a_1x^2+a_2x\}$ (or equivalently by $\{1,x,a_1x^2\}$).
    Then  $[\Mfo':\Mfo]=b$ and $[\Mfo_E:\Mfo']=\ord(a_1)=s-b=a$.
    By multiplying a unit of $\mfo$, we additionally assume that $a_1=\pi^a$.
    Using $f(\Mfo)=3b-1$, we have $\pi^{b-1}x^2=\pi^{b-a-1}\cdot (\pi^ax^2+a_2x)-a_2\pi^{b-a-1}x\in \Mfo$ so that $b-a-1\geq 0$ and $\ord(a_2)+b-a-1\geq b$, since $\{1,\pi^bx,\pi^ax^2+a_2x\}$  is an $\mfo$-basis for   $\Mfo$.
    This yields the desired form of $\Mfo$.
    The uniqueness of $c$ modulo $(\pi^{b-a-1})$ is obtained by the parallel argument to the $f=3b$ case.

Lastly, if $f(\Mfo)=3b-2$, then $\pi_E^{3b-3}\Mfo_E=\pi^{b-1}\Mfo_E\subseteq \Mfo$ so that $f(\Mfo)=3b-2\leq 3b-3$, which is a contradiction.
    Therefore, there exists no overorder $\Mfo$ with $f=3b-2$.
\end{proof}

\begin{corollary}[Analogue of Corollary \ref{cor:criterionoverorder}]\label{cor:totramoverordercriterion}
Let $\Mfo_{s,f,c}^i$ be as in Proposition \ref{prop:totramoverorderform}. Then

$\left\{\begin{array}{l}
\textit{$\Mfo_{s,f,c}^1$: overorder of $R$ if and only if } 3\mid f,~ \frac{3}{2}s\leq f\leq 2s,~ 
s \leq \frac{f}{3}+n;\\
\textit{$\Mfo_{s,f,c}^2$: overorder of $R$ if and only if } 3\mid (f+1),~ \frac{3}{2}s+\frac{1}{2}\leq f\leq 2s,~   f\leq 3n-1.
\end{array}\right. ~~~\textit{  Here
}$

\begin{enumerate}
    \item $S(\Mfo_{s,f,c}^i)=s$ and $f(\Mfo_{s,f,c}^i)=f$  for $i=1,2$;
    \item 
   $\left\{\begin{array}{l}
\Mfo_{s,f,c}^1=\Mfo_{s',f',c'}^1 \textit{  if and only if }
 (s,f)= (s',f') \textit{ and } c-c'\in (\pi^{\min(n,\frac{f}{3})-(s-\frac{f}{3})});\\ 
\Mfo_{s,f,c}^2=\Mfo_{s',f',c'}^2 \textit{  if and only if }  (s,f)= (s',f') \textit{ and } c-c'\in (\pi^{\frac{2(f+1)}{3}-s-1});
\end{array}\right.$ 

\item $\Mfo_{s,f,c}^1\neq \Mfo_{s',f',c'}^2$;

\item  $\Mfo_{s,f,c}^i$ is Gorenstein if and only if $f=2s$
for $i=1,2$.    
\end{enumerate}
These give an explicit classification of all overorders of $R$. 
\end{corollary}

\begin{proof}
The necessary and sufficient conditions for $\Mfo^i_{s,f,c}$ to be an overorder of $R$ are a direct consequence of Lemma \ref{lem:totramoverorder}. 
The claim (1) follows from the inequalities $s-\frac{f}{3}\leq \frac{f}{3}$, $s-\frac{f}{3}\leq \ord(c)-\min(n,\frac{f}{3})+n$ for $\Mfo^1_{s,f,c}$, and $s-\frac{f+1}{3}\leq \frac{f+1}{3}$, $s-\frac{f+1}{3}\leq s-\frac{f+1}{3}+\ord(c)+1$ for $\Mfo^2_{s,f,c}$.
The claim (2) follows from the claim (1) and Proposition \ref{prop:totramoverorderform}. 
The claim (3) follows from the fact that $3\mid f(\Mfo^1_{s,f,c})$ and $3 \nmid f(\Mfo^2_{s',f',c'})$.
The claim (4) follows from Proposition \ref{prop:gorenstein_cond}.    
\end{proof}

\begin{theorem}[Analogue of Theorem \ref{thm:unrfor}]\label{thm:totramfor}
    Let $R$ be an order in $E$ that is not necessarily simple. 
    Let $s=S(R)$ and $f=f(R)$.
    Then $\#\clb(R)$ is expressed as follows, only depending on $s$ and $f$:
    \[
    \#\clb(R)=
    \begin{cases}
    \displaystyle\frac{q^{\frac{s}{2}+2}+9q^{\frac{s}{2}+1}+2q^{\frac{s}{2}}+2sq^{\frac{f}{6}}(1-q)-8q^{\frac{f}{6}+1}-3q^{s-\frac{f}{3}+1}-q^{s-\frac{f}{3}}}{q^{\frac{f}{6}}(q-1)^2} &\text{if } s,\frac{f}{3} \text{ even}; \\[1.5em]
     \displaystyle\frac{5q^{\frac{s}{2}+1}+7q^{\frac{s}{2}}+2sq^{\frac{f-3}{6}}(1-q)-8q^{\frac{f+3}{6}}-q^{s-\frac{f}{3}+1}-3q^{s-\frac{f}{3}}}{q^{\frac{f-3}{6}}(q-1)^2}&\text{if } s \text{ even, } \frac{f}{3} \text{ odd};\\[1.5em]
    \displaystyle\frac{5q^{\frac{s+3}{2}}+7q^{\frac{s+1}{2}}+2sq^{\frac{f}{6}}(1-q)-8q^{\frac{f}{6}+1}-3q^{s-\frac{f}{3}+1}-q^{s-\frac{f}{3}}}{q^{\frac{f}{6}}(q-1)^2}& \text{if } s \text{ odd, } \frac{f}{3} \text{ even}; \\[1.5em]
    \displaystyle\frac{q^{\frac{s+3}{2}}+9q^{\frac{s+1}{2}}+2q^{\frac{s-1}{2}}+2sq^{\frac{f-3}{6}}(1-q)-8q^{\frac{f+3}{6}}-q^{s-\frac{f}{3}+1}-3q^{s-\frac{f}{3}}}{q^{\frac{f-3}{6}}(q-1)^2}&\text{if } s,\frac{f}{3}\text{ odd} 
\end{cases}
    \]
    when $3\mid f$, and
    \[
    \#\clb(R)=
    \begin{cases}
    \displaystyle\frac{2q^{\frac{s}{2}+2}+9q^{\frac{s}{2}+1}+q^{\frac{s}{2}}+2sq^{\frac{f+1}{6}}(1-q)-8q^{\frac{f+1}{6}+1}-q^{s-\frac{f+1}{3}+2}-3q^{s-\frac{f+1}{3}+1}}{q^{\frac{f+1}{6}}(q-1)^2} &\text{if } s,\frac{f+1}{3} \text{ even}; \\[1.5em]
    \displaystyle\frac{7q^{\frac{s}{2}+1}+5q^{\frac{s}{2}}+2sq^{\frac{f-2}{6}}(1-q)-8q^{\frac{f+4}{6}}-3q^{s-\frac{f+1}{3}+1}-q^{s-\frac{f+1}{3}}}{q^{\frac{f-2}{6}}(q-1)^2}&\text{if } s \text{ even, } \frac{f+1}{3} \text{ odd};\\[1.5em]
    \displaystyle\frac{7q^{\frac{s+3}{2}}+5q^{\frac{s+1}{2}}+2sq^{\frac{f+1}{6}}(1-q)-8q^{\frac{f+1}{6}+1}-q^{s-\frac{f+1}{3}+2}-3q^{s-\frac{f+1}{3}+1}}{q^{\frac{f+1}{6}}(q-1)^2}& \text{if } s \text{ odd, } \frac{f+1}{3} \text{ even}; \\[1.5em]
    \displaystyle\frac{2q^{\frac{s+3}{2}}+9q^{\frac{s+1}{2}}+q^{\frac{s-1}{2}}+2sq^{\frac{f-2}{6}}(1-q)-8q^{\frac{f+4}{6}}-3q^{s-\frac{f+1}{3}+1}-q^{s-\frac{f+1}{3}}}{q^{\frac{f-2}{6}}(q-1)^2}&\text{if } s,\frac{f+1}{3}\text{ odd}
    \end{cases}
    \]
    when $3 \mid (f+1)$.
\end{theorem}
\begin{proof}
     Let $R' := \mathfrak{o}[\pi^{\lceil f/3 \rceil} x]$ so that $R' \subset R$. 
     Thus Corollary \ref{cor:totramoverordercriterion} yields that $R = \Mfo_{s,f,c}^i$ for some $i$ and $c$. 
     Our strategy is to count the number of overorders of $R (=\Mfo_{s,f,c}^i)$ with fixed Serre invariant $s'$ and that of Gorenstein overorders.
        Then we use Proposition \ref{prop:orderidealcounting} to compute $\#\clb(R)$.
     
     \textbf{1. Case $R=\Mfo_{s,f,c}^1$.}
     We first treat the case when $i=1$.
    Then $\frac{3}{2}s\leq f\leq 2s$  with $3\mid f$ and $c\in \mfo$ is uniquely determined up to $(\pi^{\frac{2}{3}f-s})$, by Corollary \ref{cor:totramoverordercriterion}.(2).
     By Lemma \ref{lem:generaloverorder} (cf. Remark \ref{rmk:lem29}) and Corollary \ref{cor:totramoverordercriterion}, an overorder of $R$ is of the form $\Mfo_{s',f',c'}^{i'}$ for a unique $i',s',$ and $f'$ satisfying the following conditions:
     \begin{enumerate}

    \item When $i'=1$, $c'\in \mfo$ is uniquely determined up to $(\pi^{\frac{2}{3}f'-s'})$ and 
    \[(i)~\frac{3}{2}s'\le f' \le 2s' \text{ and }3\mid f',~(ii)\frac{f'}{3}\leq \frac{f}{3},~(iii)s'-\frac{f'}{3}\leq s-\frac{f}{3},~(iv)\frac{f'}{3}\leq \ord(c-c')+s-\frac{f}{3}.\]
    For a fixed $(s', f')$ satisfying the above $(i)$-$(iii)$, we count the number of $c'$, up to $(\pi^{\frac{2}{3}f'-s'})$, satisfying $(iv)$. 
    By translation in $\mfo/(\pi^{\frac{2}{3}f'-s'})$, it is the same as the number of $c-c'$ in $(iv)$ up to $(\pi^{2f'-s'})$, which is  $q^{\min((s-\frac{f}{3})-(s'-\frac{f'}{3}),~ 2\frac{f'}{3}-s')}$.
    Next, for a fixed $s'$, we have $ \max(\lceil\frac{s'}{2}\rceil, s' - s + \frac{f}{3}) \le \frac{f'}{3} \le  \min(\lfloor\frac{2s'}{3}\rfloor, \frac{f}{3})$ by $(i)$-$(iii)$.
Let  $j = \frac{f'}{3}$ to simplify notation.

    \item
    When $i'=2$, $c'\in \mfo$ is uniquely determined up to $(\pi^{\frac{2(f'+1)}{3}-s'-1})$ and
    \begin{align*}
    (i)~ \frac{3}{2}s'+\frac{1}{2}\leq f'\leq 2s' \text{ and } 3\mid(f'+1), (ii)~\frac{f'+1}{3}\leq 1+\ord(c')+\frac{f}{3}, (iii)~s'-\frac{f'+1}{3}\leq \frac{f}{3}, \\(iv)~s'-\frac{f'+1}{3}\leq \ord(c)+s-\frac{f}{3}, (v)~\frac{f'+1}{3}\leq s-\frac{f}{3}.
    \end{align*}
    Note that $s'-\frac{f'+1}{3}<s-\frac{f}{3}$ by a combination of the inequality $s'-\frac{f'+1}{3}<\frac{f'+1}{3}$ from $(i)$ and $(v)$.
    Thus $(iii)$ (combined with the inequality $s-\frac{f}{3}\leq \frac{f}{3}$) and $(iv)$ are redundant.
    In addition, $(ii)$ is induced by $(v)$ since $s-\frac{f}{3}\leq \frac{f}{3}$.
    Therefore, for a fixed $(s', f')$ satisfying $(i)$ and $(v)$, there are $q^{\frac{2(f'+1)}{3} - s' - 1}$ choices of $c'$.
    For a fixed $s'$, we have $\frac{s'+1}{2} \leq\ \frac{f'+1}{3}  \le \min(s - \frac{f}{3} , \lfloor \frac{2s'+1}{3} \rfloor)$ by $(i)$ and $(v)$.
Let  $k = \frac{f'+1}{3}$  to simplify notation.
   \end{enumerate}
    We sum up these with respect to $s'$.
    The number of overorders of $R$ with Serre invariant $s'$ is
\[ 
\begin{cases}
\sum\limits_{j=\lceil\frac{s'}{2}\rceil}^{\lfloor\frac{2s'}{3}\rfloor} q^{2j-s'} + \sum\limits_{k=\lceil\frac{s'+1}{2}\rceil}^{\lfloor\frac{2s'+1}{3}\rfloor} q^{2k-s'-1} = \sum\limits_{i=0}^{\lfloor\frac{s'}{3}\rfloor} q^i & \text{if } \lfloor\frac{2s'+1}{3}\rfloor \le s - \frac{f}{3}; \\[15pt]
\sum\limits_{j=\max(\lceil\frac{s'}{2}\rceil, s'-s + \frac{f}{3})}^{\lfloor\frac{2s'}{3}\rfloor} q^{\min(s - \frac{f}{3}-s'+j, 2j-s')} + \sum\limits_{k=\lceil\frac{s'+1}{2}\rceil}^{s - \frac{f}{3}} q^{2k-s'-1} = \sum\limits_{i=0}^{s - \frac{f}{3}-s'+\lfloor\frac{2s'}{3}\rfloor} q^i & \text{if } 
\substack{s - \frac{f}{3} < \lfloor\frac{2s'+1}{3}\rfloor,\\ \lfloor\frac{2s'}{3}\rfloor \le \frac{f}{3};} \\[15pt]
\sum\limits_{j=\max(\lceil\frac{s'}{2}\rceil, s'-s + \frac{f}{3})}^{\frac{f}{3}} q^{\min(s - \frac{f}{3}-s'+j, 2j-s')} + \sum\limits_{k=\lceil\frac{s'+1}{2}\rceil}^{s - \frac{f}{3}} q^{2k-s'-1} = \sum\limits_{i=0}^{s-s'} q^i & \text{if } \frac{f}{3} < \lfloor\frac{2s'}{3}\rfloor.
\end{cases}
 \]

    The number of Gorenstein overorders of $R$ with Serre invariant $s'$ is 
     \[ 
    \begin{cases}
        q^{\frac{s'}{3}}&\text{when } \lfloor\frac{2s'+1}{3}\rfloor\leq s-\frac{f}{3}, f'=2s', \text{ and } 3\mid f', s';\\
        q^{\frac{s'-1}{3}}&\text{when } \lfloor\frac{2s'+1}{3}\rfloor\leq s-\frac{f}{3}, f'=2s', \text{ and } 3\mid (f'+1);\\
        q^{s-\frac{f}{3}-\frac{s'}{3}} &\text{when }s-\frac{f}{3}< \frac{2}{3}s'\leq \frac{f}{3},  f'=2s',\text{ and } 3\mid f', s';\\
        0 &\text{otherwise}.
    \end{cases}
     \]

    Here the first two are combined to be $q^{\lfloor\frac{s'}{3}\rfloor}$ when $\lfloor \frac{2s'+1}{3} \rfloor \le s - \frac{f}{3}$, $f'=2s'$, and $s'\not\equiv 2(\mathrm{mod} \text{ 3})$. 
    We now plug all calculations into the equation in Proposition \ref{prop:orderidealcounting}.
\begin{align*}
    \#\overline{\mathrm{Cl}}(R) &= 1 
    + \sum_{\substack{1 \le s' \\ \lfloor \frac{2s'+1}{3} \rfloor \le s - \frac{f}{3}}} \frac{2(q^{\lfloor\frac{s'}{3}\rfloor+1} - 1)}{q - 1} 
    - \sum_{\substack{\lfloor \frac{2s'+1}{3} \rfloor \le s - \frac{f}{3} \\ 0<s'\not\equiv 2 \pmod{3}}} q^{\lfloor\frac{s'}{3}\rfloor} \\
    &\quad + \sum_{\substack{s - \frac{f}{3} < \lfloor \frac{2s'+1}{3} \rfloor \\ \lfloor\frac{2}{3}s'\rfloor \le \frac{f}{3}}} \frac{2(q^{s - \frac{f}{3} - s' + \lfloor \frac{2}{3}s' \rfloor + 1} - 1)}{q - 1} 
    - \sum_{\substack{s - \frac{f}{3} < \frac{2}{3}s' \le \frac{f}{3} \\ 3 \mid s'}} q^{s - \frac{f}{3} - \frac{s'}{3}} 
    + \sum_{\substack{\frac{f}{3} < \lfloor \frac{2}{3}s' \rfloor \\ s' \le s}} \frac{2(q^{s - s' + 1} - 1)}{q - 1}.
\end{align*}
    This sum is equal to the formula when $3\mid f$ described in the statement.
    
        \textbf{2. Case $R=\Mfo_{s,f,c}^2$.}
    Now, let $i=2$. 
   Then $\frac{3}{2}s+\frac{1}{2}\leq f\leq 2s$  with $3\mid (f+1)$ and $c\in \mfo$ is uniquely determined up to $(\pi^{\frac{2(f+1)}{3}-s-1})$, by Corollary \ref{cor:totramoverordercriterion}.(2).
     By Lemma \ref{lem:generaloverorder} (cf. Remark \ref{rmk:lem29}) and Corollary \ref{cor:totramoverordercriterion}, an overorder of $R$ is of the form $\Mfo_{s',f',c'}^{i'}$ for a unique $i',s',$ and $f'$ satisfying the following conditions:

    \begin{enumerate}
        \item When $i'=1$, $c'\in \mfo$ is uniquely determined up to $(\pi^{\frac{2}{3}f'-s'})$ and 
        \begin{align*}
        (i)~\frac{3}{2}s'\leq f'\leq 2s' \text{ and }3\mid f',(ii)~\frac{f'}{3}\leq \ord(c')+\frac{f+1}{3},(iii)~s'-\frac{f'}{3}\leq \frac{f+1}{3},\\
        (iv)~s'-\frac{f'}{3}\leq\ord(c)+s-\frac{f+1}{3}+1,(v)~\frac{f'}{3}\leq s-\frac{f+1}{3}.
        \end{align*}

 By the same argument as in the case where $i=1$ and $i'=2$, conditions $(ii), (iii)$ and $(iv)$ are redundant.
        For a fixed $(s',f')$ satisfying $(i)$ and $(v)$, there are $q^{\frac{2}{3}f'-s'}$ choices for $c'$, up to $(\pi^{\frac{2}{3}f'-s'})$.
        For a fixed $s'$, we have $\frac{s'}{2}\leq \frac{f'}{3}\leq \min(s-\frac{f+1}{3},\lfloor\frac{2}{3}s'\rfloor)$. Let  $j = \frac{f'}{3}$ to simplify notation.

        \item When $i'=2$, $c'\in \mfo$ is uniquely determined up to $(\pi^{\frac{2(f'+1)}{3}-s'-1})$ and 
        \begin{align*}
        (i)~\frac{3}{2}s'+\frac{1}{2}\leq f'\leq 2s' \text{ and }3\mid(f'+1), (ii)~\frac{f'+1}{3}\leq \frac{f+1}{3}, (iii)~s'-\frac{f'+1}{3}\leq s-\frac{f+1}{3},\\ (iv)~\frac{f'+1}{3}\leq \ord(c-c')+s-\frac{f+1}{3}+1.
        \end{align*}
 By the same argument as in the case where $i=i'=1$, 
 the number of $c'$ satisfying $(iv)$ up to $(\pi^{\frac{2(f'+1)}{3}-s'-1})$, for a fixed $(s', f')$, is $q^{\min((s-\frac{f+1}{3})-(s'-\frac{f'+1}{3}),\frac{2(f'+1)}{3}-s'-1)}$.
        For a fixed $s'$, we have $\max(\lceil\frac{s'+1}{2}\rceil,s'-s+\frac{f+1}{3})\leq \frac{f'+1}{3}\leq \min(\lfloor\frac{2s'+1}{3}\rfloor,\frac{f+1}{3})$ by $(i)$-$(iii)$.
        Let  $k = \frac{f'+1}{3}$  to simplify notation.
    \end{enumerate}
     We sum up these with respect to $s'$.  The number of overorders of $R$ with Serre invariant $s'$ is 
    \[ 
    \begin{cases}
        \sum\limits_{j=\lceil\frac{s'}{2}\rceil}^{\lfloor\frac{2s'}{3}\rfloor}q^{2j-s'}+\sum\limits_{k=\lceil \frac{s'+1}{2}\rceil}^{\lfloor\frac{2s'+1}{3}\rfloor}q^{2k-s'-1}=\sum\limits_{i=0}^{\lfloor\frac{s'}{3}\rfloor}q^i &\text{if } \lfloor\frac{2s'}{3}\rfloor\leq s-\frac{f+1}{3};\\
        \sum\limits_{j=\lceil\frac{s'}{2}\rceil}^{s-\frac{f+1}{3}}q^{2j-s'}+\sum\limits_{k=\max(\lceil \frac{s'+1}{2}\rceil,s'-s+\frac{f+1}{3})}^{\lfloor\frac{2s'+1}{3}\rfloor}q^{\min((s-\frac{f+1}{3})-(s'-k),2k-s'-1)}=
        \sum\limits_{i=0}^{s-\frac{f+1}{3}-s'+\lfloor\frac{2s'+1}{3}\rfloor}q^i &\text{if } \substack{ s-\frac{f+1}{3}< \lfloor\frac{2s'}{3}\rfloor, \\ \lfloor\frac{2s'+1}{3}\rfloor\leq \frac{f+1}{3};}\\
        \sum\limits_{j=\lceil\frac{s'}{2}\rceil}^{s-\frac{f+1}{3}}q^{2j-s'}+\sum\limits_{k=\max(\lceil \frac{s'+1}{2}\rceil,s'-s+\frac{f+1}{3})}^{\frac{f+1}{3}}q^{\min((s-\frac{f+1}{3})-(s'-k),2k-s'-1)}=
        \sum\limits_{i=0}^{s-s'}q^i &\text{if }\frac{f+1}{3}<\lfloor\frac{2s'+1}{3}\rfloor.
    \end{cases}
     \]
    The number of Gorenstein overorders of $R$ with Serre invariant $s'$ is 
    \[ 
    \begin{cases}
        q^{\lfloor\frac{s'}{3}\rfloor}&\text{when } \lfloor\frac{2s'}{3}\rfloor\leq s-\frac{f+1}{3}, f'=2s', \text{ and } s'\not\equiv2(\mathrm{mod}\text{ }3);\\
        q^{s-\frac{f+1}{3}-\frac{s'-1}{3}} &\text{when }s-\frac{f+1}{3}+1< \frac{2s'+1}{3}\leq \frac{f+1}{3}, f'=2s', \text{ and } s'\equiv1(\mathrm{mod}\text{ }3);\\
        0 &\text{otherwise}.
    \end{cases}
     \]
    We now plug all calculations into the equation in Proposition \ref{prop:orderidealcounting}.
    \begin{align*}
    &\#\overline{\mathrm{Cl}}(R) = 1 + 
    \sum_{\substack{1 \le s', \\ \lfloor \frac{2s'}{3} \rfloor \le s -  \frac{f+1}{3}}} \frac{2(q^{\lfloor\frac{s'}{3}\rfloor+1} - 1)}{q - 1} -
    \sum_{\substack{\lfloor \frac{2s'}{3} \rfloor \le s -  \frac{f+1}{3}, \\ 0<s'\not\equiv 2(\mathrm{mod} \text{ 3})}} q^{\lfloor\frac{s'}{3}\rfloor} ~ +
    \\&\quad\sum_{\substack{s -  \frac{f+1}{3}  < \lfloor \frac{2}{3}s' \rfloor, \\ \lfloor\frac{2s'+1}{3}\rfloor\le  \frac{f+1}{3}}} \frac{2(q^{s - \frac{f+1}{3}  - s' + \lfloor \frac{2s'+1}{3} \rfloor + 1} - 1)}{q - 1} 
    - \sum_{\substack{s-\frac{f+1}{3}+1< \frac{2s'+1}{3}  \le  \frac{f+1}{3},\\ s'\equiv 1(\mathrm{mod}\text{ }3)}} q^{s -  \frac{f+1}{3} \ - \frac{s'-1}{3}} 
    + \sum_{\substack{\frac{f+1}{3}  < \lfloor \frac{2s'+1}{3} \rfloor,\\ s' \le s}} \frac{2(q^{s - s' + 1} - 1)}{q - 1}.
    \end{align*}
    Again, this sum is equal to the formula when $3\mid (f+1)$ described in the statement.
\end{proof}

\begin{corollary}\label{cor:gorenstein_calculation_tr}
If the order $R$ in Theorem \ref{thm:totramfor} is Gorenstein, then by Proposition \ref{prop:gorenstein_cond} it is a simple extension of $\mfo$, taking the form $R^1=\mfo[\pi^n \pi_E]$ or $R^2=\mfo[\pi^n \pi_E^2]$. In this case, we have $S(R^1)=3n$ and $\mathfrak{f}(R^1)=(\pi_E^{6n})$ so the conductor exponent is $f=6n$, whereas $S(R^2)=3n+1$ and $\mathfrak{f}(R^2)=(\pi_E^{6n+2})$ so the conductor exponent is $f=6n+2$. The formula is simplified as follows:

  \begin{align*}
    \#\overline{\mathrm{Cl}}(R^1) = &
\begin{cases} 
\displaystyle \frac{5q^{\frac{n+3}{2}} + 7q^{\frac{n+1}{2}} - (6n+11)q + (6n-1)}{(q-1)^2} & \text{if } n \text{ is odd;} \\[14pt]
\displaystyle \frac{q^{\frac{n}{2}+2} + 9q^{\frac{n}{2}+1} + 2q^{\frac{n}{2}} - (6n+11)q + (6n-1)}{(q-1)^2} & \text{if } n \text{ is even.}
\end{cases}\\
 \#\overline{\mathrm{Cl}}(R^2) = &
\begin{cases}
\displaystyle \frac{7q^{\frac{3n+3}{2}} + 5q^{\frac{3n+1}{2}} - (6n+13)q^{n+1} + (6n+1)q^n}{q^n(q-1)^2} & \text{if } n \text{ is odd;} \\[14pt]
\displaystyle \frac{2q^{\frac{3n}{2}+2} + 9q^{\frac{3n}{2}+1} + q^{\frac{3n}{2}} - (6n+13)q^{n+1} + (6n+1)q^n}{q^n(q-1)^2} & \text{if } n \text{ is even.}
\end{cases}
    \end{align*}
\end{corollary}

\begin{example}[Analogue of Example \ref{ex:simple-nontrivial}]\label{ex:simpleram}
We illustrate the classification of overorders and identify which are Gorenstein, using Corollary \ref{cor:totramoverordercriterion} and Lemma \ref{lem:generaloverorder} (applicable by Remark \ref{rmk:lem29}). 
Let $\{c_1, \cdots, c_q\}$ be representatives for $\mathfrak{o}/\pi\mathfrak{o}$ with $c_1=0$ and $\mathrm{ord}(c_i)=0$ for $2 \le i \le q$. 
In the following diagrams, the Gorenstein overorders are highlighted in \textbf{bold}. 

      \textbf{1. Case $\mathfrak{o}[\pi^2 \pi_E]$ and $\mathfrak{o}[\pi^3 \pi_E]$.}
When $R = \mathfrak{o}[\pi^2 \pi_E]$ with $S(R) = 6$ and $f(R) = 12$, we have the following diagram:

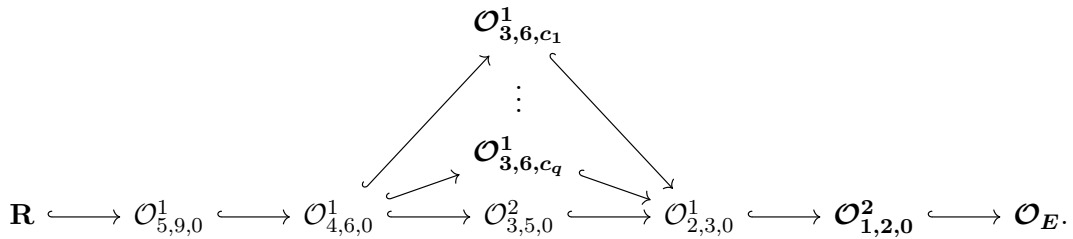
\begin{center}
\begin{tikzcd}[row sep=0.1em, column sep=2.5em]
        & & & \boldsymbol{\Mfo^1_{3,6,c_1}} \arrow[dddr, hook, crossing over] & & & \\
        & & & \vdots & & & \\
        & & & \boldsymbol{\Mfo^1_{3,6,c_q}} \arrow[dr, hook, crossing over] & & & \\
        \mathbf{R} \arrow[r, hook] & \Mfo^1_{5,9,0} \arrow[r, hook] & \Mfo^1_{4,6,0} \arrow[uuur, hook] \arrow[ur, hook] \arrow[r, hook] & \Mfo^2_{3,5,0} \arrow[r, hook] & \Mfo^1_{2,3,0} \arrow[r, hook] & \boldsymbol{\Mfo^2_{1,2,0}} \arrow[r, hook] & \boldsymbol{\Mfo_E}.
    \end{tikzcd}
\end{center}

When $R = \mathfrak{o}[\pi^3 \pi_E]$ so that $S(R) = 9$ and $f(R) = 18$, we have the following diagram:

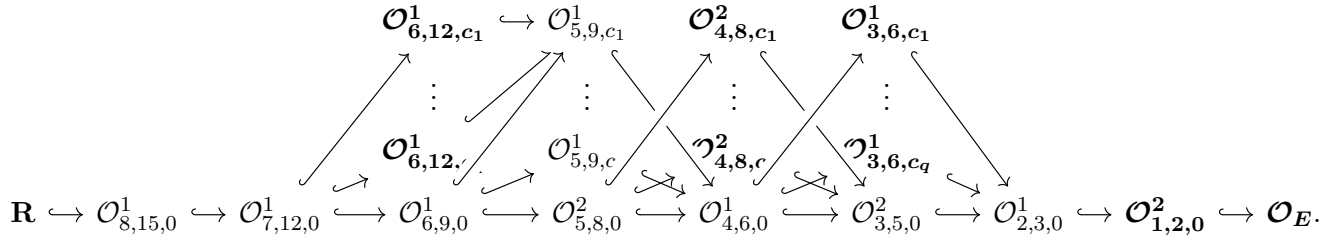
\begin{center}
\begin{tikzcd}[row sep=0.1em, column sep=1.2em]
    & & & \boldsymbol{\Mfo^1_{6,12,c_1}} \arrow[r, hook] & \Mfo^1_{5,9,c_1} \arrow[dddr, hook, crossing over] & \boldsymbol{\Mfo^2_{4,8,c_1}} \arrow[dddr, hook, crossing over] & \boldsymbol{\Mfo^1_{3,6,c_1}} \arrow[dddr, hook, crossing over] & & & \\        
    & & & \vdots & \vdots & \vdots & \vdots & & & \\
    & & & \boldsymbol{\Mfo^1_{6,12,c_q}} \arrow[uur, hook, crossing over] & \Mfo^1_{5,9,c_q} \arrow[dr, hook, crossing over] & \boldsymbol{\Mfo^2_{4,8,c_q}} \arrow[dr, hook, crossing over] & \boldsymbol{\Mfo^1_{3,6,c_q}} \arrow[dr, hook, crossing over] & & & \\
    \mathbf{R} \arrow[r, hook] & \Mfo^1_{8,15,0} \arrow[r, hook] & \Mfo^1_{7,12,0} \arrow[uuur, hook] \arrow[ur, hook] \arrow[r, hook] & \Mfo^1_{6,9,0} \arrow[uuur, hook, crossing over] \arrow[ur, hook] \arrow[r, hook] & \Mfo^2_{5,8,0} \arrow[uuur, hook, crossing over] \arrow[ur, hook] \arrow[r, hook] & \Mfo^1_{4,6,0} \arrow[uuur, hook, crossing over] \arrow[ur, hook] \arrow[r, hook] & \Mfo^2_{3,5,0} \arrow[r, hook] & \Mfo^1_{2,3,0} \arrow[r, hook] & \boldsymbol{\Mfo^2_{1,2,0}} \arrow[r, hook] & \boldsymbol{\Mfo_E}.
\end{tikzcd}
\end{center}

Here $\mathbf{\Mfo^1_{6,12,c_1}} = \mathfrak{o}[\pi^2 \pi_E]$ and thus the diagram for $\mathfrak{o}[\pi^2 \pi_E]$ is a part of that for $\mathfrak{o}[\pi^3 \pi_E]$.

\textbf{2. Case $\mathfrak{o}[\pi^2 \pi^2_E]$ and $\mathfrak{o}[\pi^3 \pi^2_E]$.}
When $R = \mathfrak{o}[\pi^2 \pi^2_E]$ with $S(R) = 7$ and $f(R) = 14$, we have the following diagram:

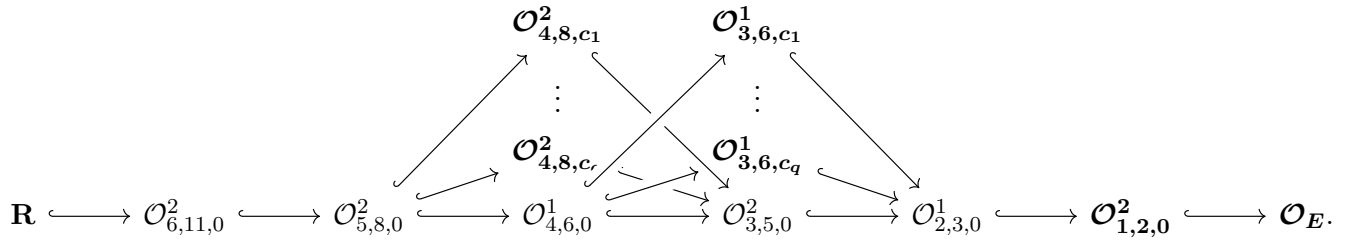
\begin{center}
\begin{tikzcd}[row sep=0.1em, column sep=2.8em]
    & & & \boldsymbol{\Mfo^2_{4,8,c_1}} \arrow[dddr, hook, crossing over] & \boldsymbol{\Mfo^1_{3,6,c_1}} \arrow[dddr, hook, crossing over] & & & \\
    & & & \vdots & \vdots & & & \\
    & & & \boldsymbol{\Mfo^2_{4,8,c_q}} \arrow[dr, hook, crossing over] & \boldsymbol{\Mfo^1_{3,6,c_q}} \arrow[dr, hook, crossing over] & & & \\
    \mathbf{R} \arrow[r, hook] & \Mfo^2_{6,11,0} \arrow[r, hook] & \Mfo^2_{5,8,0} \arrow[uuur, hook, crossing over] \arrow[ur, hook, crossing over] \arrow[r, hook] & \Mfo^1_{4,6,0} \arrow[uuur, hook, crossing over] \arrow[ur, hook, crossing over] \arrow[r, hook] & \Mfo^2_{3,5,0} \arrow[r, hook] & \Mfo^1_{2,3,0} \arrow[r, hook] & \boldsymbol{\Mfo^2_{1,2,0}} \arrow[r, hook] & \boldsymbol{\Mfo_E}.
\end{tikzcd}
\end{center}

When $R = \mathfrak{o}[\pi^3 \pi^2_E]$ with $S(R) = 10$ and $f(R) = 20$, we have the following diagram:

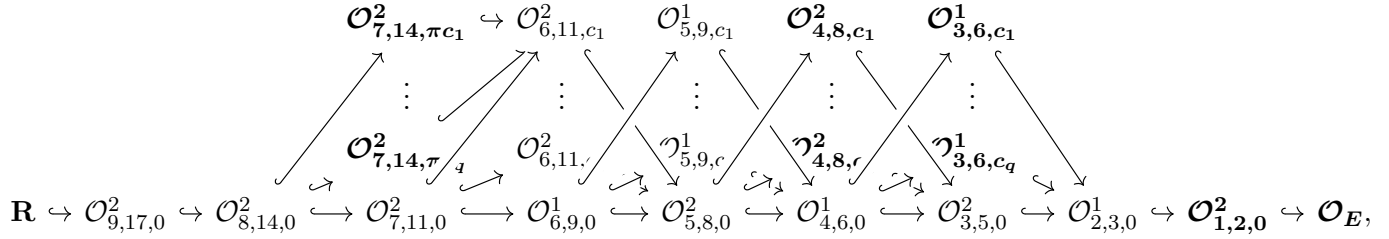
\begin{center}
\begin{tikzcd}[row sep=0.1em, column sep=0.8em]
    & & & \boldsymbol{\Mfo^2_{7,14,\pi c_1}} \arrow[r, hook] & \Mfo^2_{6,11,c_1} \arrow[dddr, hook, crossing over] & \Mfo^1_{5,9,c_1} \arrow[dddr, hook, crossing over] & \boldsymbol{\Mfo^2_{4,8,c_1}} \arrow[dddr, hook, crossing over] & \boldsymbol{\Mfo^1_{3,6,c_1}} \arrow[dddr, hook, crossing over] & & & \\
    & & & \vdots & \vdots & \vdots & \vdots & \vdots & & & \\
    & & & \boldsymbol{\Mfo^2_{7,14,\pi c_q}} \arrow[uur, hook, crossing over] & \Mfo^2_{6,11,c_q} \arrow[dr, hook, crossing over] & \Mfo^1_{5,9,c_q} \arrow[dr, hook, crossing over] & \boldsymbol{\Mfo^2_{4,8,c_q}} \arrow[dr, hook, crossing over] & \boldsymbol{\Mfo^1_{3,6,c_q}} \arrow[dr, hook, crossing over] & & & \\
    \mathbf{R} \arrow[r, hook] & \Mfo^2_{9,17,0} \arrow[r, hook] & \Mfo^2_{8,14,0} \arrow[uuur, hook, crossing over] \arrow[ur, hook, crossing over] \arrow[r, hook] & \Mfo^2_{7,11,0} \arrow[uuur, hook, crossing over] \arrow[ur, hook, crossing over] \arrow[r, hook] & \Mfo^1_{6,9,0} \arrow[uuur, hook, crossing over] \arrow[ur, hook, crossing over] \arrow[r, hook] & \Mfo^2_{5,8,0} \arrow[uuur, hook, crossing over] \arrow[ur, hook, crossing over] \arrow[r, hook] & \Mfo^1_{4,6,0} \arrow[uuur, hook, crossing over] \arrow[ur, hook, crossing over] \arrow[r, hook] & \Mfo^2_{3,5,0} \arrow[r, hook] & \Mfo^1_{2,3,0} \arrow[r, hook] & \boldsymbol{\Mfo^2_{1,2,0}} \arrow[r, hook] & \boldsymbol{\Mfo_E},
\end{tikzcd}
\end{center}
where $\{\pi c_1(=0), \cdots, \pi c_q\}$ is the set of representatives for $\mathfrak{o}/\pi^2\mathfrak{o}$ with valuation $\geq 1$, which is used for $\boldsymbol{\Mfo^2_{7,14,\pi c_i}}$'s. 
Here $\boldsymbol{\Mfo^2_{7,14,0}} = \mathfrak{o}[\pi^2 \pi^2_E]$ so that the diagram for $\mathfrak{o}[\pi^2 \pi^2_E]$ is a part of that for $\mathfrak{o}[\pi^3 \pi^2_E]$.
\end{example}

\section{Local theory: split case}\label{sec:split}
In this section, we treat the case where $E/F$ is not a field extension.
The splitting type of $E$ is defined as follows:
\begin{equation}\label{eq:splittingtype}
\left\{
\begin{array}{l l}
(1\ 2) &\text{if $E\cong F\times E'$ with $E'/F$ an unramified quadratic field extension,  }\\
(1\ 1^2) &\text{if $E\cong F\times E'$ with $E'/F$ a ramified quadratic field extension, }\\  
(1\ 1\ 1) &\text{if $E\cong F\times F \times F$.}
\end{array}
\right.
\end{equation}

\subsection{The case when $E$ has a splitting type $(1\ 2)$}
We identify an order $\Mfo$ of $E$ with an order of $F\times E'$ under the isomorphism $E\cong F\times E'$.
Note that the maximal order $\Mfo_E$ of $E$ is identified with $\mfo\times \Mfo_{E'}$.  We choose $x\in \Mfo_E^\times$ such that $\Mfo_{E'}=\mfo[x]$. Let $X^2+c_1X+c_0$ be the minimal polynomial for $x$ over $\mfo$. Here $\overline{X^2+c_1X+c_0}$ is irreducible over $\kappa$.

To maintain consistency with the previous section (cf. (\ref{eq:bracketnotation})), we introduce the following notation:
\begin{equation}\label{eq:index_of_orders_splitcase}
(\pi^a, \pi^b x+c)_\mfo:=\mfo\langle (1,1), (0,\pi^a), (0, \pi^bx+c) \rangle \subset F\times E' ~~~~  \textit{ with $a,b\in \mathbb{Z}_{\geq 0}$ and $c\in \mfo$,  as an $\mfo$-module}.\end{equation}

We denote $(\pi^a, \pi^b x+c)_\mfo$  by $\Mfo_{a,b,c}$. Here we emphasize $a,b\geq 0$. 



\begin{proposition}[Analogues of Lemma \ref{lem:overorder} and Proposition \ref{prop:overorderform}]\label{prop:splitunram}
Let $c\in \mfo$.
\begin{enumerate}
    \item Any order is of the form $\Mfo_{a,b,c}$ where $a\leq 2 \min(b,\ord(c))$.
    Conversely, an $\mfo$-module $\Mfo_{a,b,c}$ with $a\leq 2 \min(b,\ord(c))$ is an order. 

    \item 
    An order $\Mfo_{a',b',c'}$ contains an order $\Mfo_{a,b,c}$ if and only if $a'\leq a,$ $b'\leq b$, and 
    \[
    \left\{
    \begin{array}{l l}
       c' \in (\pi^{\max(\lceil\frac{a'}{2}\rceil,a'-(b-b'))})  &\textit{if $\ord(c)\geq a'$};  \\
       c'-\pi^{-(b-b')}c\in (\pi^{a'-(b-b')})  & \textit{if $\ord(c)<a'$.}
    \end{array}
    \right.
    \] Here, in the second case, $\lceil \frac{a'}{2}\rceil\leq \ord(c)-(b-b')=\ord(c')$ (thus $a'-(b-b')>0$).

    \item $\Mfo_{a',b',c'}=\Mfo_{a,b,c}$ as an order if and only if $a=a'$, $b=b'$, and $c-c'\in (\pi^a)$. 

    \item
    The Serre invariant of $\Mfo_{a,b,c}$ is $S(\Mfo_{a,b,c}) = a+b$. The conductor of $\Mfo_{a,b,c}$ is given by $\mathfrak{f}(\Mfo_{a,b,c}) = (\pi^a) \times (\pi^{\max(a, b, a+b-\ord(c))})$ as an ideal of $\Mfo_E \cong \mfo \times \Mfo_{E'}$. 
    
    \item
    An order $\Mfo_{a,b,c}$ is Gorenstein if and only if 
    \[
       \left\{
       \begin{array}{l l}
        a=2b \textit{ or }a=0  & \textit{if }\ord(c)\geq a;\\
           a=2\min(b,\ord(c)) & \textit{if }\ord(c)<a.
       \end{array}\right.
       \]
\end{enumerate}
\end{proposition}
\begin{proof}
  \begin{enumerate}
       \item
       First, we claim that any order $\Mfo$ has a basis of the form $\{(1,1),(0,\pi^a),(0,c+\pi^bx)\}$.
       Let $\{(1,1), (v_1, v_2+v_3x), (w_1,w_2+w_3x)\}$ be a basis for $\Mfo$ as a $\mfo$-module.
       Then, $\{(1,1), (0,(v_2-v_1)+v_3x), (0,(w_2-w_1)+w_3x)\}$ also forms a basis for $\Mfo$, so that we may and do assume that $v_1=w_1=0$.
       Without loss of generality, we assume that $\ord(v_3)\geq \ord(w_3)$, and hence $\frac{v_3}{w_3}\in \mfo$.
       By replacing $(0,v_2+v_3x)$ with $(0,v_2+v_3x-\frac{v_3}{w_3}(w_2+w_3x))=(0,v_2-\frac{v_3}{w_3}w_2)$, we have a basis $\{(1,1),(0,v_2-\frac{v_3}{w_3}w_2), (0,w_2+w_3x)\}$.
        We set $v_2-\frac{v_3}{w_3}w_2=u\pi^a$,  $w_3=u'\pi^b$, and $c=u'^{-1}w_2$ where $u,u'\in \mfo^\times$.
        Then, $\{(1,1),(0,\pi^a),(0,c+\pi^bx)\}=\{(1,1),u^{-1}(0,v_2-\frac{v_3}{w_3}w_2),u'^{-1}(0,w_2+w_3x)\}$ is the basis of the desired form.

 We next characterize the condition when $\Mfo=(\pi^a,\pi^bx+c)_\mfo$ is closed under multiplication. 
\[\textit{We have } ~~~~~~~~~~~ 
(0,\pi^bx+c)^2=(0,-c^2+\pi^bc_1c-\pi^{2b}c_0)+(2c-\pi^bc_1)\cdot(0,\pi^bx+c)
\textit{ in } F\times E'.
\]

      Here $c_1$ and $c_0$ are coefficients of the minimal polynomial for $x$, explained at the beginning of this section. 
        To prove the condition that $a\leq 2 \min(b,\ord(c))$, we claim that \[\ord(-c^2+\pi^bc_1c-\pi^{2b}c_0)=2\min(b,\ord(c)).\]
           Since $\ord(c_0)=0$, the equality fails only when $\ord(c)=b$ and $\overline{\frac{c}{\pi^b}}$ is a root of $\overline{X^2-c_1X+c_0}$.
        This contradicts the irreducibility of $\overline{X^2+c_1X+c_0}$ over $\kappa$.
        Therefore, the equality follows, and hence $(0,\pi^bx+c)^2\in \Mfo$ if and only if $a\leq 2\min(b,\ord(c))$.
    

       \item We have $\pi^a=\pi^{a-a'}\cdot \pi^{a'}$ and $\pi^bx+c=\pi^{b-b'}(\pi^{b'}x+c')+c-\pi^{b-b'}c'.$
       Thus, $(0,\pi^a)$ and $(0, \pi^{b}x+c)$ are contained in $(\pi^{a'},\pi^{b'}+c')_\mfo$ if and only if 
       \[
                  a'\leq a, b'\leq b, \text{ and }c-\pi^{b-b'}c'\in (\pi^{a'}).
       \]

If $\ord(c)\geq a'$, then 
 $c-\pi^{b-b'}c'\in (\pi^{a'})$ if and only if $\pi^{b-b'}c'\in (\pi^{a'})$.
       Combining this condition with the ring condition $a'\leq 2\ord(c')$ from (1) yields $c'\in (\pi^{\max(\lceil\frac{a'}{2}\rceil,a'-(b-b'))})$.

   If $\ord(c)<a'$, then $c-\pi^{b-b'}c'\in (\pi^{a'})$ if and only if $c'-\pi^{-(b-b')}c\in \pi^{a'-(b-b')}\mfo$.
       For $c'$ satisfying this condition, we have $\ord(c')=\ord(c)-(b-b')$.
       Therefore, the ring condition $a'\leq 2\ord(c')$ from (1) is the same as the inequality $\lceil\frac{a'}{2}\rceil\leq \ord(c)-(b-b')$.
       
        \item This is a direct consequence of (2).
       \item 
       The Serre invariant is $S(\Mfo_{a,b,c}) = [\Mfo_E:\Mfo_{a,b,c}]_\mfo$. 
       Consider the $\mfo$-submodule of $\Mfo_E$ generated by the set $\{(1,1), (0,1),(0,\pi^bx)\}$, denoted by $\Mfo'$.
       Then $[\Mfo_E:\Mfo']=b$, whereas $[\Mfo':\Mfo_{a,b,c}]=a$ since $\Mfo'$ is also generated by $\{(1,1),(0,1),(0,\pi^bx+c)\}$.
       Thus, $S(\Mfo_{a,b,c})=a+b$.
       
       The conductor $\mathfrak{f}(\Mfo_{a,b,c})$ is the largest ideal of $\Mfo_E$ contained in $\Mfo_{a,b,c}$. Any ideal of $\Mfo_E \cong \mfo \times \Mfo_{E'}$ is of the form $(\pi^{f_1}) \times (\pi^{f_2})$. This ideal is contained in $\Mfo_{a,b,c}$ if and only if $(\pi^{f_1},0) \in \Mfo_{a,b,c}$, $(0,\pi^{f_2}) \in \Mfo_{a,b,c}$, and $(0,\pi^{f_2}x) \in \Mfo_{a,b,c}$.
       The condition $(\pi^{f_1},0) \in \Mfo_{a,b,c}$  (equivalently $(0,\pi^{f_1})\in \Mfo_{a,b,c}$ since $(1,1)\in \Mfo_{a,b,c}$) holds if and only if $f_1 \geq a$. 
       Thus, we have the minimum $f_1 = a$.
       Similarly, $(0,\pi^{f_2}) \in \Mfo_{a,b,c}$ requires $f_2 \geq a$. For $(0,\pi^{f_2}x)$ to be in $\Mfo_{a,b,c}$, it must be of the form $y_1(0,\pi^a) + y_2(0,\pi^bx+c)$ for some $y_1, y_2 \in \mfo$. Thus $y_2 \pi^b = \pi^{f_2}$, which implies $y_2 = \pi^{f_2-b}$ and hence $f_2 \geq b$. Then $y_1 \pi^a + \pi^{f_2-b}c = 0$, meaning $\pi^{f_2-b}c \in (\pi^a)$, which gives $f_2 - b + \ord(c) \geq a$. Thus $f_2 \geq \max(a,b,a+b-\ord(c))$.
       Therefore, the minimal $f_2$ is exactly $\max(a,b,a+b-\ord(c))$.

       \item 
       By Proposition \ref{prop:gorenstein_cond}, the order
       $\Mfo:=\Mfo_{a,b,c}$ is Gorenstein if and only if 
       \begin{equation}\label{eq:cond12case}
       [\Mfo_E:\mathfrak{f}(\Mfo)]_\mfo=2S(\Mfo).
       \end{equation}
       By (4), we have $[\Mfo_E:\mathfrak{f}(\Mfo)]_\mfo = a+2\max(a, b, a+b-\ord(c))$ since $E'/F$ is an unramified quadratic extension and the length of $\Mfo_{E'}/\pi^{f_2}\Mfo_{E'}$ is $2f_2$.
       Also $2S(\Mfo) = 2(a+b)$.
       Plugging these into Equation \eqref{eq:cond12case}, we obtain
        $\left\{
       \begin{array}{l l}
                  a+2\max(a,b)=2(a+b)  & \textit{if }\ord(c)\geq a;\\
           a+2\max(a,a+b-\ord(c))=2(a+b) & \textit{if }\ord(c)< a.
       \end{array}\right.$
This is exactly the same condition as the desired one.  \qedhere
   \end{enumerate}
\end{proof}

\begin{theorem}\label{thm:splitunroverorders}
Let $(\pi^{f_1}) \times (\pi^{f_2}) \left(\subset \mfo \times \Mfo_{E'}\right)$ be the conductor of an order $R$ in $\Mfo_E$. 
Then $\#\overline{\mathrm{Cl}}(R)$ is given by the following formula:
\begin{align*}
    \#\overline{\mathrm{Cl}}(R) =&~ 2 \sum_{0 \leq b' \leq S(R)-f_1} \left( \sum_{\substack{0 \leq a' \leq f_1 \\ a' \leq 2b' \\ a' \leq S(R)-f_2}} q^{\min \left(\lfloor \frac{a'}{2} \rfloor, S(R)-f_1-b' \right)} + \sum_{\substack{S(R)-f_2 < a' \leq f_1 \\ a' \leq 2b' \\ \lceil \frac{a'}{2} \rceil - b' \leq f_1-f_2}} q^{S(R)-f_1-b'} \right) \\
    &-(S(R)-f_1+1) - \sum_{\substack{0 < a' \leq f_1 \\ a' \leq S(R)-f_2 \\ 2 \mid a'}} q^{\min \left( \frac{a'}{2}, S(R)-f_1-\frac{a'}{2} \right)} - \sum_{\substack{S(R)-f_2 < a' \leq f_1 \\ 2 \mid a' \\ S(R)-f_1 \leq S(R)-f_2}} q^{S(R)-f_1-\frac{a'}{2}} \\
    &- \sum_{0 < b' \leq S(R)-f_1} \sum_{\substack{0 < a' \leq f_1 \\ a' < 2b' \\ 2 \mid a' \\ a' \leq S(R)-f_2 \\ \frac{a'}{2} \leq S(R)-f_1-b'}} \left( q^{\frac{a'}{2}} - q^{\frac{a'}{2}-1} \right) - \sum_{\substack{S(R)-f_2 < a' \leq f_1 \\ 2 \mid a' \\ S(R)-f_2 < S(R)-f_1}} q^{S(R)-f_2-\frac{a'}{2}}.
\end{align*}
Note that this formula relies entirely on the invariant quantities $f_1, f_2,$ and $S(R)$. For an explicit closed-form evaluation of this multi-summation as a rational function in $q$, see Appendix \ref{app:formula}.
\end{theorem}

\begin{proof}
An overorder of $R(=\Mfo_{a,b,c})$ is of the form $\Mfo_{a',b',c'}$ with $a'\leq 2 \min(b',\ord(c'))$ satisfying the conditions of Proposition \ref{prop:splitunram}.(2).
Our strategy is to count the number of $\Mfo_{a',b',c'}$'s with fixed $a', b'$ and that of Gorenstein overorders. Then we use Proposition \ref{prop:orderidealcounting} to compute $\#\clb(R)$.

\textbf{1. Case $a'\leq \ord(c)$.}
In this case, $c'\in (\pi^{\max(\lceil\frac{a'}{2}\rceil,a'-(b-b'))})$, so that there exist $q^{\min(\lfloor\frac{a'}{2}\rfloor, b-b')}$ choices of $c'$ up to $(\pi^{a'})$ for fixed $a'(\leq a)$ and $b'(\leq b)$ such that $0\leq a'\leq 2b'$.
In the following, we compute the number of Gorenstein overorders among these overorders.

\begin{itemize}
    \item If $a'=0$, then $c'$ is unique up to $(\pi^{a'})$ and  $\Mfo_{a',b',c'}$ is Gorenstein by Proposition \ref{prop:splitunram} for each $0\leq b'\leq b$. 
    Thus, the number of Gorenstein overorders when $a'=0$ is $b+1$.
    \item If $a'=2b'$ and $a'>0$, then $c'\in (\pi^{b'})$ by Proposition \ref{prop:splitunram}.(2). Thus $\ord(c')\geq b'$ and $\Mfo_{a',b',c'}$ is Gorenstein by Proposition \ref{prop:splitunram}.(5). 
    Therefore, for fixed $a'$ and $b'$ in this case, there are $q^{\min(b',b-b')}$ Gorenstein overorders.
    \item If $0<a'<2b'$, then $\Mfo_{a',b',c'}$ is Gorenstein if and only if $a'=2\ord(c')$ by Proposition \ref{prop:splitunram}.(5).
    It happens only if $2|a'$ and $\frac{a'}{2}\geq a'-(b-b')$ since $c'\in (\pi^{\max(\lceil\frac{a'}{2}\rceil,a'-(b-b'))})$ and $\ord(c')=\frac{a'}{2}$.
    For fixed $a'$ and $b'$ satisfying these conditions, the number of Gorenstein overorders is $q^{\frac{a'}{2}}-q^{\frac{a'}{2}-1}$.
\end{itemize}

\textbf{2. Case $a'>\ord(c)$.}
In this case, $c'\in\pi^{-(b-b')}c+ (\pi^{a'-(b-b')})$ where $\lceil\frac{a'}{2}\rceil\leq \ord(c')=\ord(c)-(b-b')$ by Proposition \ref{prop:splitunram}.(2), so that there exist $q^{b-b'}$ choices of $c'$ up to $(\pi^{a'})$.

We now compute the number of Gorenstein overorders.
Since $\ord(c')=\ord(c)-(b-b')$, $\Mfo_{a',b',c'}$ is Gorenstein if and only if $a'=2\min(b',\ord(c)-(b-b'))$ by Proposition \ref{prop:splitunram}.(5).
\begin{itemize}
    \item If $a'=2b'$, then $\Mfo_{a',b',c'}$ is Gorenstein if and only if $b'\leq \ord(c) -(b-b')$, equivalently $b\leq \ord(c)$.
    \item If $0< a'<2b'$, then $\Mfo_{a',b',c'}$ is Gorenstein if and only if $a'=2\ord(c)-2(b-b')$ and $b> \ord(c)$.
\end{itemize}
In both cases, for fixed $a'$ and $b'$, there are $q^{b-b'}$ Gorenstein overorders.

Summing up these two cases using the equation in Proposition \ref{prop:orderidealcounting} yields the intermediate formula expressed purely in terms of the initial local parameters $a, b,$ and $c$:
\[
\begin{aligned}
&\#\overline{\mathrm{Cl}}(R) = 2 \sum_{0 \leq b' \leq b} \left( \sum_{\substack{0 \leq a' \leq a \\ a' \leq 2b' \\ a' \leq \ord(c)}} q^{\min \left(\lfloor \frac{a'}{2} \rfloor, b-b' \right)} + \sum_{\substack{\ord(c) < a' \leq a \\ a' \leq 2b' \\ \lceil \frac{a'}{2} \rceil - b' \leq \ord(c) - b}} q^{b-b'} \right) -(b+1) - \\
&\sum_{\substack{0 < a' \leq a \\ a' \leq \ord(c) \\ 2 \mid a'}} q^{\min \left( \frac{a'}{2}, b-\frac{a'}{2} \right)} - \sum_{\substack{\ord(c) < a' \leq a \\ 2 \mid a' \\ b \leq \ord(c)}} q^{b-\frac{a'}{2}}- \sum_{0 < b' \leq b} \sum_{\substack{0 < a' \leq a \\ a' < 2b' \\ 2 \mid a' \\ a' \leq \ord(c) \\ \frac{a'}{2} \leq b-b'}} \left( q^{\frac{a'}{2}} - q^{\frac{a'}{2}-1} \right) - \sum_{\substack{\ord(c) < a' \leq a \\ 2 \mid a' \\ \ord(c) < b}} q^{\ord(c)-\frac{a'}{2}}.   
\end{aligned}
\]
To express this formula intrinsically in terms of the invariant quantities $f_1, f_2$, and $S(R)$, we recall that $S(R) = a+b$ and $\mathfrak{f}(R) = (\pi^a) \times (\pi^{\max(a, b, a+b-\ord(c))})$ by Proposition \ref{prop:splitunram}.(4). 
The relationship between the local parameters and the invariants is given by:
\[
    f_1 = a, \qquad b = S(R) - f_1, \qquad \min(a, b, \ord(c)) = S(R) - f_2.
\]
By Proposition \ref{prop:splitunram}.(3), the order $\Mfo_{a,b,c}$ is uniquely determined by the reduction of $c$ modulo $(\pi^a)$. Thus, $\Mfo_{a,b,c}$ and $\#\clb(\Mfo_{a,b,c})$ do not depend on the exact value of $\ord(c)$ when $\ord(c) \ge a$. 
We now claim that replacing $\ord(c)$ with $\min(a,b,\ord(c))$ in this equation  preserves the total sum.

\begin{itemize}
    \item 
First, $\ord(c)$ appears as an exponent only in the final summation, which is non-empty only if $\ord(c) < a$ and $\ord(c) < b$. In this region, $\ord(c) = \min(a,b,\ord(c))$, thus this term is unaffected. 

\item Second, if $\ord(c) \ge a$, then the condition $a' \le \ord(c)$ holds for all $a' \le a$, and the summations over $a' > \ord(c)$ are empty. Replacing $\ord(c)$ with $a$ thus leaves the evaluation invariant. 
\item Finally, suppose that $b < \ord(c) < a$. Replacing $\ord(c)$ with $b$ transfers terms with $b < a' \le \ord(c)$ from the first positive summation to the second. For these terms, $a' > b$ and $a' \le 2b'$ ensure $\lfloor a'/2 \rfloor \ge b-b'$, simplifying the summand $q^{\min(\lfloor a'/2 \rfloor, b-b')}$ to $q^{b-b'}$, which matches the second sum. An identical transfer occurs for the first two negative sums, while the rest are unaffected as their inequalities force $a' < b$ or yield empty sums. 
\end{itemize}
Thus, replacing $\ord(c)$ with $\min(a,b,\ord(c))$ leaves the sum invariant. Substituting $a=f_1$, $b=S(R)-f_1$, and $\min(a,b,\ord(c))=S(R)-f_2$ yields the stated invariant formula.
\end{proof}

\begin{corollary} \label{cor:gorenstein_calculation12}
If the order $R$ in Theorem \ref{thm:splitunroverorders} is Gorenstein, then $S(R) = f_1/2 + f_2$ and $\#\overline{\mathrm{Cl}}(R)$ is entirely determined by the conductor components $f_1$ and $f_2$ (here, $f_2\geq f_1$), as follows:
\[
\#\overline{\mathrm{Cl}}(R) = \frac{q^{\lfloor f_1/4 \rfloor} \big( (f_2 - f_1) Q_{f_1/2}(q) + P_{f_1/2}(q) \big) - 4(f_2+3)q + 4(f_2-1)}{(q-1)^2},
\]
where the polynomials $P_m(q)$ and $Q_m(q)$ depend only on the parity of $m = f_1/2$:
\[ P_m(q) = \begin{cases} q^2+10q+5 & \text{if } m \text{ is even,} \\ 5q^2+10q+1 & \text{if } m \text{ is odd;} \end{cases} ~~~~~ \textit{ and } ~~~~~ 
    Q_m(q) = \begin{cases} q^2+2q-3 & \text{if } m \text{ is even,} \\ 3q^2-2q-1 & \text{if } m \text{ is odd.} \end{cases}
\]
\end{corollary}

\begin{proof}
Since $R$ is Gorenstein if and only if         $[\Mfo_E:\mff(R)]_\mfo=2S(R)$ by Equation \eqref{eq:cond12case}, we have $S(R) = f_1/2 + f_2$. Since $f_1\leq 2(S(R)-f_1)$ by the proof of Theorem \ref{thm:splitunroverorders}, we have $f_1\leq f_2$.
Substituting this relation into Theorem \ref{thm:splitunroverorders} expresses all bounds in terms of $f_1$ and $f_2$. Evaluating the summations by the parity of $f_1/2$ yields the stated formula.
\end{proof}
   

\subsection{The case when $E$ has a splitting type $(1\ 1^2)$}
In this case, $E$ is an \'etale algebra such that $E\cong F\times E'$ where $E'$ is a ramified quadratic extension over $F$.
We identify $\Mfo_{E}\cong \mfo\times \Mfo_{E'}$.
We choose a  uniformizer $x$ of $E'$ so that $\Mfo_{E'}=\mfo[x]$.
Let $X^2+c_1X+c_0$ be the minimal polynomial for $x$ over $\mfo$.
Note that $\ord(c_1)>0$ and that $\ord(c_0)=1$.
We continue to use the notation $\Mfo_{a,b,c}$ to denote the $\mfo$-module $(\pi^a,\pi^bx+c)_\mfo$ defined in (\ref{eq:index_of_orders_splitcase}). Here we emphasize $a,b\geq 0$.

\begin{proposition}\label{prop:splitram}(Analogue of Proposition \ref{prop:splitunram})
Let $c\in\mfo$.
    \begin{enumerate}
        \item Any order is of the form $\Mfo_{a,b,c}$ where $a\leq \min(2b+1,2\ord(c))$.
        Conversely, an $\mfo$-module $\Mfo_{a,b,c}$ with  $a\leq \min(2b+1,2\ord(c))$ is an order.
    \item An order $\Mfo_{a',b',c'}$ contains an order $\Mfo_{a,b,c}$ if and only if $a'\leq a,$ $b'\leq b$, and 
    \[
    \left\{
    \begin{array}{l l}
       c' \in (\pi^{\max(\lceil\frac{a'}{2}\rceil,a'-(b-b'))})  &\textit{if $\ord(c)\geq a'$};  \\
       c'-\pi^{-(b-b')}c\in (\pi^{a'-(b-b')})  & \textit{if $\ord(c)<a'$.}
    \end{array}
    \right.
    \] Here, in the second case, $\lceil \frac{a'}{2}\rceil\leq \ord(c)-(b-b')=\ord(c')$ (thus $a'-(b-b')>0$).
    \item $\Mfo_{a',b',c'}=\Mfo_{a,b,c}$ as an order if and only if $a=a'$, $b=b'$, and $c-c'\in (\pi^{a})$.
    \item 
    The Serre invariant of $\Mfo_{a,b,c}$ is $S(\Mfo_{a,b,c}) = a+b$. The conductor of $\Mfo_{a,b,c}$ is given by $\mathfrak{f}(\Mfo_{a,b,c}) = (\pi^a) \times (x^{\max(2a, 2b+1, 2a+2b-2\ord(c)+1)-1})$ as an ideal of $\Mfo_E \cong \mfo \times \mfo[x]$.
    \item
    An order $\Mfo_{a,b,c}$ is Gorenstein if and only if 
       \[
       \left\{
       \begin{array}{l l}
       a=2b+1 \textit{ or }a=0 & \textit{if }\ord(c)\geq a;\\
           a=\min(2b+1,2\ord(c))&\textit{if }\ord(c)<a.
       \end{array}\right.
       \] 
    \end{enumerate}
\end{proposition}
\begin{proof}
For (1),  by the same argument as in the proof of Proposition \ref{prop:splitunram}.(1), any order $\Mfo$ admits a basis of the form $\{(1,1), (0,\pi^a), (0,c+\pi^bx)\}$ as an $\mfo$-module.
 In order to characterize the condition when $\Mfo=(\pi^a,\pi^bx+c)_\mfo$ is closed under multiplication, we observe that 
        \[
     (0,\pi^bx+c)^2=(0,-c^2+\pi^bc_1c-\pi^{2b}c_0)+(2c-\pi^bc_1)\cdot(0,\pi^bx+c) \textit{ in }F\times E'.
       \]
       Here, $c_1$ and $c_0$ are coefficients of the minimal polynomial for $x$, explained at the beginning of this section.
      Thus it suffices to show that 
       $\ord(-c^2+\pi^bc_1c-\pi^{2b}c_0)=\min(2b+1,2\ord(c))$.
       Here we have \[\ord(c^2)=2\ord(c),\ \ord(\pi^b c_1c)\geq b+\ord(c)+1, \textit{ and }\ord(\pi^{2b}c_0)=2b+1.\] 
       The claim follows from two facts that $b+\ord(c)+1> \min (2\ord(c),2b+1)$ and $2\ord(c)\neq 2b+1$.  
       
        The proofs of (2)-(3) are identical to that of Proposition \ref{prop:splitunram}.(2)-(3) and thus we omit. For (4), the proof of the Serre invariant $S(\Mfo_{a,b,c})=a+b$ is identical to that of Proposition \ref{prop:splitunram}.(4).

%

  Next, we determine the conductor $\mathfrak{f}(\Mfo_{a,b,c})$, the largest ideal $(\pi^{f_1}) \times (x^{f_2})$ of $\Mfo_E \cong \mfo \times \Mfo_{E'}$ contained in $\Mfo_{a,b,c}$. 
 This ideal is contained in $\Mfo_{a,b,c}$ if and only if the elements $(\pi^{f_1},0)$, $(0,x^{f_2})$, and $(0,x^{f_2+1})$ belong to $\Mfo_{a,b,c}$.
 Since $(w_1, w_1) \in \Mfo_{a,b,c}$ for any $w_1 \in \mfo$, an element $(w_1, w_2) \in \Mfo_E$ belongs to $\Mfo_{a,b,c}$ if and only if $(0, w_2 - w_1) \in \Mfo_{a,b,c}$. 
        By the definition of $\Mfo_{a,b,c}$(:=$(\pi^a, \pi^b x+c)_\mfo$, cf. \eqref{eq:index_of_orders_splitcase}), this is equivalent to $w_2 - w_1 \in \mathfrak{o}\langle \pi^a, \pi^b x + c \rangle \subset \Mfo_{E'}$. 
        
        For $(\pi^{f_1}, 0) \in \Mfo_{a,b,c}$,  we must have $f_1 \ge a$. 
        To determine the minimal $f_2$ such that $0 \times x^{f_2}\Mfo_{E'} \subset \Mfo_{a,b,c}$, we let $k_1$ be the smallest integer such that $\pi^{k_1}x \in \mathfrak{o}\langle \pi^a, \pi^b x + c \rangle$ and $k_2$ be the smallest integer such that $\pi^{k_2} \in \mathfrak{o}\langle \pi^a, \pi^b x + c \rangle$. 
    We claim that $f_2 = \max(2k_1+1, 2k_2)-1$.

    First, suppose $k_1 \ge k_2$. 
    In this case,  $\pi^{k_1}$ (by the minimality of $k_2$) and $\pi^{k_1}x$ (by the definition of $k_1$) are in $\mathfrak{o}\langle \pi^a, \pi^b x + c \rangle$.
    Since these two elements generate $\pi^{k_1}\Mfo_{E'}$ as an $\mathfrak{o}$-module, we have $x^{2k_1}\Mfo_{E'} = \pi^{k_1} \Mfo_{E'} \subset \mathfrak{o}\langle \pi^a, \pi^b x + c \rangle$.
    Moreover, if $x^{2k_1-1}\Mfo_{E'} \subset \mathfrak{o}\langle \pi^a, \pi^b x + c \rangle$, then $\pi^{k_1-1}x \in \mathfrak{o}\langle \pi^a, \pi^b x + c \rangle$ as $\pi^{k_1-1}x \in x^{2k_1-1}\Mfo_{E'}$, which contradicts the minimality of $k_1$. Thus $f_2 = 2k_1= \max(2k_1+1, 2k_2)-1$. 
    
        Second, suppose $k_2 > k_1$. In this case, $\pi^{k_2-1}x$ (by the minimality of $k_1$) and $\pi^{k_2}$ (by the definition of $k_2$) are in  $\mathfrak{o}\langle \pi^a, \pi^b x + c \rangle$. Since these elements generate $\pi^{k_2-1}x\Mfo_{E'}$ as an $\mathfrak{o}$-module, we have $x^{2k_2-1}\Mfo_{E'}=\pi^{k_2-1}x\Mfo_{E'} \subset \mathfrak{o}\langle \pi^a, \pi^b x + c \rangle$.
        Moreover, if $x^{2k_2-2}\Mfo_{E'} \subset \mathfrak{o}\langle \pi^a, \pi^b x + c \rangle$, then $\pi^{k_2-1} \in \mathfrak{o}\langle \pi^a, \pi^b x + c \rangle$, which contradicts the minimality of $k_2$. Thus $f_2 = 2k_2-1= \max(2k_1+1, 2k_2)-1$.

       It remains to compute $k_1$ and $k_2$. Since $\pi^kx = \pi^{k-b}(\pi^bx+c) - \pi^{k-b}c$, it follows that $\pi^kx \in \mathfrak{o}\langle \pi^a, \pi^b x + c \rangle$ if and only if $k \ge b$ and $\ord(\pi^{k-b}c) \ge a$, which yields $k_1 = \max(b, a+b-\ord(c))$. Similarly, $\pi^k \in \mathfrak{o}\langle \pi^a, \pi^b x + c \rangle$ if and only if $k \ge a$, so that $k_2 = a$. Substituting these into the formula for $f_2$ yields $f_2 = \max(2a, 2b+1, 2a+2b-2\ord(c)+1)-1$, as desired.     

      For (5), by Proposition \ref{prop:gorenstein_cond}, the order $\Mfo:=\Mfo_{a,b,c}$ is Gorenstein if and only if $[\Mfo_E:\mff(\Mfo)]_\mfo=2S(\Mfo)$.
        By (4), we have $[\Mfo_E:\mff(\Mfo)]_\mfo = a + \max(2a, 2b+1, 2a+2b-2\ord(c)+1)-1$, and $2S(\Mfo) = 2(a+b)$.
        Plugging these into the equation, we obtain 
        \begin{align*}
        \left\{
       \begin{array}{l l}
       a+\max(2a,2b+1)-1=2(a+b)  & \textit{if }\ord(c)\geq a;\\
           a+\max(2a,2a+2b-2\ord(c)+1)-1=2(a+b) & \textit{if }\ord(c)<a.
       \end{array}\right.
       \end{align*}
       This is exactly the same condition as the desired one. 
\end{proof}
    

\begin{theorem}\label{thm:splitramorders}
Let $(\pi^{f_1}) \times (x^{f_2}) \left(\subset \mfo \times \Mfo_{E'}\right)$ be the conductor of an order $R$ in $\Mfo_E$. 
Letting $b = S(R) - f_1$ and $C = \lceil\frac{2S(R) - f_2}{2}\rceil$, $\#\overline{\mathrm{Cl}}(R)$ is given by the following formula:
\[\begin{aligned}
    &\#\overline{\mathrm{Cl}}(R) = 2 \sum_{0 \le b' \le b} \left( \sum_{\substack{0 \le a' \le f_1 \\ a' \le 2b'+1 \\ a' \le C}} q^{\min \left(\lfloor \frac{a'}{2} \rfloor, b-b' \right)} + \sum_{\substack{C < a' \le f_1 \\ a' \le 2b'+1 \\ \lceil \frac{a'}{2} \rceil - b' \le C - b}} q^{b-b'} \right) -(b+1) -   
    \\ & \sum_{\substack{0 < a' \le f_1 \\ a' \le C \\ 2 \mid (a'-1)}} q^{\min \left( \frac{a'-1}{2}, b-\frac{a'-1}{2} \right)} -\sum_{\substack{C < a' \le f_1 \\ 2 \mid (a'-1) \\ b+1 \le C}} q^{b-\frac{a'-1}{2}}- \sum_{0 < b' \le b} \sum_{\substack{0 < a' \le f_1 \\ a' < 2b'+1 \\ 2 \mid a' \\ a' \le C \\ \frac{a'}{2} \le b-b'}} \left( q^{\frac{a'}{2}} - q^{\frac{a'}{2}-1} \right) - \sum_{\substack{C < a' \le f_1 \\ 2 \mid a' \\ C \le b}} q^{C-\frac{a'}{2}}.
\end{aligned}
\]
Note that this formula relies entirely on the invariant quantities $f_1, f_2,$ and $S(R)$. For an explicit closed-form evaluation of this multi-summation as a rational function in $q$, see Appendix \ref{app:formula}.
\end{theorem}

\begin{proof}
    The proof is parallel to that of Theorem \ref{thm:splitunroverorders}.
An overorder of $R(=\Mfo_{a,b,c})$ is of the form $\Mfo_{a',b',c'}$ with $a' \le \min(2b'+1, 2\ord(c'))$ satisfying the conditions of Proposition \ref{prop:splitram}.(2). 
We enumerate the number of such overorders and Gorenstein overorders for fixed $a'(\leq a)$ and $b'(\leq b)$ such that $0\leq a'\leq 2b'+1$.

\textbf{1. Case $a' \le \ord(c)$.}
As in the split unramified case, there are $q^{\min(\lfloor a'/2 \rfloor, b-b')}$ choices for $c'$ up to $(\pi^{a'})$ for fixed $a'$ and $b'$. Among these, we count Gorenstein overorders using Proposition \ref{prop:splitram}.(5).
\begin{itemize}
    \item If $a'=0$, then $c'$ is unique up to $(\pi^{a'})$ and $\Mfo_{a',b',c'}$ is Gorenstein by Proposition \ref{prop:splitram}.(5) for each $0\leq b'\leq b$.
    Thus, the number of Gorestein overorders when $a'=0$ is $b+1$.
    \item If $a'=2b'+1 > 0$, then $\ord(c') \ge \lceil a'/2 \rceil = b'+1$, so that $\Mfo_{a',b',c'}$ is Gorenstein. There are $q^{\min(b', b-b')}$ such orders.
    \item If $0 < a' < 2b'+1$, then the order is Gorenstein if and only if $a' = 2\ord(c')$, which requires $2 \mid a'$ and $a'/2  \ge a'-(b-b')$. There are $q^{a'/2} - q^{a'/2-1}$ such orders.
\end{itemize}

\textbf{2. Case $a' > \ord(c)$.}
There are $q^{b-b'}$ choices for $c'$ up to $(\pi^{a'})$ for fixed $a'$ and $b'$, each with $\lceil \frac{a'}{2}\rceil\leq \ord(c') = \ord(c)-(b-b')$. The order is Gorenstein if and only if $a' = \min(2b'+1, 2\ord(c)-2(b-b'))$.
    \begin{itemize}
    \item 
    If $a'=2b'+1$, then $\Mfo_{a',b',c'}$ is Gorenstein if and only if $2b'+1\leq 2\ord(c) -2(b-b')$, which is rewritten as $2b+1\leq 2\ord(c)$, equivalently $b+1\leq \ord(c)$.
    \item
    If $0< a'<2b'+1$, then  $\Mfo_{a',b',c'}$ is Gorenstein if and only if $a'=2\ord(c)-2(b-b')$ and $2b+1> 2\ord(c)$. Here, $2b+1> 2\ord(c)$ if and only if $\ord(c)\leq b$.
    \end{itemize}
    In both cases, there are $q^{b-b'}$ Gorenstein overorders for fixed $a'$ and $b'$.

Summing up these two cases using the equation in Proposition \ref{prop:orderidealcounting} yields the intermediate formula expressed purely in terms of the initial local parameters $a, b,$ and $c$:
       \begin{align*}
    &\#\overline{\mathrm{Cl}}(R)=2 \sum_{0\leq b'\leq b }
    \left(\sum_{\substack{0\leq a'\leq a\\ a'\leq 2b'+1\\a'\leq \ord(c) }}
    q^{\min (\lfloor\frac{a'}{2} \rfloor,b-b')}+\sum_{\substack{\ord(c)<a'\leq a\\a'\leq 2b'+1\\ \lceil \frac{a'}{2}\rceil  \leq \ord(c)-(b-b')}}q^{b-b'}\right)
    -(b+1)\\
    &-\sum_{\substack{0< a' \leq a,\\a'\leq \ord(c) \\ 2|(a'-1)}}q^{\min(\frac{a'-1}{2},b-\frac{a'-1}{2})}-\sum_{\substack{\ord(c)<a'\leq a\\ 2|(a'-1)\\b+1\leq \ord(c)}}q^{b-\frac{a'-1}{2}}
    -\sum_{\substack{0<b'\leq b}}\sum_{{\substack{0< a' \leq a,\\  a'< 2b'+1\\ 2\mid a'\\a'\leq\ord(c)\\ \frac{a'}{2}\leq b-b'}}}(q^{\frac{a'}{2}}-q^{\frac{a'}{2}-1})
 -\sum_{\substack{\ord(c)<a'\leq a\\ 2|a'\\ \ord(c)\leq b}}q^{\ord(c)-\frac{a'}{2}}.
    \end{align*} To express this in terms of invariants, we note that $f_1 = a$, $b = S(R)-f_1$, and $\min(2a, 2b+1, 2\ord(c)) = 2S(R)-f_2$ by Proposition \ref{prop:splitram}.(4). 
We now claim that replacing $\ord(c)$ with $C:=\lceil(2S(R)-f_2)/2 \rceil$ in this equation preserves the sum.
\begin{itemize}
    \item First, $\ord(c)$ appears as an exponent only in the final sum, which is non-empty only if $\ord(c) \le b$ and $\ord(c) < a$. In this region, $2\ord(c) =  2S(R)-f_2$, thus this term is unaffected.
    
    \item Second, if $2\ord(c) \ge 2a$, then the summations over $a' > \ord(c)$ are empty, so replacing $\ord(c)$ with $a$ preserves the sum.
    \item Finally, if $2b+1 < 2\ord(c) < 2a$, then replacing $\ord(c)$ with $b+1$ transfers terms between the first and second sums. A similar calculation to the proof of Theorem \ref{thm:splitunroverorders} shows the summands match, ensuring the total sum remains invariant.
\end{itemize}
Substituting the invariant expressions into the above Equation yields the stated formula.
\end{proof}

\begin{corollary} \label{cor:gorenstein_calculation112}
Suppose the order $R = \Mfo_{a,b,c}$ in Theorem \ref{thm:splitramorders} is Gorenstein. Then $S(R) = (f_1 + f_2)/2$ and $\#\overline{\mathrm{Cl}}(R)$ is entirely determined by the conductor components $f_1$ and $f_2$, as follows:
\begin{enumerate}
    \item If $f_1$ is odd (so that $f_2 = 2f_1 - 1$), let $k = \lfloor f_1/4 \rfloor$. Then
   {\small
   \[
    \#\overline{\mathrm{Cl}}(R) = \begin{cases}
    \displaystyle \frac{2q^{k+2} + 12q^{k+1} + 2q^k - (16k+16)q + 16k}{(q-1)^2} & \text{if } f_1 = 4k+1, \\[12pt]
    \displaystyle \frac{8q^{k+2} + 8q^{k+1} - (16k+24)q + 16k+8}{(q-1)^2} & \text{if } f_1 = 4k+3.
    \end{cases}
    \]
     }%
    \item If $f_1$ is even (so that $f_2 \ge 2f_1$), let $k = \lfloor f_1/4 \rfloor$ and $b = (f_2 - f_1)/2$. Then
   {\footnotesize
   \[
    \#\overline{\mathrm{Cl}}(R) = \begin{cases}
    \displaystyle \frac{(b-2k+1)q^{k+2} + (2b-4k+12)q^{k+1} - (3b-6k-3)q^k - (4b+8k+14)q + 4b+8k-2}{(q-1)^2} & \text{if } f_1 = 4k, \\[12pt]
    \displaystyle \frac{(3b-6k+3)q^{k+2} - (2b-4k-12)q^{k+1} - (b-2k-1)q^k - (4b+8k+18)q + 4b+8k+2}{(q-1)^2} & \text{if } f_1 = 4k+2.
    \end{cases}
    \]
     }%
\end{enumerate}
Note that when $R = \Mfo_{a,b,c}$, we have $f_1 = a$ and $f_2 = \begin{cases} 2a - 1 & \text{if } f_1 \text{ is odd,} \\ a + 2b & \text{if } f_1 > 0 \text{ is even.} \end{cases}$
\end{corollary}


\begin{proof}
 Since $R$ is Gorenstein if and only if         $[\Mfo_E:\mff(R)]_\mfo=2S(R)$ by Equation \eqref{eq:cond12case}, we have $S(R) = (f_1 + f_2)/2$. 
 Substituting the Gorenstein parameter constraints from Proposition \ref{prop:splitram}.(5) into the conductor formula in Proposition \ref{prop:splitram}.(4), we identify the specific relations between the invariants:
    If $f_1$ is odd, then the conditions $a=2b+1$ and $\ord(c) \ge a$ force $f_1 = a$ and $f_2 = 2a-1 = 2f_1-1$. 
    If $f_1 > 0$ is even, then the condition $a=2\ord(c) \le 2b$ forces $f_1 = a$ and $f_2 = a+2b = f_1+2b$.
      
    The explicit formulas follow by evaluating the multisummation in Theorem \ref{thm:splitramorders} under these constraints. The closed forms of these sums depend on the residue of $f_1$ modulo $4$. 
\end{proof}

\subsection{The case when $E$ has a splitting type $(1\ 1\ 1 )$}
In this case, $E$ is an \'etale algebra such that $E\cong F\times F\times F$. We identify $\Mfo_E\cong \mfo\times\mfo\times \mfo$.
As in Equation (\ref{eq:index_of_orders_splitcase}), we use the following notation:
\begin{equation}
    ((v_1,v_2),(w_1,w_2))_\mfo:=\mfo\langle(1,1,1),(0,v_1,v_2),(0,w_1,w_2)\rangle \textit{ with }v_1,v_2,w_1,w_2\in \mfo.
\end{equation}
We denote $((\pi^a,0),(c,\pi^b))_\mfo$ by $\Mfo_{a,b,c}$. Here we emphasize $a,b\geq 0$.

\begin{proposition} \label{prop:totsplit}
 Let $c\in \mfo$.
\begin{enumerate}
    \item Any order is of the form $\Mfo_{a,b,c}$ where \[\begin{cases}
        c\in (\pi^{\lceil\frac{a}{2}\rceil}) & \text{if }a\leq 2b;\\
        c\in (\pi^{a-b}) \text{ or }c\in \pi^b+(\pi^{a-b}) &\text{if }a>2b.
    \end{cases}.\]
    Conversely, an $\mfo$-module $\Mfo_{a,b,c}$ satisfying the above condition is an order.
    \item 
    An order $\Mfo_{a',b',c'}$ satisfying (1) contains an order $\Mfo_{a,b,c}$ if and only if $a'\leq a,$ $b'\leq b$, and 
\[
    \left\{
    \begin{array}{l l}
       c' \in (\pi^{a'-(b-b')})&\textit{if $\ord(c)\geq a'$};  \\
        c'-\pi^{-(b-b')}c\in (\pi^{a'-(b-b')})  & \textit{if $\ord(c)<a'$}.
    \end{array}
    \right.
    \]
    In the second case, $\lceil \frac{a'}{2} \rceil\leq \ord(c)-(b-b')=\ord(c')$ if $a'\leq 2b'$, and $a'>b$ if $a'>2b'$.
    
    \item 
    $\Mfo_{a',b',c'}=\Mfo_{a,b,c}$ as an order if and only if $a=a'$, $b=b'$, and $c-c'\in (\pi^a)$.
    
    \item   
The Serre invariant of $\Mfo_{a,b,c}$ is $S(\Mfo_{a,b,c}) = a+b$. 
The conductor $\mathfrak{f}(\Mfo_{a,b,c})$ of $\Mfo_{a,b,c}$ is given by $\pi^{f_1}\mfo\times \pi^{f_2}\mfo\times \pi^{f_3}\mfo  = (\pi^{\max\big(b,\ a+b-\ord(c-\pi^b)\big)}) \times (\pi^{a}) \times (\pi^{\max\big(b,\ a+b-\ord(c)\big)})$. 
More precisely, 
$(f_1,f_3)=\begin{cases}
    \big(\max(a,b),\ b\big) & \text{if } \ord(c)\geq a,\\
\big(\max(b, a+b-\ord(c-\pi^b)),\ a+b-\ord(c)\big) & \text{if }\ord(c)<a.
\end{cases}$
 
    \item
    An order $\Mfo_{a,b,c}$ is Gorenstein if and only if
 \[
    \left\{
    \begin{array}{l l}
      a=0\textit{ or }b=0  &\textit{if }\ord(c)\geq a;  \\
      \ord(c-\pi^b)=a-b>0,  &\textit{if }\ord(c)< a \textit{ and } c- \pi^b \in (\pi^{b+1}); \\
      \ord(c)=\frac{a}{2}\leq b,\text{ or }\ord(c)=a-b>b    &\textit{if }\ord(c)< a \textit{ and } c- \pi^b \notin (\pi^{b+1}).
    \end{array}
    \right.
    \]

\end{enumerate}
        
\end{proposition}
\begin{proof}
\begin{enumerate}
    \item  By an argument similar to that used in the proof of Proposition \ref{prop:splitunram}.(1), any order $\Mfo$ admits a basis of the form $\{(1,1,1),(0,\pi^a,0),(0,c,\pi^b)\}$ as an $\mfo$-module.
 In order to characterize the condition when $\Mfo$ is closed under multiplication, we observe that 
    \[
(0,c,\pi^b)^2=\pi^b(0,c,\pi^b)+(0,c^2-\pi^bc,0)
\textit{ in } F\times F\times F.
\]
    Thus, $\Mfo$ is closed under multiplication if and only if $c^2-\pi^bc=c(c-\pi^b)\in (\pi^a)$.
    \begin{itemize}
        \item Suppose that $a\leq 2b$.
        If $c\in (\pi^{\lceil\frac{a}{2}\rceil})$, then $c$, $c-\pi^b\in (\pi^{\lceil\frac{a}{2}\rceil})$ and thus $c^2-\pi^bc\in (\pi^a)$.
        Otherwise, we have $\ord(c)=\ord(c-\pi^b)<\frac{a}{2}$, and thus $c^2-\pi^bc\notin (\pi^a)$.
        Therefore, $c(c-\pi^b)\in (\pi^a)$ if and only if $c\in (\pi^{\lceil\frac{a}{2}\rceil})$.
\item
    Suppose that $a>2b$. 
    The identity $c-(c-\pi^b)=\pi^b$ yields that $\min(\ord(c),\ord(c-\pi^b))\leq b$.
    Thus, if $c(c-\pi^b)\in (\pi^a)$, then $\max(\ord(c),\ord(c-\pi^b))\geq a-b$. This yields that either $c\in (\pi^{a-b})$ or $c-\pi^b \in (\pi^{a-b})$.
    Conversely, if $c \in (\pi^{a-b})$, then $\ord(c-\pi^b) = \ord(\pi^b) = b$ since $a-b > b$. Thus we have $\ord(c(c-\pi^b)) \ge a$. The case for $c-\pi^b \in (\pi^{a-b})$ is
           symmetric.
    
\end{itemize}
    
    \item       
      We have $\pi^a=\pi^{a-a'}\cdot \pi^{a'}$ and $(0, c, \pi^b) = \pi^{-a'}(c-\pi^{b-b'}c')(0, \pi^{a'}, 0) + \pi^{b-b'}(0, c', \pi^{b'})$. 
       Thus, $(0, \pi^a, 0)$ and $(0, c, \pi^b)$ are contained in $\Mfo_{a',b',c'}$ if and only if 
       \[
                  a'\leq a, b'\leq b, \text{ and }c-\pi^{b-b'}c'\in (\pi^{a'}).
       \]

       If $\ord(c)\geq a'$, then $c-\pi^{b-b'}c'\in (\pi^{a'})$ if and only if $c'\in (\pi^{a'-(b-b')})$.
       If $\ord(c)<a'$, then $c-\pi^{b-b'}c'\in (\pi^{a'})$ if and only if $c'\in \pi^{-(b-b')}c+(\pi^{a'-(b-b')})$.
 
       In the second case, 
if $a'\leq 2b'$, then $\ord(c')=\ord(c)-(b-b')\geq\lceil \frac{a'}{2}\rceil$ since $c'\in (\pi^{\lceil\frac{a'}{2}\rceil})$ by (1). 
       If $a' >2b'$, then either $c'\in (\pi^{a'-b'})$ or $c'\in \pi^{b'}+(\pi^{a'-b'})$ by (1), which yields $\ord(c')\geq b'$. 
       Thus $c\in(\pi^b)$ and $a'>b$, since  $c\in \pi^{b-b'}c'+(\pi^{a'})$ and $\ord(c)<a'$.
       
 
    \item
    This is a direct consequence of (2).

    \item  The proof of $S(\Mfo_{a,b,c})=a+b$ is identical to that of Proposition \ref{prop:splitunram}.(4).

     Next, we determine the conductor $\mathfrak{f}(\Mfo_{a,b,c})$, the largest ideal $(\pi^{f_1}) \times (\pi^{f_2}) \times (\pi^{f_3})$ of $\Mfo_E \cong \mfo \times \mfo
           \times \mfo$ contained in $\Mfo_{a,b,c}$.
      An element $(w_1, w_2, w_3) \in \Mfo_E$ belongs to $\Mfo_{a,b,c}$ if and only if $(0, w_2-w_1, w_3-w_1) \in \mfo \langle (0, \pi^a, 0), (0, c, \pi^b) \rangle$, equivalently $w_3-w_1\in (\pi^b)$ and $w_2-w_1-\pi^{-b}c(w_3-w_1)\in (\pi^a)$.
    We compute $f_1, f_2,$ and $f_3$ as follows:
     \begin{itemize}
         \item For $(0, \pi^{f_2}, 0) \in \Mfo_{a,b,c}$, we require $-\pi^{f_2}\in (\pi^a)$, yielding $f_2 = a$.
      \item For $(\pi^{f_1}, 0, 0) \in \Mfo_{a,b,c}$, we require $-\pi^{f_1} \in (\pi^b)$ (so $f_1 \ge b$) and 
      $\pi^{f_1-b}(c-\pi^b) \in (\pi^a)$. This forces $f_1 - b + \ord(c-\pi^b) \ge a$, yielding $f_1 = \max(b, a+b-\ord(c-\pi^b))$.
      \item For $(0, 0, \pi^{f_3}) \in \Mfo_{a,b,c}$, we require $-\pi^{f_3} \in(\pi^b)$ (so $f_3 \ge b$) and 
           $-\pi^{f_3-b}c \in (\pi^a)$. This forces $f_3 - b + \ord(c) \ge a$, yielding $f_3 = \max(b, a+b-\ord(c))$.
  \end{itemize}
  The piecewise evaluation of $(f_1, f_3)$ follows from these formulas.

    \item 
    The order $\Mfo:=\Mfo_{a,b,c}$ is Gorenstein 
    if and only if $[\Mfo_E:\mathfrak{f}(\Mfo)]_\mfo=2[\Mfo_E:\Mfo]_\mfo$ by Equation (\ref{eq:cond12case}).
    By (4), we have $[\Mfo_E:\mathfrak{f}(\Mfo)]_\mfo=\left\{\begin{array}{l l}
      a+b+\max(a,b)   & \textit{if }\ord(c)\geq a,\\
      3a+2b-\ord(c)-\ord(c-\pi^b)   & \textit{if }\ord(c)< a,
    \end{array}\right.$ and $2S(\Mfo)=2(a+b)$.
    Plugging these into the equation, we obtain
\[
\left\{\begin{array}{l l}
     \max(a,b) =a+b  & \textit{if }\ord(c)\geq a,\\
      a-\ord(c)-\ord(c-\pi^b)=0   & \textit{if }\ord(c)< a.
    \end{array}\right.
\]

If $\ord(c)\geq a$, then the equation is equivalent to $a=0$ or $b=0$. 
If $\ord(c)<a$, then we divide into two cases according to whether  $c-\pi^b\in (\pi^{b+1})$ or not.
If $c-\pi^b\in (\pi^{b+1})$, the condition reduces to $ a-b=\ord(c-\pi^b)$ with $a>b$.
Otherwise, it reduces to $a=\ord(c) +\min (\ord(c),b)$.
This is exactly the same condition as the desired one.
    \qedhere
\end{enumerate}    
\end{proof}

\begin{corollary}\label{cor:totallysplit:overorder_cond}
Combining Proposition \ref{prop:totsplit}.(1)-(2),  an $\mfo$-module $\Mfo_{a',b',c'}$ is an overorder of an order $\Mfo_{a,b,c}$ if and only if $a'\leq a,$ $b'\leq b$, and 
\begin{enumerate}
    \item 
if $a'\leq \ord(c)$, then $\left\{
    \begin{array}{l l}
       c' \in (\pi^{\max(\lceil\frac{a'}{2}\rceil,a'-(b-b'))})  &\textit{if } a'\leq 2b';  \\
       c'\in (\pi^{\max(a'-b',a'-(b-b'))}) &\textit{if }  \max(2b', b)< a';\\
       c'\in (\pi^{a'-b'}) \textit{ or }  c'-\pi^{b'}\in (\pi^{a'-b'})&\textit{if } 2b'< a'\leq b,
    \end{array}
    \right.$
\item if   $a'> \ord(c)$ (so that $\ord(c)-(b-b')=\ord(c')$), then $c'-\pi^{-(b-b')}c\in (\pi^{a'-(b-b')})$ and   \[ \left\{
    \begin{array}{l l}
   \lceil\frac{a'}{2}\rceil\leq \ord(c)-(b-b')  & \textit{if } a'\leq 2b';\\[8pt]
c'\in(\pi^{a'-b'})  \textit{ and }  a'>b   &\textit{if } 2b'<a' \textit{and} ~c\in (\pi^{a'+b-2b'});\\[8pt] 
   c'-\pi^{b'}\in (\pi^{a'-b'}) \textit{ and } a'>b
        &\textit{if } 2b'< a' \textit{and} ~c-\pi^{b}\in (\pi^{\min(a',a'+b-2b')}).\\
        \end{array}
    \right. \]
\end{enumerate}
\end{corollary}
    \begin{proof}
For (1), suppose that $a'\leq \ord(c)$. 
     The first case is obvious.
     For the remaining cases where $2b'< a'$, by Proposition \ref{prop:totsplit}.(1)-(2),
    $\Mfo_{a',b',c'}$ is an overorder of $\Mfo_{a,b,c}$ if and only if $c'\in (\pi^{a'-(b-b')})$ and $c'$ lies in either $(\pi^{a'-b'})$ or $\pi^{b'}+(\pi^{a'-b'})$.
    If $c'\in (\pi^{a'-b'})$, then this condition reduces to $c'\in (\pi^{\max(a'-b',a'-(b-b'))})$.
    If $c'\in \pi^{b'}+(\pi^{a'-b'})$, then $c'\in (\pi^{a'-(b-b')})$ if and only if $a'\leq b$. 
    This directly yields the second case. 
    For the third case where $a'\leq b$,  we have $a'-b'>b'\geq a'-(b-b')$, and thus $c'\in (\pi^{\max(a'-b',a'-(b-b'))})$ simplifies to $c'\in (\pi^{a'-b'})$.

    For (2), suppose that $a'> \ord(c)$. 
The first case  is obvious.
For the remaining cases where $2b'< a'$ (so that $a'>b$),  
we combine the condition $c'\in (\pi^{a'-b'})$ or $c'\in \pi^{b'}+(\pi^{a'-b'})$ in Proposition \ref{prop:totsplit}.(1) with $c'-\pi^{-(b-b')}c\in (\pi^{a'-(b-b')})$ in Proposition \ref{prop:totsplit}.(2).

We rewrite  $c'-\pi^{-(b-b')}c\in (\pi^{a'-(b-b')})$ in terms of $c$ so that $c\in \pi^{(b-b')}c'+(\pi^{a'})$.
This condition on $c$ implies  $c\in (\pi^{\min(a',a'+b-2b')})$ if $c'\in (\pi^{a'-b'})$,
and $c\in \pi^b+ (\pi^{\min(a',a'+b-2b')})$ if $c'\in \pi^{b'}+(\pi^{a'-b'})$.
Here, two cases for $c$ 
are mutually exclusive since $b<\min(a',a'+b-2b')$ by $a'>b$. 
Note that in the former case $c\in (\pi^{a'+b-2b'})$ since $a'> \ord(c)$.

Therefore, the combination of two conditions is equivalent to $c'-\pi^{-(b-b')}c\in (\pi^{a'-(b-b')})$ and 
\begin{itemize}
 \item 
    If $c \in (\pi^{a'+b-2b'})$, then $c'\in(\pi^{a'-b'})$, or equivalently, $a'-b'\leq \ord(c)-(b-b')$.   

    \item    
    If $c-\pi^b\in (\pi^{\min(a',a'+b-2b')})$, then  $c'-\pi^{b'}\in (\pi^{a'-b'})$.  \qedhere
\end{itemize}
    \end{proof}

  

\begin{theorem}\label{thm:totsplitformula}
Let $(\pi^{f_1}) \times (\pi^{f_2}) \times (\pi^{f_3}) \left(\subset \mfo \times \mfo \times \mfo \right)$ be the conductor of an order $R$ in $\Mfo_E$, and let $S=S(R)$ be its Serre invariant. 
To simplify the notation, we denote $F_i := S - f_i$ for $i=1,2,3$.
By permuting the factors of $\Mfo_E \cong \mfo^3$ if necessary, we may assume that $F_1 \le f_2$ and $F_3 \le f_2$.
Then $\#\overline{\mathrm{Cl}}(R)$ is given by the following formula:
\begin{align*}
        &\#\clb(R)=2\sum_{0\leq b'\leq F_2}\left(
        \sum_{0\leq a'\leq F_3, 2b'}q^{\min(\lfloor a'/2 \rfloor,F_2-b')}
        +\sum_{2b', F_2 < a'\leq F_3}q^{\min(b',F_2-b')}
        +\sum_{2b'<a'\leq F_3, F_2}2q^{b'}\right.\\
        &\left.
        +\sum\limits_{\substack{F_3<a'\leq f_2, 2b'\\ \lceil a'/2 \rceil-b'\leq F_3-F_2}}q^{F_2-b'}
        +\sum_{\substack{F_3, 2b', F_2 < a'\leq f_2\\ a'-2b'\leq F_3-F_2}}q^{F_2-b'}
        +\sum_{\substack{F_3, 2b', F_2 < a'\leq f_2\\ \min(a', a'+F_2-2b') \leq F_1}}q^{\min(b',F_2-b')}
        \right)-(F_2+1)\\
        & -\sum_{0<b'\leq F_2}\sum\limits_{\substack{0<a'\leq F_3, 2b'\\ a'/2\leq F_2-b', ~2\mid a'}}(q^{a'/2}-q^{a'/2-1})
        +\sum\limits_{\substack{0<a'\leq F_3, F_2\\ 2\mid a'}}q^{a'/2-1}
        -\sum_{F_2 < a'\leq F_3}1
        -\sum_{\substack{0<b'\leq F_2\\ 2b'\leq F_2}}\sum_{F_2 < a'\leq F_3}(q^{b'}-q^{b'-1})\\
        &-\sum_{0<a'\leq F_3, F_2}2
        -\sum_{0<b'\leq F_2}\sum_{2b'<a'\leq F_3, F_2}2(q^{b'}-q^{b'-1})
        -\sum_{\substack{F_3<a'\leq f_2, 2F_3\\ F_3\leq F_2, ~2\mid a'}}q^{F_3-a'/2}
        +\sum_{\substack{F_3<a'\leq f_2, 2F_2\\ F_2+1 \leq F_1, ~2\mid a'}}q^{F_2-a'/2}\\
        &-\sum_{\substack{F_3<a'\leq f_2, ~F_2<F_3\\ 2\mid (a'+F_2-F_3)}}q^{\frac{F_2+F_3-a'}{2}}
        -\sum_{F_3, F_2 < a'\leq F_1}1
        -\sum_{\substack{0<b'\leq F_2\\ 2b'\leq F_2}}\sum\limits_{F_3, F_2 < a'\leq F_1}(q^{b'}-q^{b'-1})
        - \sum_{\substack{0<b'\leq F_2\\ F_2 < 2b'}}\sum_{\substack{F_3, 2b' < a'\leq f_2\\ a'+F_2-2b' = F_1}}q^{F_2-b'}.
\end{align*}
Note that this formula relies entirely on the invariant quantities $f_1, f_2, f_3,$ and $S(R)$. For an explicit closed-form evaluation of this multi-summation as a rational function in $q$, see Appendix \ref{app:formula}.
    \end{theorem}

\begin{proof}
       The proof is parallel to that of Theorem \ref{thm:splitunroverorders}.
 An overorder of $R(=\Mfo_{a,b,c})$ is of the form $\Mfo_{a',b',c'}$ where $a'$, $b'$, and $c'$ satisfy the conditions of Corollary \ref{cor:totallysplit:overorder_cond}.
We enumerate the number of such overorders and Gorenstein overorders with fixed $a'(\leq a)$ and $ b'(\leq b)$. 
    
    \textbf{1. Case $a'\leq \ord(c)$ and $a'\leq 2b'$:} 
There are $q^{\min(\lfloor\frac{a'}{2}\rfloor,b-b')}$ choices for $c'$ up to $(\pi^{a'})$ for fixed $a'$ and $b'$.
Among these, we count Gorenstein overorders using  Proposition \ref{prop:totsplit}.(5).

    \begin{itemize}
\item 
    If $\ord(c')\geq a'$, then $\Mfo_{a',b',c'}$ is Gorenstein if and only if $a'=0$ or $b'=0$.
    Since $a'\leq 2b'$, it suffices to consider the case where $a'=0$.
   Then $c'$ is unique up to $(\pi^{a'})$ for each $0\leq b'\leq b$.
    Thus, the number of Gorenstein overorders when $a'=0$ is $b+1$.
    
\item 
    If $\ord(c')<a'$ and $c'-\pi^{b'}\in (\pi^{b'+1})$, then $\Mfo_{a',b',c'}$ is Gorenstein if and only if $\ord(c'-\pi^{b'})=a'-b'>0$.
    Since $\ord(c'-\pi^{b'})$ is larger than $b'$, we have $a'-b'>b'$, which contradicts the condition $a'\leq 2b'$.
    Thus, there exists no Gorenstein overorder of $R$ in this case.
    
\item 
    If $\ord(c')<a'$ and $c'-\pi^{b'}\not\in(\pi^{b'+1})$, then $\Mfo_{a',b',c'}$ is Gorenstein if and only if $a' = 2\ord(c')$, which requires $2|a'$ and $\frac{a'}{2}\geq a'-(b-b')$ since $c'\in (\pi^{\max(\lceil\frac{a'}{2}\rceil,a'-(b-b'))})$.
If $a'=0$, then it reduces to the first case.
If $a'>0$, then the number of Gorenstein overorders  is $\left\{
    \begin{array}{l l}
     q^{\frac{a'}{2}}-q^{\frac{a'}{2}-1}    &  \textit{if $a'<2b'$};\\
       q^{\frac{a'}{2}}-2q^{\frac{a'}{2}-1} & \textit{if $a'=2b'$}.
    \end{array}
    \right.$
    Here, the second expression reflects  $c'-\pi^{b'}\notin (\pi^{b'+1})$.
  
    \end{itemize}

    \textbf{2. Case $a'\leq \ord(c)$ and $\max(2b',b)<a'$:}
There are  $q^{\min(b',b-b')}$ choices for $c'$ up to $(\pi^{a'})$ for fixed $a'$ and $b'$.
Among these, we count Gorenstein overorders  using Proposition \ref{prop:totsplit}.(5).
   \begin{itemize}
        \item 
    If $\ord(c')\geq a'$, then $\Mfo_{a',b',c'}$ is Gorenstein if and only if $b'=0$, since $a'>0$.
  Then $c'$ is unique up to $(\pi^{a'})$ for each $b<a'\leq\min(a, \ord(c))$. 
    Thus, the number of Gorenstein overorders when $b'=0$ is $\min(a, \ord(c))-b$. 

    \item
    If $\ord(c')<a'$, then $c'-\pi^{b'}\notin (\pi^{b'+1})$ since $c'\in (\pi^{\max(a'-b',a'-(b-b'))})\subset (\pi^{a'-b'})$ and $b'<a'-b'$.
    Thus, $\Mfo_{a',b',c'}$ is Gorenstein if and only if $\ord(c')=a'-b'$, which requires  that $a'-b'\geq a'-(b-b')$ and $b'>0$.
    There are $q^{b'}-q^{b'-1}$ such orders.
    \end{itemize}
    
\textbf{3. Case $a'\leq \ord(c)$ and $2b'<a'\leq b$:} 
Since $c'-\pi^{b'}\in (\pi^{a'-b'})$ or $c'\in (\pi^{a'-b'})$, 
there are $2q^{b'}$ choices for $c'$ up to $(\pi^{a'})$ for fixed $a'$ and $b'$.
    Among these, we count Gorenstein overorders. 
    \begin{itemize}
        \item 
    If $\ord(c')\geq a'$, then $\Mfo_{a',b',c'}$ is Gorenstein if and only if $b'=0$, since $a'>0$.
    Then $c'$ is unique up to $(\pi^{a'})$ for each $0< a'\leq \min(a, b,\ord(c))$.
Thus, the number of Gorenstein overorders when $b'=0$ is $\min(a, b,\ord(c))$.

        \item
    If $\ord(c')<a'$, then $\Mfo_{a',b',c'}$ is Gorenstein if and only if $\ord(c'-\pi^{b'})=a'-b'$ or $\ord(c')=a'-b'$. 
    For fixed $a'$ and $b'$, the number of Gorenstein overorders is $
    \left\{
    \begin{array}{l l}
       2(q^{b'}-q^{b'-1})  & \textit{if }b'\neq 0;  \\
       1  & \textit{if }b'=0.
    \end{array}\right.$
\end{itemize}

    \textbf{4. Case $a'>\ord(c)$ and $a'\leq 2b'$:} 
    Since $c'-\pi^{-(b-b')}c\in (\pi^{a'-(b-b')})$ with $\lceil\frac{a'}{2}\rceil\leq \ord(c)-(b-b')$, 
    there are $q^{b-b'}$ choices for $c'$ up to $(\pi^{a'})$  for fixed $a'$, $b'$.
    Among these, we count Gorenstein overorders. 
\begin{itemize}
    \item 
    If $\ord(c')\geq a'$, then there exists no Gorenstein overorder since  $a', b'\neq 0$.
    
    \item
    If $\ord(c')<a'$, then $\Mfo_{a',b',c'}$ is Gorenstein if and only if $c'- \pi^{b'} \notin  (\pi^{b'+1})$ and $a'=2\ord(c')$ 
    by the same argument as in \textbf{1. Case $a'\leq \ord(c)$ and $a'\leq 2b'$}.
It happens only if $2|a'$ and $\frac{a'}{2}=\ord(c)-(b-b')$.
      For fixed $a'$ and $b'$ such that $\frac{a'}{2}=\ord(c)-(b-b')$ so that $a'=2\ord(c')$, the number of Gorenstein overorders is 
    $
    \left\{
    \begin{array}{l l}
    q^{b-b'}&\textit{if $c-\pi^{b}\notin (\pi^{b+1})$};\\
    0 & \textit{otherwise}.
    \end{array}
    \right.$ 
    Here we use the fact that $c'- \pi^{b'}\in (\pi^{b'+1})$ if and only if $c- \pi^{b} \in (\pi^{b+1})$. 
    \end{itemize}

    \textbf{5. Case $a'>\ord(c)$, $2b'<a'$, and $c\in (\pi^{a'+b-2b'})$:}
    Since $c'-\pi^{-(b-b')}c\in (\pi^{a'-(b-b')})$ with $a'-b'\leq \ord(c)-(b-b')$ and $a'>b$ by Corollary \ref{cor:totallysplit:overorder_cond}, there are $q^{b-b'}$ choices for $c'$ up to $(\pi^{a'})$.
    Among these, we count the number of Gorenstein overorders using Proposition \ref{prop:totsplit}.(5).
    \begin{itemize}
    \item 
    If $\ord(c')\geq a'$, then $\pi^{-(b-b')}c\in (\pi^{a'-(b-b')})$, which contradicts the condition $a'>\ord(c)$. Thus, there exists no overorder of $R$ in this case.
    \item
    If $\ord(c')<a'$ and $c'- \pi^{b'} \in (\pi^{b'+1})$, then $b'=\ord(c')=\ord(c)-(b-b')\geq a'-b'$, which contradicts condition $a'>2b'$.
    Thus, there exists no overorder of $R$ in this case.
    \item
    If $\ord(c')<a'$ and $c' - \pi^{b'}\not\in(\pi^{b'+1})$, then $\Mfo_{a',b',c'}$ is Gorenstein if and only if $\ord(c')=a'-b'$, which occurs exactly when $a'-b'=\ord(c)-(b-b')$.
    For fixed $a'$ and $b'$ satisfying $a'-b'=\ord(c)-(b-b')$, the number of Gorenstein overorders is $q^{b-b'}$.
    \end{itemize}

    \textbf{6. Case $a'>\ord(c)$, $2b'<a'$, and $c-\pi^b\in (\pi^{\min(a',a'+b-2b'))})$:} 
We claim that there are $q^{\min(b',b-b')}$ choices for $c'$ up to $(\pi^{a'})$.

By Corollary \ref{cor:totallysplit:overorder_cond}, we have $c'-\pi^{-(b-b')}c\in (\pi^{a'-(b-b')})$, $c'-\pi^{b'}\in (\pi^{a'-b'})$, and $a'>b$.   
    If $ 2b' \leq b$, then $c-\pi^b\in (\pi^{a'})$ and thus $c'-\pi^{-(b-b')}c\in (\pi^{a'-(b-b')})$ is equivalent to $c'-\pi^{b'}\in (\pi^{a'-(b-b')})$.
Combining $c'-\pi^{b'}\in (\pi^{a'-b'})$ and $c'-\pi^{b'}\in (\pi^{a'-(b-b')})$, we have $c'-\pi^{b'}\in (\pi^{a'-b'})$.

    If $2b'>b$, then $c-\pi^b\in (\pi^{a'+b-2b'})$ 
    and thus 
    $c'-\pi^{b'}\in(\pi^{a'-b'})$ is equivalent to $c'-\pi^{-(b-b')}c\in (\pi^{a'-b'})$.
 Combining $c'-\pi^{-(b-b')}c\in (\pi^{a'-b'})$ and $c'-\pi^{-(b-b')}c\in (\pi^{a'-(b-b')})$, we have $c'-\pi^{-(b-b')}c\in (\pi^{a'-(b-b')})$.
 These two cases prove the claim. 
 
We count Gorenstein overorders for fixed $a'$ and $b'$ with $a'>\max(\ord(c), 2b', b)$. Here $c- \pi^b \in (\pi^{b+1})$. 
    \begin{itemize}
        \item 
If $\ord(c') (=\ord(c)-(b-b')) \geq a'$ then $b'\geq a'$ which contradicts   condition $a'>2b'$.
    Thus, there exists no overorder of $R$ in this case.   
    \item
    If $\ord(c')<a'$, then $\Mfo_{a',b',c'}$ is Gorenstein if and only if $\ord(c'-\pi^{b'})=a'-b'$. 
    If $2b'\leq b$, then $c'-\pi^{b'}\in (\pi^{a'-b'})$ and the desired number is 
$    
    \left\{
    \begin{array}{l l}
     1 & \textit{if $b'=0$};\\
    q^{b'}-q^{b'-1}&\textit{if $b'>0$}.
    \end{array}
    \right.
$
    If $2b'>b$, then $c'-\pi^{-(b-b')}c\in (\pi^{a'-(b-b')})$ and the desired number is
$
    \left\{
    \begin{array}{l l}
     q^{b-b'} & \textit{if $\ord(c-\pi^b)=a'+b-2b'$};\\
    0&\textit{otherwise}. 
    \end{array} 
    \right. $
\end{itemize}

Summing up these using the equation in Proposition \ref{prop:orderidealcounting} yields the intermediate formula expressed purely in terms of the initial local parameters $a, b,$ and $c$:        
    \begin{align*}
        &\#\clb(R)=2\sum_{0\leq b'\leq b}\left(
        \sum_{0\leq a'\leq a, \ord(c), 2b'}q^{\min(\lfloor a'/2 \rfloor,b-b')}
        +\sum_{2b', b < a'\leq a, \ord(c)}q^{\min(b',b-b')}
        +\sum_{2b'<a'\leq a, \ord(c), b}2q^{b'}\right.\\
        &\left.
        +\sum\limits_{\substack{\ord(c)<a'\leq a, 2b'\\ \lceil a'/2 \rceil-b'\leq \ord(c)-b}}q^{b-b'}
        +\sum_{\substack{\ord(c), 2b', b < a'\leq a\\ a'-2b'\leq \ord(c)-b}}q^{b-b'}
        +\sum_{\substack{\ord(c), 2b', b < a'\leq a\\ c-\pi^b\in(\pi^{\min(a',a'+b-2b')})}}q^{\min(b',b-b')}
        \right)-(b+1)\\
        &
        -\sum_{\substack{b'\leq b\\ 0<b'}}\sum\limits_{\substack{0<a'\leq a, \ord(c), 2b'\\ a'/2\leq b-b', ~2\mid a'}}(q^{a'/2}-q^{a'/2-1})
        +\sum\limits_{\substack{0<a'\leq a, \ord(c), b\\ 2\mid a'}}q^{a'/2-1}
        -\sum_{\substack{a'\leq a, \ord(c)\\ b < a'}}1
        -\sum_{\substack{0<b'\leq b\\2b'\leq b}}\sum_{b < a'\leq a, \ord(c)}(q^{b'}-q^{b'-1})\\
        &-\sum_{\substack{a'\leq a, \ord(c), b\\ 0<a'}}2
        -\sum_{0<b'\leq b}\sum_{2b'<a'\leq a, \ord(c), b}2(q^{b'}-q^{b'-1})
        -\sum_{\substack{\ord(c)<a'\leq a, 2\ord(c)\\ \ord(c)\leq b, ~2\mid a'}}q^{\ord(c)-a'/2}
        +\sum_{\substack{\ord(c)<a'\leq a, 2b\\ c-\pi^b\in (\pi^{b+1}), ~2\mid a'}}q^{b-a'/2}\\
        &-\sum_{\substack{\ord(c)<a'\leq a\\ b<\ord(c)\\ 2\mid (a'+b-\ord(c))}}q^{b-\frac{a'+b-\ord(c)}{2}}
        -\sum_{\substack{\ord(c), b < a'\leq a\\ c-\pi^b\in (\pi^{a'})}}1
        -\sum_{\substack{0<b'\leq b\\2b'\leq b}}\sum\limits_{\substack{\ord(c), b < a'\leq a\\ c-\pi^b\in (\pi^{a'})}}(q^{b'}-q^{b'-1})
        - \sum_{\substack{0<b'\leq b\\ b < 2b'}}\sum_{\substack{\ord(c), 2b' < a'\leq a\\ \ord(c-\pi^b)=a'+b-2b'}}q^{b-b'}.
    \end{align*}

To express this in terms of $f_1, f_2, f_3$ and $S(R)$, 
we note that $f_2 = a$, $F_1 = \min(a, \ord(c-\pi^b))$, and $F_3 = \min(a, \ord(c))$ by Proposition \ref{prop:totsplit}.(4). 
These imply $F_1, F_3 \le f_2$. Thus, given an arbitrary conductor of $R$, we can always permute its components so that $f_2$ satisfies these bounds, justifying our assumption. We now have $b = F_2$, and we claim that the parameters involving $c$ are completely determined by $F_3$ and $F_1$, which will be explained below:
\begin{itemize}
  \item Since all summation indices satisfy $a' \le a = f_2$, the upper bound condition $a' \le \ord(c)$ is equivalent to $a' \le \min(a, \ord(c)) = F_3$. Furthermore, when $\ord(c)<a$, we have
  $\ord(c)=F_3$. Because $\ord(c)$ appears as a term (e.g., in $\ord(c)-b$) exclusively within summations restricted by $\ord(c) < a' \le a$, we can replace $\ord(c)$ with $F_3$ without loss
  of generality.

\item $\ord(c-\pi^b)$ is compared against three specific bounds:
           $\min(a', a'+b-2b')$, $a'+b-2b'$, and $b+1$. The first bound is at most $a$. For the second bound, the corresponding summation explicitly restricts $2b' > b$, which, combined with $a' \le
           a$, guarantees $a'+b-2b' < a$. Since both of these bounds are $\leq a$, substituting $F_1$ for $\ord(c-\pi^b)$ yields an equivalent condition. For the third bound, $\ord(c-\pi^b) \ge
           b+1$, the equivalence holds identically if $b+1 \le a$. If $b+1 > a$ (i.e., $b \ge a$), the summation restriction $\ord(c) < a' \le a$ forces $\ord(c-\pi^b) = \ord(c) < a < b+1$,
           contradiction. 
\end{itemize}
Substituting these invariant expressions into the intermediate equation yields the stated formula.
 \end{proof}

\begin{corollary} \label{cor:gorenstein_calculation111}
If the order $R$ in Theorem \ref{thm:totsplitformula} is Gorenstein, then $S(R) = (f_1+f_2+f_3)/2$ and $\#\overline{\mathrm{Cl}}(R)$ is entirely determined by the conductor components $f_1, f_2,$ and $f_3$. 
Specifically, two of $f_1, f_2, f_3$ are equal, and the remaining component, denoted by $m$, is the smallest and even. Let $S = S(R)$ and $k = \lfloor m/4 \rfloor$. Then,
{\small
\[
\#\overline{\mathrm{Cl}}(R) = 
\begin{cases}
\displaystyle \frac{ (S-6k+1)q^{k+2} + (2S-12k+16)q^{k+1} - (3S-18k-7)q^k - (4S+18)q + 4S - 6 }{(q-1)^2} & \text{if } m=4k, \\[14pt]
\displaystyle \frac{ (3S-18k-2)q^{k+2} - (2S-12k-22)q^{k+1} - (S-6k-4)q^k - (4S+18)q + 4S - 6 }{(q-1)^2} & \text{if } m=4k+2.
\end{cases}
\]
}%
\end{corollary}

\begin{proof}
Since $R(=\Mfo_{a,b,c})$ is Gorenstein if and only if $[\Mfo_E:\mathfrak{f}(R)]_\mfo=2S(R)$ by Equation \eqref{eq:cond12case}, we have $S(R) = (f_1+f_2+f_3)/2$. 
To establish the relations between the invariants, we substitute the Gorenstein conditions from Proposition \ref{prop:totsplit}.(5) into the conductor formula $(f_1, f_2, f_3)$ given by Proposition \ref{prop:totsplit}.(4). We analyze the cases:

\begin{itemize}
    \item If $\ord(c) \ge a$, then $(f_1, f_3) = (\max(a,b), b)$, and the Gorenstein condition is $a=0$ or $b=0$. 
    If $a=0$, we have $(f_1, f_2, f_3) = (b, 0, b)$. If $b=0$, we have $(f_1, f_2, f_3) = (a, a, 0)$. Thus two components are equal, and the remaining component $m = 0$ is even and the smallest.
    
    \item If $\ord(c) < a$ and $c-\pi^b \in (\pi^{b+1})$, the Gorenstein condition is $\ord(c-\pi^b) = a-b > 0$. 
    Here $\ord(c)=b<a$. Thus $(f_1, f_2, f_3) = (2b, a, a)$. Here, two components are equal, and the remaining component $m=2b$ is even. Since $a-b >b$, $m$ is the smallest.

    \item If $\ord(c) < a$ and $c-\pi^b \notin (\pi^{b+1})$, the Gorenstein condition is $\ord(c) = a/2 \le b$ or $\ord(c) = a-b > b$.
    If $\ord(c) = a/2 \le b$, then $c-\pi^b \notin (\pi^{b+1})$ forces $\ord(c-\pi^b) = a/2$. The conductor is $(f_1, f_2, f_3) = (a/2+b, a, a/2+b)$.  Since $a/2 \le b$, $m=a$ is the smallest.
    If $\ord(c) = a-b > b$, then $\ord(c-\pi^b) = b$. The conductor is $(f_1, f_2, f_3) = (a, a, 2b)$. 
    Since $a-b > b$, $m=2b$ is the smallest.
\end{itemize}

In all cases, the tuple $(f_1, f_2, f_3)$ satisfies the stated constraint. Evaluating the multisummation in Theorem \ref{thm:totsplitformula} under these constraints yields
\[
\#\overline{\mathrm{Cl}}(R) =\sum_{i=0}^{k-1}(4S-24i-6)q^i
+\begin{cases}
(S-6k+1)q^k & \textit{if $m=4k$},\\[4pt]
(3S-18k-2)q^k & \textit{if $m=4k+2$}.
\end{cases}
\]
Evaluating this sum as a rational function in $q$ gives the stated formulas.
\end{proof}

\section{Global closed formula in the Gorenstein case}\label{sec:globalGorenstein}
\begin{remark}
    By the product formula in Proposition \ref{prop:global_local}, the global counting of ideal classes applies to an arbitrary (not necessarily Gorenstein) cubic order in a number field. However, since the exact local terms for general non-Gorenstein cubic orders are combinatorially highly complex (as demonstrated by the multi-summations in Sections \ref{sec:irred} and \ref{sec:split}), we omit the explicit global closed formula for the general case. For specific computations of these local terms, we refer the reader to the Python algorithms provided in Appendix \ref{app:formula}.
    \end{remark}
   
In this final section, we synthesize the local Gorenstein formulas derived in Sections \ref{sec:irred}-\ref{sec:split} with the product formula of Proposition \ref{prop:global_local} to establish a global closed formula for $\#(\mathrm{Cl}(R)\backslash\clb(R))$. 
This formula is expressed entirely in terms of the conductor $\mathfrak{f}(R)$ for any Gorenstein order $R$ in a cubic number field $E/F$. We retain the notation from Section \ref{sec:global_strategy} and introduce the following definitions:
          \begin{itemize}

\item We write the conductor of $R$ as
\[  \mathfrak{f}(R)=\prod_{\mathfrak{P}\in |\Mfo_E|}\mathfrak{P}^{f_{\mathfrak{P}}}, ~~~~\textit{where $f_\mathfrak{P}\in \mathbb{Z}_{\geq 0}$ and $f_\mathfrak{P}= 0$ for almost all $\mathfrak{P}$}. \]

\item Let $S_R := \{v \in |\mfo| \mid R_v \subsetneq \Mfo_{E_v}\}$ be the finite set of primes $v$ below some $\mathfrak{P}$ with $f_\mathfrak{P} > 0$.
For $v\in S_R$, let $\tau(v)$ denote the splitting type of $R_v$, which is one of
\[ 
\{\textit{unr},\ \textit{tr},\ (1\,2),\ (1\,1^2),\ (1\,1\,1)\}
 \]
where \textit{unr} and \textit{tr} correspond to the unramified and the totally ramified cases in Section \ref{sec:irred}, respectively, and the rest correspond to the splitting types introduced in Equation (\ref{eq:splittingtype}).

We naturally have $\#\clb(R_v)=1$ if $v\notin S_R$.

\item For $v\in |\mfo|$, let $\mathbf{f}(v) := (f_\mathfrak{P})_{\mathfrak{P}\mid v}$ be the tuple of conductor exponents at primes of $\Mfo_E$ above $v$, having length $1,1,2,2,$ or $3$ corresponding to $\tau(v)$.
\end{itemize}

We now combine the five local Gorenstein closed-form expressions derived in Sections \ref{sec:irred} and \ref{sec:split} into a single family of local factors $G_{\tau}$, indexed by the splitting type $\tau$.
\begin{definition}\label{def:Gtau}
For each splitting type $\tau$, and for each admissible conductor tuple $\mathbf f$, we define $G_\tau(q;\mathbf f)$ by the following compact formulae.  For fixed $\mathbf f$, each displayed quotient represents a polynomial in $q$.
\begin{enumerate}
\item[\textit{(unr)}] (cf. Corollary \ref{cor:gorenstein_calculation_unr}) For an even integer $f\in \mathbb{Z}_{\geq 0}$, let $n:=f/2$. Then,
{\footnotesize
\[
G_{\textit{unr}}(q;f)=\begin{cases}
\displaystyle \frac{q^{n/2+2}+7q^{n/2+1}+4q^{n/2}-(6n+9)q+(6n-3)}{(q-1)^2} & \textit{if $n$ is even};\\[14pt]
\displaystyle \frac{4q^{(n+3)/2}+7q^{(n+1)/2}+q^{(n-1)/2}-(6n+9)q+(6n-3)}{(q-1)^2} & \textit{if $n$ is odd}.
\end{cases}
\]
}%

\item[\textit{(tr)}] (cf. Corollary \ref{cor:gorenstein_calculation_tr}) For $f\in \mathbb{Z}_{\geq 0}$ satisfying $f\equiv 0$ or $2 \pmod 6$,
{\footnotesize
\[
G_{\textit{tr}}(q;f)=\begin{cases}
\displaystyle \frac{q^{n/2+2}+9q^{n/2+1}+2q^{n/2}-(6n+11)q+(6n-1)}{(q-1)^2} & \textit{if $f\equiv 0 \pmod 6$ and $n=\tfrac{f}{6}$ is even};\\[14pt]
\displaystyle \frac{5q^{(n+3)/2}+7q^{(n+1)/2}-(6n+11)q+(6n-1)}{(q-1)^2} & \textit{if $f\equiv 0 \pmod 6$ and $n=\tfrac{f}{6}$ is odd};\\[14pt]
\displaystyle \frac{2q^{3n/2+2}+9q^{3n/2+1}+q^{3n/2}-(6n+13)q^{n+1}+(6n+1)q^n}{q^n(q-1)^2} & \textit{if $f\equiv 2\pmod 6$ and $n=\tfrac{f-2}{6}$ is even};\\[14pt]
\displaystyle \frac{7q^{(3n+3)/2}+5q^{(3n+1)/2}-(6n+13)q^{n+1}+(6n+1)q^n}{q^n(q-1)^2} & \textit{if $f\equiv 2 \pmod 6$ and $n=\tfrac{f-2}{6}$ is odd}.
\end{cases}
\]
}%

\item[\textit{$(1\,2)$}] (cf. Corollary \ref{cor:gorenstein_calculation12}) For $(f_1,f_2)\in \mathbb{Z}_{\geq 0}^2$ with $f_1$ even, let $m:=f_1/2$. Then,
{\footnotesize
\[
G_{(1\,2)}(q;f_1,f_2)=\frac{q^{\lfloor f_1/4\rfloor}\bigl((f_2-f_1)Q_m(q)+P_m(q)\bigr)-4(f_2+3)q+4(f_2-1)}{(q-1)^2},
\]
}%
where the polynomials $P_m(q)$ and $Q_m(q)$ are given by
\[
P_m(q)=\begin{cases}q^2+10q+5 & \textit{if $m$ is even};\\5q^2+10q+1 & \textit{if $m$ is odd},\end{cases}\qquad
Q_m(q)=\begin{cases}q^2+2q-3 & \textit{if $m$ is even};\\3q^2-2q-1 & \textit{if $m$ is odd}.\end{cases}
\]

\item[\textit{$(1\,1^2)$}] (cf. Corollary \ref{cor:gorenstein_calculation112}) For $(f_1,f_2)\in \mathbb{Z}_{\geq 0}^2$ satisfying one of the following conditions:
\[ [f_1=0 \text{ and } f_2\geq 0], \quad [f_1 \text{ odd and } f_2=2f_1-1], \quad \text{or} \quad [f_1>0 \text{ even},~ f_2\geq 2f_1 \text{ and } f_2-f_1 \text{ even}], \]
we set $k:=\lfloor f_1/4 \rfloor$, and
{\scriptsize
\[
G_{(1\,1^2)}(q;f_1,f_2)=
\begin{cases}
\displaystyle \frac{q^{k+2}+12q^{k+1}+3q^k-(16k+14)q+16k-2}{(q-1)^2} & \textit{if $f_1=4k+1$};\\[14pt]
\displaystyle \frac{6q^{k+2}+10q^{k+1}-(16k+22)q+16k+6}{(q-1)^2} & \textit{if $f_1=4k+3$};\\[14pt]
\displaystyle \frac{(b-2k+1)q^{k+2}+(2b-4k+12)q^{k+1}-(3b-6k-3)q^k-(4b+8k+14)q+4b+8k-2}{(q-1)^2} & \textit{if $f_1=4k$, $b:=\tfrac{f_2-f_1}{2}$};\\[14pt]
\displaystyle \frac{(3b-6k+3)q^{k+2}-(2b-4k-12)q^{k+1}-(b-2k-1)q^k-(4b+8k+18)q+4b+8k+2}{(q-1)^2} & \textit{if $f_1=4k+2$, $b:=\tfrac{f_2-f_1}{2}$}.
\end{cases}
\]
}%

\item[\textit{$(1\,1\,1)$}] (cf. Corollary \ref{cor:gorenstein_calculation111}) For $(f_1,f_2,f_3)\in \mathbb{Z}_{\geq 0}^3$ where two components are equal, and the remaining component $m$ is the smallest and even. 
Let $S:=\tfrac{f_1+f_2+f_3}{2}$ and $k := \lfloor m/4 \rfloor$. 
Then,
{\scriptsize
\[
G_{(1\,1\,1)}(q;f_1,f_2,f_3)=
\begin{cases}
\displaystyle \frac{ (S-6k+1)q^{k+2} + (2S-12k+16)q^{k+1} - (3S-18k-7)q^k - (4S+18)q + 4S - 6 }{(q-1)^2} & \textit{if $m=4k$},\\[14pt]
\displaystyle \frac{ (3S-18k-2)q^{k+2} - (2S-12k-22)q^{k+1} - (S-6k-4)q^k - (4S+18)q + 4S - 6 }{(q-1)^2} & \textit{if $m=4k+2$}.
\end{cases}
\]
}%
\end{enumerate}
\end{definition}

\begin{theorem}\label{thm:globalGorensteinProduct}
Let $R$ be a Gorenstein order of conductor $\mathfrak{f}(R)$ in a cubic extension $E/F$ of number fields. Then
\begin{multline*}
\#(\mathrm{Cl}(R)\backslash\clb(R))
=\prod_{\tau(v)=\textit{unr}}G_{\textit{unr}}(q_v;\mathbf{f}(v))
\cdot\prod_{\tau(v)=\textit{tr}}G_{\textit{tr}}(q_v;\mathbf{f}(v))
\cdot\prod_{\tau(v)=(1\,2)}G_{(1\,2)}(q_v;\mathbf{f}(v))\\
\cdot\prod_{\tau(v)=(1\,1^2)}G_{(1\,1^2)}(q_v;\mathbf{f}(v))
\cdot\prod_{\tau(v)=(1\,1\,1)}G_{(1\,1\,1)}(q_v;\mathbf{f}(v)).
\end{multline*}
\end{theorem}

\begin{proof}
This follows immediately from the product formula in Proposition \ref{prop:global_local} and the explicit local computations collated in Definition \ref{def:Gtau}.
\end{proof}

\section{Application to orbit counting for Bhargava's $2 \times 3 \times 3$ cubes}\label{sec:bhargava}

In this section, we explain the fixed-ring consequence of the main formula for
Bhargava's prehomogeneous vector space \cite{Bha04}.  The integral
representation is over $\mathbb Z$, so throughout this section a cubic order
means a $\mathbb Z$-order in a cubic number field.  Let
\(V=\mathbb Z^2\otimes\mathbb Z^3\otimes\mathbb Z^3\), and let
\(\Gamma=\mathrm{GL}_2(\mathbb Z)\times\mathrm{SL}_3(\mathbb Z)^2\) act on
\(V\) in the standard way.  Bhargava proved the following parametrization.

\begin{theorem}[{\cite[Theorem 2]{Bha04}}]\label{thm:bhargava}
There is a canonical bijection between the set of $\Gamma$-orbits on
\(V\) and the set of isomorphism classes of pairs \((R,(I,I'))\), where \(R\)
is a cubic ring and \((I,I')\) is an equivalence class of balanced pairs of
fractional \(R\)-ideals.
\end{theorem}

For a cubic order \(R\), put \(K=R\otimes_{\mathbb Z}\mathbb Q\).  A pair of
fractional \(R\)-ideals \((I,I')\) in \(K\) is called \emph{balanced} if
\(II'\subseteq R\) and \(N(I)N(I')=1\), where
\(N(I)=[R:R\cap I]/[I:R\cap I]\).  Two balanced pairs \((I_1,I_1')\) and
\((I_2,I_2')\) are equivalent if there exists \(a\in K^\times\) such that
\(I_1=aI_2\) and \(I_1'=a^{-1}I_2'\).

Now fix a Gorenstein cubic \(\mathbb Z\)-order \(R\).  Let \(X_R\) denote the
set of integral \(\Gamma\)-orbits whose associated cubic ring is isomorphic to
\(R\).  Under Bhargava's bijection, \(X_R\) is identified with the set of
equivalence classes of balanced pairs \((R,(I,I'))\) with this fixed order
\(R\).  The ideal class group \(\mathrm{Cl}(R)\) acts on these classes by
\([J]\cdot[(I,I')]=[(JI,J^{-1}I')]\), where \(J\) is an invertible fractional
\(R\)-ideal and \(J^{-1}=(R:J)\).  The fixed-ring question considered here is
to compute \(\#(\mathrm{Cl}(R)\backslash X_R)\).

The Gorenstein hypothesis identifies this orbit problem with the ideal class
monoid.  For a fractional \(R\)-ideal \(I\), write \(I^\vee=(R:I)\).

\begin{proposition}\label{prop:bhargava_bijection}
Let \(R\) be a Gorenstein cubic \(\mathbb Z\)-order in a cubic number field.
Then every balanced pair \((I,I')\) is equivalent to a pair of the form
\((I,I^\vee)\).  Consequently, the set of equivalence classes of balanced
pairs with fixed order \(R\) is naturally identified with the ideal class
monoid \(\overline{\mathrm{Cl}}(R)\).
\end{proposition}

\begin{proof}
Let \((I,I')\) be a balanced pair.  The containment \(II'\subseteq R\) gives
\(I'\subseteq (R:I)=I^\vee\).  Since \(R\) is Gorenstein, local duality gives
\(N(I^\vee)=N(I)^{-1}\) by \cite[Characterization A 2.6(v)]{Jen15}.  The
balancing condition gives \(N(I')=N(I)^{-1}\), hence \(N(I')=N(I^\vee)\).
Because \(I'\subseteq I^\vee\), equality of norms forces \(I'=I^\vee\).

It remains to identify the equivalence relation.  If
\((I_1,I_1^\vee)\sim (I_2,I_2^\vee)\), then there exists \(a\in K^\times\)
such that \(I_1=aI_2\) and \(I_1^\vee=a^{-1}I_2^\vee\).  Conversely, if
\(I_1=aI_2\), then \((aI_2)^\vee=a^{-1}I_2^\vee\), so the second equality
follows automatically.  Thus equivalence classes of balanced pairs are exactly
principal equivalence classes of fractional \(R\)-ideals, namely the elements
of \(\overline{\mathrm{Cl}}(R)\).
\end{proof}

Under Proposition \ref{prop:bhargava_bijection}, the action of \(\mathrm{Cl}(R)\)
on \(X_R\) becomes the usual action of invertible ideal classes on
\(\overline{\mathrm{Cl}}(R)\) by multiplication.  Therefore
\(\#(\mathrm{Cl}(R)\backslash X_R)=
\#(\mathrm{Cl}(R)\backslash\overline{\mathrm{Cl}}(R))\).

\begin{theorem}\label{thm:bhargava_orbit_count}
Let \(R\) be a Gorenstein cubic \(\mathbb Z\)-order in a cubic number field,
and let \(X_R\) be the set of integral \(\Gamma\)-orbits whose associated cubic
ring is isomorphic to \(R\).  Then \(X_R\) is naturally identified with
\(\overline{\mathrm{Cl}}(R)\).  Consequently, the number of
\(\mathrm{Cl}(R)\)-equivalence classes, or genera in the product formula
sense, of these \(\Gamma\)-orbits is
\(\#(\mathrm{Cl}(R)\backslash\overline{\mathrm{Cl}}(R))\), which is given
explicitly by Theorem \ref{thm:globalGorensteinProduct} with \(F=\mathbb Q\).
\end{theorem}

\begin{remark}\label{rem:bhargava_gorenstein}
If \(R\) is not Gorenstein, Proposition \ref{prop:bhargava_bijection} need not
hold.  The ordinary dual \(I^\vee=(R:I)\) need not have inverse norm: for
non-invertible ideals where the defect appears, one can have
\(N(I)N(I^\vee)>1\).  In such cases a balanced partner must be chosen, if it
exists, as a proper sublattice of \(I^\vee\), and the fixed-ring orbit problem
is no longer determined by the orbit set of \(\overline{\mathrm{Cl}}(R)\)
alone.

The arithmetic counterpart of this obstruction is visible in Sections
\ref{sec:irred}--\ref{sec:split} and Appendix \ref{app:formula}.  For
non-Gorenstein local cubic orders, the local ideal class monoid count is given
by finite overorder summations.  In the Gorenstein case, the relation between
the Serre invariant and the conductor simplifies these summations to the
conductor-controlled polynomial local factors used in Theorem
\ref{thm:globalGorensteinProduct}.  A sequel in preparation studies the
corresponding fixed-ring problem through the Artin stack of Wood balanced
pairs, its groupoid-volume interpretation, and the derived homological
structure of the balancing condition.
\end{remark}


\bibliographystyle{alpha}
\bibliography{References}

\appendix
\section{Python algorithms for $\#\overline{\mathrm{Cl}}(R)$: local split case}\label{app:formula}
In Section \ref{sec:split}, the cardinalities $\#\overline{\mathrm{Cl}}(R)$ for general orders of split types $(1\ 2)$, $(1\ 1^2)$, and $(1\ 1\ 1)$ are given as multi-summations in Theorems \ref{thm:splitunroverorders}, \ref{thm:splitramorders}, and \ref{thm:totsplitformula}. 
We provide the following Python algorithms which evaluate these sums into a closed formula. For each split type, the algorithm takes the invariants $\mathbf{f}$ and $S(R)$ as inputs and outputs the exact polynomial $P(q)$ such that $\#\overline{\mathrm{Cl}}(R) = P(q)/(q-1)^2$.

\begin{small}
\begin{verbatim}
import math
def format_poly(coeffs):
    res = {}
    for p, v in coeffs.items():
        if v == 0: continue
        res[p] = res.get(p, 0) + v
        res[p + 1] = res.get(p + 1, 0) - 2 * v
        res[p + 2] = res.get(p + 2, 0) + v
    terms = []
    for p in sorted(res.keys(), reverse=True):
        if res[p] != 0:
            terms.append(f"{res[p]} q^{p}")
    return " + ".join(terms).replace("+ -", "- ") if terms else "0"
def compute_cl_12(f1, f2, SR):
    a, b, C = f1, SR - f1, SR - f2
    coeffs = {}
    def add_q(power, coef=1): coeffs[power] = coeffs.get(power, 0) + coef
    for bp in range(b + 1):
        for ap in range(a + 1):
            if ap <= 2 * bp and ap <= C:
                add_q(min(ap // 2, b - bp), 2)
        for ap in range(C + 1, a + 1):
            if ap <= 2 * bp and math.ceil(ap / 2) - bp <= C - b:
                add_q(b - bp, 2)
                
    add_q(0, -(b + 1))
    for ap in range(1, a + 1):
        if ap <= C and ap % 2 == 0: add_q(min(ap // 2, b - ap // 2), -1)
    for ap in range(C + 1, a + 1):
        if ap % 2 == 0 and b <= C: add_q(b - ap // 2, -1)
    for bp in range(1, b + 1):
        for ap in range(1, a + 1):
            if ap < 2 * bp and ap % 2 == 0 and ap <= C and ap // 2 <= b - bp:
                add_q(ap // 2, -1)
                add_q(ap // 2 - 1, 1)
    for ap in range(C + 1, a + 1):
        if ap % 2 == 0 and C < b: add_q(C - ap // 2, -1)

    return format_poly(coeffs)
def compute_cl_112(f1, f2, SR):
    a, b, C = f1, SR - f1, math.ceil((2 * SR - f2) / 2)
    coeffs = {}
    def add_q(power, coef=1): coeffs[power] = coeffs.get(power, 0) + coef

    for bp in range(b + 1):
        for ap in range(a + 1):
            if ap <= 2 * bp + 1 and ap <= C: add_q(min(ap // 2, b - bp), 2)
        for ap in range(C + 1, a + 1):
            if ap <= 2 * bp + 1 and math.ceil(ap / 2) - bp <= C - b:
                add_q(b - bp, 2)
                
    add_q(0, -(b + 1))
    for ap in range(1, a + 1):
        if ap <= C and ap % 2 == 1:
            add_q(min((ap - 1) // 2, b - (ap - 1) // 2), -1)
    for ap in range(C + 1, a + 1):
        if ap % 2 == 1 and b + 1 <= C: add_q(b - (ap - 1) // 2, -1)
    for bp in range(1, b + 1):
        for ap in range(1, a + 1):
            if ap < 2 * bp + 1 and ap % 2 == 0 and ap <= C \
               and ap // 2 <= b - bp:
                add_q(ap // 2, -1)
                add_q(ap // 2 - 1, 1)
    for ap in range(C + 1, a + 1):
        if ap % 2 == 0 and C <= b: add_q(C - ap // 2, -1)

    return format_poly(coeffs)
def compute_cl_111(f1, f2, f3, SR):
    F_o, f_o = [SR - f1, SR - f2, SR - f3], [f1, f2, f3]
    for i in range(3):
        others = [j for j in range(3) if j != i]
        if F_o[others[0]] <= f_o[i] and F_o[others[1]] <= f_o[i]:
            f1, f2, f3 = f_o[others[0]], f_o[i], f_o[others[1]]
            break
            
    F1, F2, F3 = SR - f1, SR - f2, SR - f3
    coeffs = {}
    def add_q(power, coef=1): coeffs[power] = coeffs.get(power, 0) + coef
    for bp in range(F2 + 1):
        for ap in range(f2 + 1):
            if ap <= F3 and ap <= 2 * bp: add_q(min(ap // 2, F2 - bp), 2)
            if 2 * bp < ap and F2 < ap and ap <= F3: add_q(min(bp, F2 - bp), 2)
            if 2 * bp < ap and ap <= F3 and ap <= F2: add_q(bp, 4)
            if F3 < ap and ap <= f2 and ap <= 2 * bp \
               and math.ceil(ap / 2) - bp <= F3 - F2:
                add_q(F2 - bp, 2)
            if F3 < ap and 2 * bp < ap and F2 < ap and ap <= f2 \
               and ap - 2 * bp <= F3 - F2:
                add_q(F2 - bp, 2)
            if F3 < ap and 2 * bp < ap and F2 < ap and ap <= f2 \
               and min(ap, ap + F2 - 2 * bp) <= F1:
                add_q(min(bp, F2 - bp), 2)

    add_q(0, -(F2 + 1))
    for bp in range(1, F2 + 1):
        for ap in range(1, f2 + 1):
            if ap <= F3 and ap <= 2 * bp and ap % 2 == 0 \
               and ap // 2 <= F2 - bp:
                add_q(ap // 2, -1)
                add_q(ap // 2 - 1, 1)

    for ap in range(1, f2 + 1):
        if ap <= F3 and ap <= F2 and ap % 2 == 0: add_q(ap // 2 - 1, 1)
        if F2 < ap and ap <= F3: add_q(0, -1)
            
    for bp in range(1, F2 + 1):
        if 2 * bp <= F2:
            for ap in range(1, f2 + 1):
                if F2 < ap and ap <= F3:
                    add_q(bp, -1)
                    add_q(bp - 1, 1)
                    
    for ap in range(1, f2 + 1):
        if ap <= F3 and ap <= F2: add_q(0, -2)
            
    for bp in range(1, F2 + 1):
        for ap in range(1, f2 + 1):
            if 2 * bp < ap and ap <= F3 and ap <= F2:
                add_q(bp, -2)
                add_q(bp - 1, 2)
                
    for ap in range(1, f2 + 1):
        if F3 < ap and ap <= f2 and ap <= 2 * F3 and F3 <= F2 and ap % 2 == 0: 
            add_q(F3 - ap // 2, -1)
        if F3 < ap and ap <= f2 and ap <= 2 * F2 and F2 + 1 <= F1 and ap % 2 == 0: 
            add_q(F2 - ap // 2, 1)
        if F3 < ap and ap <= f2 and F2 < F3 and (ap + F2 - F3) % 2 == 0: 
            add_q((F2 + F3 - ap) // 2, -1)
        if F3 < ap and F2 < ap and ap <= F1: add_q(0, -1)
            
    for bp in range(1, F2 + 1):
        if 2 * bp <= F2:
            for ap in range(1, f2 + 1):
                if F3 < ap and F2 < ap and ap <= F1:
                    add_q(bp, -1)
                    add_q(bp - 1, 1)
        if 2 * bp > F2:
            for ap in range(1, f2 + 1):
                if F3 < ap and 2 * bp < ap and ap <= f2 \
                   and ap + F2 - 2 * bp == F1:
                    add_q(F2 - bp, -1)

    return format_poly(coeffs)
\end{verbatim}
\end{small}
\end{document}